\documentclass[12pt,reqno,a4paper]{amsart}

\usepackage[breaklinks,colorlinks,plainpages,hypertexnames=false,plainpages=false]{hyperref}
\hypersetup{urlcolor=blue, citecolor=blue, linkcolor=blue}

\usepackage[utf8]{inputenc}

\usepackage[doi=true,isbn=false,style=alphabetic,sorting=nyt,backend=biber,maxnames=99,maxalphanames=4,giveninits=true]{biblatex}
\AtEveryBibitem{\clearlist{language}}
\bibliography{refs}
\DeclareSourcemap{ % do not use URL if DOI is present
  \maps[datatype=bibtex]{
    \map{
      \step[fieldsource=doi, final]
      \step[fieldset=url, null]
    }
  }
}

\usepackage{amsmath}
\usepackage{amsthm}
\usepackage{thmtools} % workaround for cref bug https://tex.stackexchange.com/questions/730148/cref-refers-to-lemmas-as-theorems
\usepackage{amssymb}
\usepackage{amsfonts}
\usepackage{enumerate}
\usepackage{mathtools}
\usepackage{tikz}
\usepackage{tikz-cd}
\usepackage{hhline}
\usepackage{stmaryrd}
\usepackage{enumitem}
\usepackage{centernot}
\usepackage{mathrsfs}
\usepackage{comment}
\usepackage{listings}
\usepackage{booktabs}
\usepackage{multicol}
\usepackage{nameref}
\usepackage{hhline}
\usepackage[capitalise, noabbrev]{cleveref} %for \cref
\usepackage{subcaption}

\usepackage[noend]{algorithmic} % for the algorithmic environment

\usepackage{mathdots}
\usepackage{cleveref} % for \cref
\usepackage{nicematrix} % for NiceMatrix
\usepackage{mathalpha} % for various math fonts
\DeclareMathAlphabet{\dutchcal}{U}{dutchcal}{m}{n}

\newcommand{\tree}{%
  \mathord{%
    \vcenter{\hbox{%
        \begin{tikzpicture}[scale=0.1]
          \useasboundingbox (-0.25,-0.5) rectangle (0.75,1);
          \draw (0,0) -- (90:1cm);
          \draw (0,0) -- (210:1cm);
          \draw (0,0) -- (330:1cm);
        \end{tikzpicture}%
      }}%
  }%
}
\newcommand{\prank}{{r_{\tree}}}
\newcommand{\phylogeneticrank}{\prank}
\newcommand{\phylogeneticdensity}{{\rho_{\,\tree}}}

\usepackage{array}
\newcolumntype{L}[1]{>{\raggedright\let\newline\\\arraybackslash\hspace{0pt}}m{#1}}
\newcolumntype{C}[1]{>{\centering\let\newline\\\arraybackslash\hspace{0pt}}m{#1}}
\newcolumntype{R}[1]{>{\raggedleft\let\newline\\\arraybackslash\hspace{0pt}}m{#1}}

\usepackage{bm}

\usepackage{tikz}
\usetikzlibrary{arrows}
\usetikzlibrary{decorations.pathreplacing}
\usetikzlibrary{matrix}
\usetikzlibrary{calc}
\usetikzlibrary{shapes}
\usetikzlibrary{patterns}
\usetikzlibrary{fit,backgrounds,scopes}
\newtheoremstyle{theoremstyle}
{10pt}      %  Space above
{5pt}       %  Space below
{\itshape}  %  Body font
{}          %  Indent amount (empty = no indent, \parindent = para indent)
{\bfseries} %  Thm head font
{}         %  Punctuation after thm head
{ }      %  Space after thm head: " " = normal interword space;
{}          %  Thm head spec (can be left empty, meaning `normal')

\newtheoremstyle{algorithmstyle}
{10pt}      %  Space above
{5pt}       %  Space below
{}  %  Body font
{}          %  Indent amount (empty = no indent, \parindent = para indent)
{\bfseries} %  Thm head font
{}         %  Punctuation after thm head
{ }      %  Space after thm head: " " = normal interword space;
{}          %  Thm head spec (can be left empty, meaning `normal')

\newtheoremstyle{examplestyle}
{10pt}      %  Space above
{5pt}       %  Space below
{}          %  Body font
{}          %  Indent amount (empty = no indent, \parindent = para indent)
{\bfseries} %  Thm head font
{}         %  Punctuation after thm head
{ }      %  Space after thm head: " " = normal interword space;
{}          %  Thm head spec (can be left empty, meaning `normal')

\makeatletter % subalign = substack + align
\newcommand{\subalign}[1]{%
  \vcenter{%
    \Let@ \restore@math@cr \default@tag
    \baselineskip\fontdimen10 \scriptfont\tw@
    \advance\baselineskip\fontdimen12 \scriptfont\tw@
    \lineskip\thr@@\fontdimen8 \scriptfont\thr@@
    \lineskiplimit\lineskip
    \ialign{\hfil$\m@th\scriptstyle##$&$\m@th\scriptstyle{}##$\hfil\crcr
      #1\crcr
    }%
  }%
}
\makeatother

\theoremstyle{theoremstyle}
\newtheorem{theorem}{Theorem}[section]
\newtheorem{lemma}[theorem]{Lemma}
\newtheorem{proposition}[theorem]{Proposition}
\newtheorem{corollary}[theorem]{Corollary}
\newtheorem*{claim}{Claim}

\theoremstyle{examplestyle}
\newtheorem{conjecture}[theorem]{Conjecture}
\newtheorem{example}[theorem]{Example}
\newtheorem{definition}[theorem]{Definition}

\newtheorem{question}[theorem]{Question}
\newtheorem{remark}[theorem]{Remark}

\newtheorem{convention}[theorem]{Convention}

\theoremstyle{algorithmstyle}
\newtheorem{algorithm}[theorem]{Algorithm}

\makeatletter
\newcommand{\customlabel}[2]{%
   \protected@write \@auxout {}{\string \newlabel {#1}{{#2}{\thepage}{#2}{#1}{}} }%
   \hypertarget{#1}{#2}
}
\makeatother

\newcommand{\RR}{\mathbb{R}}

\newcommand{\ZZ}{\mathbb{Z}}

\newcommand{\suchthat}{\;\ifnum\currentgrouptype=16 \middle\fi|\;}
\newcommand{\bigmid}{\left.\vphantom{\Big\{} \suchthat \vphantom{\Big\}}\right.}

\DeclareMathOperator{\domain}{domain}

\begin{document}

\title{The phylogenetic rank of a graph}

\author[Ashworth]{Franklin Ashworth}
\address{Department of Mathematics, Durham University, United Kingdom.}
\email{franklin.ashworth@durham.ac.uk}
\author[Clarke]{Oliver Clarke}
\address{Department of Mathematics, Durham University, United Kingdom.}
\email{oliver.clarke@durham.ac.uk}
\urladdr{https://sites.google.com/view/oclarke-homepage/}
\author[Giansiracusa]{Jeffrey Giansiracusa}
\address{Department of Mathematics, Durham University, United Kingdom.}
\email{jeffrey.giansiracusa@durham.ac.uk}
\urladdr{https://sites.google.com/view/jeffreygiansiracusa}
\author[Jones]{\\ Jackson Jones}
\address{Department of Mathematics, Durham University, United Kingdom.}
\email{jackson.jones@durham.ac.uk}
\author[Quijas-Aceves]{Julio Quijas-Aceves}
\address{Department of Mathematics, Durham University, United Kingdom.}
\email{julio.i.quijas@durham.ac.uk }
\urladdr{https://julioquacdur.github.io/}
\author[Ren]{Yue Ren}
\address{Department of Mathematics, Durham University, United Kingdom.}
\email{yue.ren2@durham.ac.uk}
\urladdr{https://yueren.de}

\subjclass[2020]{05C12, 05C62, 54E35}

\date{\today}

\keywords{phylogenetic rank, graph embedding, graph dimension}

\begin{abstract}
     The Pachter-Sturmfels phylogenetic rank of a graph $G$ is the minimal number of metric trees needed to embed $G$ isometrically.  Here, all edges of $G$ are of length one and the product of metric trees is endowed with the supremum norm. We develop both a greedy and an exact algorithm for computing phylogenetic ranks. Using our algorithms, we construct a database of phylogenetic ranks which includes all graphs on 6 and 7 vertices.  In particular, we exhibit examples disproving that the phylogenetic rank is hereditary, bounded by $\lceil \frac{n}{2} \rceil$, and a generalised 4-point conjecture by Pachter and Sturmfels. In addition, we show that the phylogenetic rank is subadditive under $1$-sums and certain $2$-vertex-sums, and that it is trivially upper bounded by $n-1$, where $n$ denotes the number of vertices. We also provide a complete classification of graphs with phylogenetic rank 1 and construct several infinite families with phylogenetic rank $\lceil \frac{n}{2} \rceil$.  
\end{abstract}

\maketitle

\section{Introduction}

The \emph{phylogenetic rank} of a finite simple connected graph $G$, denoted $\prank(G)$, is the minimum number $k$ necessary to isometrically embed $G$ into a product of metric trees $\Gamma_1\times\dots\times\Gamma_k$.  Here, isometric embedding means mapping the vertices of $G$ into $\Gamma_1\times\dots\times\Gamma_k$ such that distances between vertices are preserved, where edges of $G$ all have length $1$, edges of $\Gamma_i$ may have arbitrary length, and the product $\Gamma_1\times\dots\times\Gamma_k$ is endowed with the supremum of the component metrics.

This question is motivated in two ways:  Our main motivation comes from non-Archimedean optimisation, which is a new framework for optimising data over a product of metric trees \cite{LFMR2026}.  In non-Archimedean optimisation, $\prank(G)$ is the minimal embedding dimension of $G$ regarded as a dataset.  The name \emph{phylogenetic rank} originates from the highly cited work of Speyer and Sturmfels on tropical mathematics \cite[Section ``Phylogenetics'']{SpeyerSturmfels2009}, in the algebraic statistics book by Pachter and Sturmfels it is also referred to as \emph{tree rank} \cite[Section 3.5]{AlgStatsForBio}.  In both works, it is used to describe a mixture of different evolutionary histories.

Isometric embeddings of graphs into graph products is a classical topic in the intersection of combinatorics and metric geometry: \cite{djokovic1973hypercubeSubgraphs} classifies isometric subgraphs of hypercubes $\{0,1\}^k$, and \cite{winkler1984embedKn} extends the theory to products of complete graphs $K_n^k$ endowed with the Hamming metric ($\ell^1$).  The minimal number of factors needed for such an embedding is called the \emph{isometric dimension} \cite{GrahamWinkler1985}.  In contrast, the minimal number of factors needed for an embedding into a product of paths $P_n^k$ endowed with the supremum metric ($\ell^\infty$) is the \emph{strong isometric dimension} \cite{FitzpatrickNowakowski2000}.  For a more comprehensive overview, we refer the reader to \cite{HIK11}.

A natural generalisation is the embedding of arbitrary finite metric spaces into $\RR^n$ endowed with the supremum norm.  A foundational result by Wolfe states that any finite metric space with $n$ points embeds isometrically into $\RR^{n-2}$ \cite{wolfe1967imbedding}.  An asymptotic strengthening of Wolfe's result can be found in \cite{Petrov2010EmbeddingsIntoSupLines}.  Since $\RR$ is a metric tree, Wolfe's result yields the upper bound $\prank(G) \leq n-2$ for any graph $G$ on $n$ vertices.  However, these upper bounds are not necessarily tight in general as a metric tree can encode richer distance information than $\mathbb{R}$. As a simple example, if $G$ is a tree with $m$ leaves then $\prank(G) = 1$ but requires $d = O(\log m)$ dimensions to embed into $\RR^d$ \cite{Linial1995geometry}.

% Of course, embeddings of graphs is generally a well studied topic. \yue{TODO: elaborate on prior work.} \ollie{A couple of very old papers with some connection: \cite{djokovic1973hypercubeSubgraphs} classifies `isometric subgraphs' (see Definition~\ref{def: isometric subgraph}) of the hypercube - i.e. subgraphs whose distance matrix is a submatrix of the distance matrix of the whole graph. \cite{winkler1984embedKn} studies embeddings of graphs into a product of complete graphs. In Winkler's setup, the product of complete graphs has the $1$-norm (and not our sup-norm), so it has more to do with \textit{Hamming distances}.}  \ollie{Here's a paper \cite{Petrov2010EmbeddingsIntoSupLines} about embedding finite metric spaces into $\ell_\infty^n$, which connects to us because any such embedding is also a tree embedding where the tree embeddings are path graphs. However, for them, metric tree spaces cannot be embedded into a one-dimensional space so their results are just upper bounds on the phylogenetic rank. But they reference a paper of D. Wolfe called ` Imbedding a finite metric set in an N-dimensional Minkowski space', which shows that a finite metric space with $n$ points embeds into $\ell_\infty^{n-2}$. (I'm struggling to get a copy of this paper.)}

\subsection*{Outline and main results of the paper}
In \cref{sec:background}, we recapitulate the definition of phylogenetic rank and describe some of its properties including: an upper bound of $n-1$ for $n$-vertex graphs (\cref{prop: phylogenetic rank less than n-1}), monotonicity under isometric subgraphs (\cref{prop: isometric subgraph rank bound}) and graph substitutions (\cref{prop:graphSubstitution}), and subadditivity under $1$-sums (\cref{prop: separable graph rank}) and joining paths to edges (\cref{prop:joiningPath}).

In \cref{sec:examples}, we go over various examples of graphs with known phylogenetic rank, which are: several infinite families of graphs attaining phylogenetic rank $\lceil \frac{n}{2}\rceil$ (\cref{prop: cycle has tree embedding number n/2} and \cref{thm: examples with rank n/2}), complete bipartite graphs which all have phylogenetic rank $2$ (\cref{prop:bipartiteGraph}) and a complete classification of all graphs with phylogenetic rank $1$ (\cref{thm: phylogenetic rank one}).

In Sections \ref{sec:heuristicAlgorithm} and \ref{sec:exactAlgorithm}, we introduce a greedy algorithm (efficient, but doesn't always provide the optimal value) and an exact algorithm (far less efficient) for computing phylogenetic ranks.  The greedy \cref{alg:greedyAlgorithm} grows the product of metric trees iteratively whilst using the distance matrix as a checklist and linear optimisation to prioritise growth candidates.  For the exact \cref{alg:exactAlgorithm} we study so-called \emph{locally antipodal pairs}, or \emph{laps} for short, which are pairs of vertices that are local maxima of the distance function.  We reformulate \cref{alg:exactAlgorithm} as three simple poset traversals (\cref{rem:exactAlgorithm}) and explore criteria that speed up \cref{alg:exactAlgorithm} and are also useful for \cref{sec:examples} (\cref{thm:compatiblePairs} and \cref{cor:compatiblePairs}).

In \cref{sec:computations}, we discuss various computations we have done with \cref{alg:greedyAlgorithm} and \cref{alg:exactAlgorithm}.  In particular, we present a database of graphs and their phylogenetic ranks covering all bi-connected graphs up to $7$ vertices exactly and all graphs up to $8$ vertices approximately, showing in particular that the upper bound $\prank(G) \leq n-2$ is not tight for small graphs.  Moreover, we exhibit examples disproving that the phylogenetic rank: is hereditary, is bounded by $\lceil\frac{n}{2}\rceil$, and satisfies a generalised $4$-point conjecture by Pachter and Sturmfels \cite[Section 3.5]{AlgStatsForBio} \cite[Section ``Phylogenetics'']{SpeyerSturmfels2009}.

\subsection*{Acknowledgements}

Giansiracusa, Clarke, and Ren are supported by the EPSRC grant ``Mathematical Foundations of Intelligence: An 'Erlangen Programme' for AI'' (EP/Y028872/1).  Ren is supported by the UKRI Future Leaders Fellowship ``Computational Tropical Geometry and its Applications'' (MR/S034463/2).

%%% Local Variables:
%%% ispell-local-dictionary: "en_GB"
%%% mode: LaTeX
%%% TeX-master: "graph_embedding_paper"
%%% End:

\section{Phylogenetic rank and its properties}\label{sec:background}

In this section we fix our main conventions for graphs and finite metric spaces. We recall the notion of \textit{phylogenetic rank} (also known as \textit{tree rank} \cite[Section~3.5]{AlgStatsForBio}) and give some elementary properties.

\subsection{Phylogenetic rank}

\begin{convention}
  \label{con:main}
  Throughout the paper, let $G=(V,E)$ denote an undirected, finite, simple, and connected graph with vertex set $V$ and edge set $E\subseteq\binom{V}{2}$.  We consider $G$ as a finite metric space $(V,d_G)$ where
  \begin{equation*}
    d_G(v,v')\coloneqq\min\left\{k\in\ZZ_{\geq 0}\bigmid
      \begin{array}{l}
        \text{there are } v_0,\dots,v_k\in V \text{ with } v_0=v, v_k=v',\\
        \text{and } \{v_{i-1},v_i\}\in E\text{ for all }i=1,\dots,k
      \end{array}
    \right\},
  \end{equation*}   
i.e., each edge has length 1.

We will also consider metric trees, denoted  $\Gamma$ (and occasionally $\Lambda$) where the edge lengths are arbitrary positive real numbers,
i.e., $\Gamma=(V_\Gamma,E_\Gamma,d_\Gamma)$ for some vertex set $V_\Gamma$, edge set $E_\Gamma\subseteq\binom{V_\Gamma}{2}$, and metric $d_\Gamma\colon V_\Gamma\times V_\Gamma\rightarrow\RR_{\geq 0}$.  Unlike $G$, we will consider $\Gamma$ as a one-dimensional metric space, where an edge $\{v,v'\}\in E_\Gamma$ is isometric to a real interval of length $d_\Gamma(v,v')$ \cite[Section 3.2]{BBI2001}.
  This inconsistency is purely for the sake of convenience, so that when given a map $\varphi\colon G\rightarrow \Gamma$ and a graph extension $G\subseteq G'$, we may extend $\varphi$ to $\varphi'\colon G'\rightarrow \Gamma$ by mapping extra vertices of $G'$ to points on the boundary or interior of edges of $\Gamma$.

  Finally, we will often consider product of metric trees $\Gamma_1\times\dots\times\Gamma_k$ which we will generally endow with the supremum norm:
  \begin{equation*}
    d_{\Gamma_1\times\dots\times\Gamma_k}\Big( (x_1,\dots,x_k), (y_1,\dots,y_k)\Big) \coloneqq\sup_{i=1,\dots,k} d_{\Gamma_i}(x_i,y_i).
  \end{equation*}
\end{convention}

\begin{definition}
  Let $G$ be a graph metric space.  A \emph{tree-embedding} is an isometric embedding $\iota=(\iota_1,\dots,\iota_k)\colon G \rightarrow T_1 \times \dots \times T_k$ where each $T_i$ is a metric tree, and we say that $\iota$ is \textit{minimal}, if $(\iota_1,\dots,\hat\iota_\ell,\dots,\iota_k)\colon G \rightarrow \Gamma_1 \times \dots \times\hat\Gamma_\ell\times\dots \times \Gamma_k$ is not an isometric embedding for each $\ell \in [k]$.

  The \emph{phylogenetic rank} $\prank(G)$ is the minimal $k\in\ZZ_{\geq 0}$ such that there exist metric trees $\Gamma_1, \ldots, \Gamma_k$ and an isometric embedding $\iota\colon G \rightarrow \Gamma_1 \times \dots \times \Gamma_k$ \cite[Section~3.5]{AlgStatsForBio}, in which case we say that $\iota$ \textit{attains the rank} or is \textit{globally minimal}.
\end{definition}

% \begin{remark}\label{rmk: minimal tree embedding trivial extensions}
%   A minimal tree embedding $\iota : X \rightarrow T_1 \times \dots \times T_k$ admits a large family of trivial variations obtained by embedding each $T_i$ into a larger tree. We avoid dealing with these cases by tacitly assuming that the leaves of each tree are the image of some point in $X$.
% \end{remark}

We illustrate these definitions with some examples.

\begin{example}
  If $G$ is a tree then the identity map is a tree embedding, so $\prank(G) = 1$.
\end{example}

\begin{example}\label{example: Kn has rank 1}
  If $G=K_n$ is the complete graph on $n$ vertices, pick $\Gamma=K_{1,n}$ to be the star graph whose edges all have length $1/2$ and $\iota\colon K_n \rightarrow K_{1,n}$ to be any map that maps the vertices of $K_n$ to the leaves of $K_{1,n}$. This is an isometric embedding and thus shows that $\prank(K_n)=1$.
\end{example}

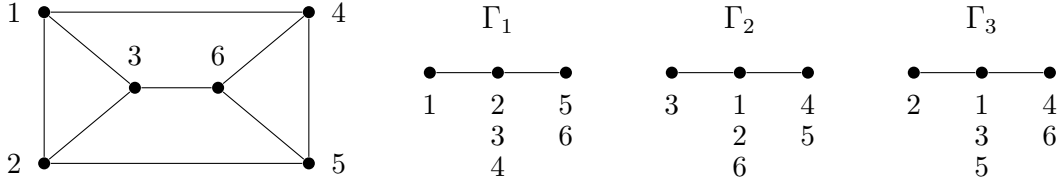
\begin{figure}
  \begin{tikzpicture}[
    vtx/.style={circle, fill, inner sep=1.6pt},
    every label/.style={font=\small\linespread{0.9}\selectfont, align=center, label distance=2pt}
    ]
    %%% left: the 3-prism %%%
    \node[vtx, label={180:1}] (v1) at (1,1.0) {};
    \node[vtx, label={180:2}] (v2) at (1,-1.0) {};
    \node[vtx, label={90:3}]  (v3) at (2.2,0) {};
    \node[vtx, label={0:4}]  (v4) at (4.5,1.0) {};
    \node[vtx, label={0:5}] (v5) at (4.5,-1.0) {};
    \node[vtx, label={90:6}]  (v6) at (3.3,0) {};
    \draw (v1) -- (v2) -- (v3) -- (v1);
    \draw (v4) -- (v5) -- (v6) -- (v4);
    \draw (v1) -- (v4);  \draw (v2) -- (v5);  \draw (v3) -- (v6);

    %%% right: three paths side by side, tags above %%%
    \begin{scope}[shift={(5.2,0.2)}]
      % --- \Gamma_1 ---
      \def\labelsA{1/{1}, 2/{2\\3\\4}, 3/{5\\6}}
      \foreach \i/\lab in \labelsA {
        \node[vtx, label={below:\lab}] (p1-\i) at (\i*0.9, 0) {};
      }
      \draw (p1-1) -- (p1-2) -- (p1-3);
      \node[font=\normalsize] at (1.8, 0.7) {$\Gamma_{1}$};

      % --- \Gamma_2 ---
      \def\labelsB{1/{3}, 2/{1\\2\\6}, 3/{4\\5}}
      \foreach \i/\lab in \labelsB {
        \node[vtx, label={below:\lab}] (p2-\i) at (\i*0.9+3.2, 0) {};
      }
      \draw (p2-1) -- (p2-2) -- (p2-3);
      \node[font=\normalsize] at (5.0, 0.7) {$\Gamma_{2}$};

      % --- \Gamma_3 ---
      \def\labelsC{1/{2}, 2/{1\\3\\5}, 3/{4\\6}}
      \foreach \i/\lab in \labelsC {
        \node[vtx, label={below:\lab}] (p3-\i) at (\i*0.9+6.4, 0) {};
      }
      \draw (p3-1) -- (p3-2) -- (p3-3);
      \node[font=\normalsize] at (8.2, 0.7) {$\Gamma_{3}$};
    \end{scope}
  \end{tikzpicture}\vspace{-3mm}
  \caption{Minimal tree embedding of the three-prism from Example~\ref{example: prism}.}
  \label{fig: prism with embed}
\end{figure}

\begin{example}\label{example: prism}
  Let $G = (V, E)$ be the $3$-prism depicted on the left of Figure~\ref{fig: prism with embed}, i.e., $V=[6]$ and $E= \{12,\, 13,\, 23,\, 14,\, 25,\, 45,\, 36,\, 46,\, 56\}$.
  Then there is a tree-embedding $\iota=(\iota_1,\iota_2,\iota_3)\colon G \rightarrow \Gamma_1 \times \Gamma_2 \times \Gamma_3$ as shown on the right of Figure~\ref{fig: prism with embed}. For instance, the map $\iota$ sends $1\in V$ to the left-most vertex of the first path $\Gamma_1$ and the middle vertex of the second and third paths $\Gamma_2$ and $\Gamma_3$. To see that this is indeed a tree embedding, consider the distance matrices $D_G, D_{\iota_1(G)}, D_{\iota_2(G)}, D_{\iota_3(G)}$ of $G, \iota_1(G), \iota_2(G), \iota_3(G)$ respectively, and notice that for $v_1,v_2\in V$ we have $(D_G)_{v_1,v_2}=\max((D_{\iota_1(G)})_{v_1,v_2},(D_{\iota_2(G)})_{v_1,v_2},(D_{\iota_3(G)})_{v_1,v_2})$:
  \begin{center}
    \begin{tikzpicture}
      \node (DG)
      {%
        $
        \setlength{\arraycolsep}{4pt}
        \begin{pNiceMatrix}[first-row, first-col]
          & 1 & 2 & 3 & 4 & 5 & 6 \\
          1 & 0 & \textcolor{blue!90!black}1 & \textcolor{blue!90!black}1 & \textcolor{blue!90!black}1 & \textcolor{blue!90!black}2 & \textcolor{blue!90!black}2 \\
          2 & \textcolor{blue!90!black}1 & 0 & \textcolor{blue!90!black}1 & \textcolor{blue!90!black}2 & \textcolor{blue!90!black}1 & \textcolor{blue!90!black}2 \\
          3 & \textcolor{blue!90!black}1 & \textcolor{blue!90!black}1 & 0 & \textcolor{blue!90!black}2 & \textcolor{blue!90!black}2 & \textcolor{blue!90!black}1 \\
          4 & \textcolor{blue!90!black}1 & \textcolor{blue!90!black}2 & \textcolor{blue!90!black}2 & 0 & \textcolor{blue!90!black}1 & \textcolor{blue!90!black}1 \\
          5 & \textcolor{blue!90!black}2 & \textcolor{blue!90!black}1 & \textcolor{blue!90!black}2 & \textcolor{blue!90!black}1 & 0 & \textcolor{blue!90!black}1 \\
          6 & \textcolor{blue!90!black}2 & \textcolor{blue!90!black}2 & \textcolor{blue!90!black}1 & \textcolor{blue!90!black}1 & \textcolor{blue!90!black}1 & 0
        \end{pNiceMatrix}
        $
      };
      \node[below] at (DG.south) {$\eqqcolon D_G$};
      \node[anchor=west,xshift=0mm] (DT1) at (DG.east)
      {%
        $
        \setlength{\arraycolsep}{4pt}
        \begin{pNiceMatrix}[first-row,first-col]
          & 1 & 2 & 3 & 4 & 5 & 6 \\
          & 0 & \textcolor{blue!90!black}1 & \textcolor{blue!90!black}1 & \textcolor{blue!90!black}1 & \textcolor{blue!90!black}2 & \textcolor{blue!90!black}2 \\
          & \textcolor{blue!90!black}1 & 0 & 0 & 0 & \textcolor{blue!90!black}1 & 1 \\
          & \textcolor{blue!90!black}1 & 0 & 0 & 0 & 1 & \textcolor{blue!90!black}1 \\
          & \textcolor{blue!90!black}1 & 0 & 0 & 0 & \textcolor{blue!90!black}1 & \textcolor{blue!90!black}1 \\
          & \textcolor{blue!90!black}2 & \textcolor{blue!90!black}1 & 1 & \textcolor{blue!90!black}1 & 0 & 0 \\
          & \textcolor{blue!90!black}2 & 1 & \textcolor{blue!90!black}1 & \textcolor{blue!90!black}1 & 0 & 0
        \end{pNiceMatrix}
        $
      };
      \node[below] at (DT1.south) {$\eqqcolon D_{\iota_1(G)}$};
      \node[anchor=west,xshift=0mm] (DT2) at (DT1.east)
      {%
        $
        \setlength{\arraycolsep}{4pt}
        \begin{pNiceMatrix}[first-row,first-col]
          & 1 & 2 & 3 & 4 & 5 & 6 \\
          & 0 & 0 & \textcolor{blue!90!black}1 & \textcolor{blue!90!black}1 & 1 & 0 \\
          & 0 & 0 & \textcolor{blue!90!black}1 & 1 & \textcolor{blue!90!black}1 & 0 \\
          & \textcolor{blue!90!black}1 & \textcolor{blue!90!black}1 & 0 & \textcolor{blue!90!black}2 & \textcolor{blue!90!black}2 & \textcolor{blue!90!black}1 \\
          & \textcolor{blue!90!black}1 & 1 & \textcolor{blue!90!black}2 & 0 & 0 & \textcolor{blue!90!black}1 \\
          & 1 & \textcolor{blue!90!black}1 & \textcolor{blue!90!black}2 & 0 & 0 & \textcolor{blue!90!black}1 \\
          & 0 & 0 & \textcolor{blue!90!black}1 & \textcolor{blue!90!black}1 & \textcolor{blue!90!black}1 & 0
        \end{pNiceMatrix}
        $
      };
      \node[below] at (DT2.south) {$\eqqcolon D_{\iota_2(G)}$};
      \node[anchor=west,xshift=0mm] (DT3) at (DT2.east)
      {%
        $
        \setlength{\arraycolsep}{4pt}
        \begin{pNiceMatrix}[first-row,first-col]
          & 1 & 2 & 3 & 4 & 5 & 6 \\
          & 0 & \textcolor{blue!90!black}1 & 0 & \textcolor{blue!90!black}1 & 0 & 1 \\
          & \textcolor{blue!90!black}1 & 0 & \textcolor{blue!90!black}1 & \textcolor{blue!90!black}2 & \textcolor{blue!90!black}1 & \textcolor{blue!90!black}2 \\
          & 0 & \textcolor{blue!90!black}1 & 0 & 1 & 0 & \textcolor{blue!90!black}1 \\
          & \textcolor{blue!90!black}1 & \textcolor{blue!90!black}2 & 1 & 0 & \textcolor{blue!90!black}1 & 0 \\
          & 0 & \textcolor{blue!90!black}1 & 0 & \textcolor{blue!90!black}1 & 0 & \textcolor{blue!90!black}1 \\
          & 1 & \textcolor{blue!90!black}2 & \textcolor{blue!90!black}1 & 0 & \textcolor{blue!90!black}1 & 0
        \end{pNiceMatrix}
        $
      };
      \node[below] at (DT3.south) {$\eqqcolon D_{\iota_3(G)}$};
    \end{tikzpicture}
  \end{center}
  The explicit embedding shows $\phylogeneticrank(G)\leq 3$.  The fact that $\phylogeneticrank(G)=3$ follows from \cref{thm: examples with rank n/2} since $G=K_6 \setminus C_6$.
\end{example}

We will use the following construction repeatedly throughout the paper to construct tree embeddings:

\begin{definition}
  Let $G$ be a graph and $v \in V$ a vertex.  The \emph{path map at $v$} is the 1-Lipschitz map $p_v\colon G \to \RR$ defined by sending $v$ to the origin and any vertex $v'$ to its distance from $v$.  The image of $p_v$ is the interval that is the convex hull of the set of distance, $D_v \coloneqq \{d_G(v,v') \mid v' \in V \}$.  We regard this interval as a metric tree $\Gamma_v$ with vertex set $D_v$.
\end{definition}

% \begin{proposition}\label{prop: path map is 1-Lipschitz}
%   The path map $p_v\colon G \rightarrow \Gamma_v$ is $1$-Lipschitz and $\Gamma_v$ is a tree metric.
% \end{proposition}

% \begin{proof}
%   Fix $a, b \in V$. By the reverse triangle inequality we have
%   \[
%     d(p_v(a),p_v(b)) = |d_V(a, v) - d_V(v, b)| \le d_G(a, b),
%   \]
%   so $p_v$ is $1$-Lipschitz. To see that $D_v$ is a tree metric, observe that it is a subspace of $\RR$, which is a tree. Alternatively, for any four points $a \le b \le c \le d \in D_v$ we have
%   \[
%     d(a,b) + d(c,d) \le d(a,c) + d(b,d) = d(a,d) + d(b,c).
%   \]
%   So $D_v$ satisfies the four point condition, so it is a tree metric.
% \end{proof}

As a straightforward application of path maps, we obtain the following upper bound of the phylogenetic rank:

\begin{proposition}\label{prop: phylogenetic rank less than n-1}
  Let $G$ be a graph, then $\prank(G) \le |V|-1$.
\end{proposition}

\begin{proof}
  % For each $x \in X$, we construct a metric tree $T_x$ homeomorphic to a line segment. The vertex set of $T_x$ is $V(T_x) := \{d_X(x, y) : y \in X \} \subset \RR_{\ge 0}$ and we directly define the metric $d_{T_x}(a, b) = |a - b|$. It is straightforward to check that this is indeed the metric graph whose underlying graph is a path and whose edges are the pairs $\{a,b\} \in \binom{V(T_x)}{2}$ such that there exists no $c \in V(T_v)$ that lies strictly between $a$ and $b$.
  Let $V'$ be the set $V$ with one point removed. Then the product of path maps $\iota \coloneqq (p_v)_{v \in V'}\colon V \rightarrow \prod_{v \in V'} \Gamma_v$ is a tree embedding:  Each $p_v$ is $1$-Lipschitz, and hence $\iota$ is $1$-Lipschitz.  And for any $a, b \in V$, say $a \in V'$, we have $d_{\Gamma_a}(p_a(a),p_a(b)) = |d_G(a,a) - d_G(a,b)| = d_G(a,b)$.  Combining both, we obtain $d(\iota(a),\iota(b)) = d_G(a,b)$ and thus $\iota$ is an isometry.
\end{proof}

We close the section with some references to literature:

\begin{remark}\
  \begin{enumerate}
    \item In \cite{AlgStatsForBio}, the space of metrics on $n$ points with rank at most $r$ is denoted $\mathcal T_n^r$. Each such metric lies in a product $\Gamma_1\times\dots\times\Gamma_k$ equipped with the supremum norm, which is referred to as a \textit{mixture} of tree metrics. Pachter and Sturmfels highlight that understanding the structure of the space $\mathcal T_n^r$ could be used in comparative genomics to obtain a consensus among conflicting tree metrics. 
    
    \item 
    % \ollie{Check with Julio's notes that this is the Fr\'echet embedding and find a reference}
    Proposition~\ref{prop: phylogenetic rank less than n-1} has been improved and generalised as follows. Wolfe \cite{wolfe1967imbedding} showed that $G$ embeds into $\RR^{|V|-2}$ with the supremum norm. And in general, it has been shown in \cite{Petrov2010EmbeddingsIntoSupLines} that, for any fixed $c \in \ZZ_{> 0}$, there exists $n_0$ such that any graph $G$ with $n \geq n_0$ vertices has phylogenetic rank $\prank(G) \leq n-c$.
    
    \item Proposition~\ref{prop: phylogenetic rank less than n-1} can also be proven using tropical algebraic geometry by exploiting the connection between metric trees and tropical Grassmannians \cite{SpeyerSturmfels2004}:  By Terracini's Lemma \cite[Section~5.3]{landsberg2011tensors}, the $(n-1)$-th secant variety of Grassmannian ${\rm Gr}(2,n)$ under the Pl\"ucker embedding fills the ambient space. Tropicalising this tells us that the $(n-1)$-fold product of tree metrics can realise any metric (including metrics arising from graphs).
  \end{enumerate}
\end{remark}

\subsection{Monotonicity of phylogenetic rank}
In this section we study the phylogenetic rank under certain graph operations.

%
%%% Hereditary property is not correct %%%
%

% \begin{proposition}
%   Let $H$ be a connected induced subgraph of $G$, then $\prank(H) \le \prank(G)$.
% \end{proposition}

% \begin{proof}
%   \textcolor{red}{The proof in the previous file is not correct because the metric of an induced subgraph can be completely different to that of the original graph! (But, I (Ollie) haven't seen any examples where the rank increases by taking induced subgraphs (yet))}
% \end{proof}

% \begin{proposition}
%   Fix a positive integer $r$. The set of graphs $G$ with $\prank(G) \le r$ form a \textit{connected-hereditary class}, i.e., a set of graphs closed under taking induced subgraphs that are connected. Moreover, this hereditary class has among its minimal forbidden induced subgraphs the cycles $C_i$ for $i \ge \max\{4, 2r+1\}$.
% \end{proposition}

% \begin{proof}
%   \textcolor{red}{Requires us to know cycles have rank $n/2$}
% \end{proof}

\begin{definition}\label{def: isometric subgraph}
  Let $G$ be a graph and $H$ a subgraph. We say that $H$ is an \textit{isometric subgraph} if the inclusion $H \to G$ is an isometry.  Note that any isometric subgraph is an induced subgraph.
\end{definition}

\begin{proposition}\label{prop: isometric subgraph rank bound}
  Let $H$ be an isometric subgraph of $G$, then $\prank(H) \le \prank(G)$.
\end{proposition}

\begin{proof}
  The restriction of a tree embedding of $G$ to $H$ is a tree embedding.
\end{proof}

% \subsection{Separable graphs and clique substitutions}
For $W \subset V$, $G[W]$ denotes the induced subgraph on vertex set $W$. Recall that a graph $G = (V, E)$ is \textit{separable} if there exists a vertex $v \in V$ such that the induced subgraph $G[V \setminus \{v\}]$ is disconnected. In this case, we call $v$ a \textit{cut vertex} of $G$.  The following proposition shows that, for computing phylogenetic ranks of graphs, it suffices to consider non-separable graphs.

\begin{proposition}\label{prop: separable graph rank}
  Let $G$ be a separable graph with cut vertex $v$ and $C_1, \dots, C_k$ the connected components of $G[V \setminus \{v\}]$. For each $i \in [k]$, define $D_i = G[V(C_i) \cup \{v\}]$, then
  $
  \prank(G) = \max\{\prank(D_i) \mid i \in [k]\}
  $.
\end{proposition}

\begin{proof}
  Observe that $D_i$ is an isometric subgraph of $G$ for each $i \in [k]$. Explicitly, for any $x, y \in V(D_i)$ distinct, a shortest path from $x$ to $y$ in $H$ is also a shortest path in $G$. So by Proposition~\ref{prop: isometric subgraph rank bound}, we have
  \[
    \prank(G) \ge \max\{\prank({D_1}), \prank({D_2}), \dots, \prank({D_k})\}.
  \]
  To show equality, we give a construction. We begin by recalling the notion of a wedge-sum of spaces.

  Let $X_1, \dots, X_k$ be metric spaces and fix some $x_i \in X_i$ for each $i \in [k]$. The \textit{wedge sum} of $X_1, \dots, X_k$ at the points $x_1, \dots, x_k$ is defined to be the disjoint union $\bigsqcup_i X_i / \sim$ with points $x_1 \sim x_2 \sim \dots \sim x_k$ identified. The metric on the wedge product $X$ is obtained by extending the metrics on each of the summands by $d_X(u,v) \coloneqq d_{X_i}(u, x_i) + d_{X_j}(v, x_j)$ for any $u \in X_i$ and $v \in X_j$ with $i \neq j$.
    
  Without loss of generality, let us assume
  $\prank({D_1}) \ge \prank({D_2}) \ge \dots \ge \prank({D_k})$ and for each $i$ fix an embedding
  \[
    \iota_i\colon D_i \rightarrow \Gamma^i_1 \times \Gamma^i_2 \times \dots \times \Gamma^i_{\prank({D_i})}.
  \]
  We define $\Gamma^i_j \coloneqq \Gamma^i_{\prank({D_i})}$ for all $j \in \{\prank_{D_i}+1, \prank({D_i})+2 \dots, \prank({D_1})\}$.
  Now let us define
  \[
    \iota\colon G \rightarrow \Gamma_1 \times \Gamma_2 \times \dots \times \Gamma_{\prank({D_1})}
  \]
  where the tree $\Gamma_j$ is the wedge sum of the trees $\Gamma^i_j$ at the point $(\iota_i(v))_j \in \Gamma^i_j$, and the $j$-th component of the map $\iota$ sends a vertex $x \in V(D_i) \subseteq V(G)$ to $\iota_i(x) \in \Gamma^i_j \subseteq \Gamma_i$.
  Note that $\Gamma_j$ is a tree because it is a wedge of trees. Moreover, the map $\iota$ is a $1$-Lipschitz because it is $1$-Lipshitz on each subgraph $D_i$, which cover the graph. Let $x, y \in V(G)$ and consider a shortest path $P$ from $x$ to $y$. If $x \in V(D_i)$ and $y \in (D_j)$ with $i \neq j$, then $P$ must pass through the vertex $v$. So
  \[
    d_G(x,y) = d_G(x,v) + d_G(y,v) = d_{D_i}(x,v) + d_{D_j}(y,v).
  \]
  So, by the construction of $\Gamma_1 \times \dots \times \Gamma_{r_{D_1}}$, the length of the path $P$ is preserved. So we have constructed a rank $r_{D_1}$ tree embedding of $G$. This concludes the proof.
\end{proof}

\begin{definition}
  Let $G$ and $H$ be graphs on disjoint vertex sets and $v \in V(G)$ a vertex. The \textit{graph substitution $G[v \rightarrow H]$} is the graph obtained by removing $v$ from $G$ and replacing it with a copy of $H$ such that each vertex of $H$ is connected to the neighbourhood of $v$.
\end{definition}

\begin{proposition}\label{prop:graphSubstitution}
  Let $G$ and $H$ be graphs and $v \in V(G)$ a vertex. Then
  \[
    \prank(G) \le \prank(G[v \rightarrow H]).
  \]
  Moreover, if the diameter of $H$ is at most two, then $\prank(G[v \rightarrow H]) \le \prank(G) + \prank(H)$.
\end{proposition}

\begin{proof}
  Firstly, we show $\prank(G) \le \prank(G[v \rightarrow H])$ by proving that $G$ is a isometric subgraph of $G[v \rightarrow H]$. To do this, identify the distinguished vertex $v$ with an arbitrary vertex $v \in V(H)$, so that $G$ is an induced subgraph of $G[v \rightarrow H]$.

  Suppose that $x, y \in V(G) \subseteq V(G[v \rightarrow H])$ and let $P$ be a shortest path from $x$ to $y$. Assume by contradiction that $P$ contains at least two distinct vertices of $H$. Let $\alpha, \omega \in V(H)$ be, respectively, the first and last occurrence of a vertex of $H$ in $P$. Thus the vertices in $P$ can be written as
  \[
    V(P) = (x \eqqcolon w_0,\, w_1,\, \dots,\, w_t,\, \alpha,\, \dots,\, \omega,\, w_{t+1},\, \dots,\, w_r \coloneqq y).
  \]
  By assumption $w_t, w_{t+1} \in V(G)$ and each are adjacent to a vertex of $H$. So, by the construction of $G[v \rightarrow H]$, both $w_t$ and $w_{t+1}$ are adjacent to $v$. So, we may construct a strictly shorter path from $x$ to $y$ with vertices
  $(w_0,\, w_1,\, \dots,\, w_t,\, v,\, w_{t+1},\, \dots,\, w_r)$,
  which contradicts the minimality of $P$. So we have proved that $P$ contains at most one vertex from $H$. Moreover, the same argument shows that, if $P$ contains a vertex from $H$, then it may be assumed to be $v$. Thus, there exists a shortest path in $G$ with the same length as $P$. Hence $G$ is an isometric subgraph of $G[v \rightarrow H]$. So by Proposition~\ref{prop: isometric subgraph rank bound}, we have $\prank(G) \le \prank(G[v \rightarrow H])$.

  Assume $H$ has diameter at most two. We prove that $\prank(G[v \rightarrow H]) \le \prank(G) + \prank(H)$ by constructing a tree embedding. Let $\iota\colon G \rightarrow \Gamma_1 \times \dots \times \Gamma_{\prank(G)}$ be a tree embedding that attains the rank of $G$. We extend $\iota$ to a map $\hat \iota$ on $G[v \rightarrow H]$ by
  \[
    \hat \iota\colon G[v \rightarrow H] \rightarrow \Gamma_1 \times \dots \times \Gamma_{\prank(G)},
    \quad
    \hat \iota(w) = \begin{cases}
      \iota(w) & \text{if } w \in V(G), \\
      \iota(v) & \text{if } w \in V(H),\, w \neq v.
    \end{cases}
  \]
  By the construction of $G[v \rightarrow H]$, it follows that $\hat \iota$ is $1$-Lipschitz and since the restriction of $\hat \iota$ to $G$ is identical to $\iota$, the only distances that are not preserved by $\hat \iota$ are those between vertices of $H$.

  Let $\kappa\colon H \rightarrow \Lambda_{1} \times \dots \times \Lambda_{\prank(H)}$ be a tree embedding that attains the rank of $H$. Without loss of generality, we will assume that minimal tree embeddings are pruned to have the smallest possible trees. Since $H$ has diameter at most two, it follows that each tree has diameter at most two. So, for each $i \in [\prank(H)]$, there is a point $p_i$ in the tree $\Lambda_i$ that has distance at most $1$ to all leaves of $\Lambda_i$. Define $p = (p_1, \dots, p_{\prank(H)})$. We extend $\kappa$ to a map $\hat \kappa$ on $G[v \rightarrow H]$ by
  \[
    \hat \kappa\colon G[v \rightarrow H] \rightarrow  \Lambda_{1} \times \dots \times \Lambda_{\prank(H)},
    \quad
    \hat \kappa(w) = \begin{cases}
      \kappa(w) & \text{if } w \in V(H),\\
      p & \text{if } w \in V(G),\, w \neq v.
    \end{cases}
  \]
  For each $w \in V(H)$ and $w' \in V(G)$, we have $d(\hat \kappa(w), \hat \kappa(w')) \le d(\kappa(w), p) \le 1$. Each vertex of $G$, except possibly $v$, is mapped to the same point in the product of trees. So, we have that $\hat \kappa$ is a $1$-Lipschitz extension of $\kappa$. Thus
  \[
    \hat \iota \times \hat \kappa\colon G[v \rightarrow H] \rightarrow \Gamma_1 \times \dots \times \Gamma_{\prank(G)} \times \Lambda_1 \times \dots \times \Lambda_{\prank(H)}, \quad
    w \mapsto (\hat \iota(w), \hat \kappa(w))
  \]
  is a tree embedding of $G[v \rightarrow H]$. So $\prank(G[v \rightarrow H]) \le \prank(G) + \prank(H)$.
\end{proof}

The final extension we consider is a special case of a \textit{$2$-vertex-sum}. Given two graphs $G$ and $H$ on disjoint vertex sets, a $2$-vertex-sum is a graph obtained by identifying two distinct vertices of $G$ with two in $H$.

\begin{proposition}\label{prop:joiningPath}
  Let $G$ be a graph and $vw \in E(G)$ an edge. Fix an integer $n \ge 3$ and let $H$ be the graph obtained from $G$ by adding $n-2$ new vertices 
  joining a path $(v \eqqcolon v_1,\, \dots,\, v_n \coloneqq w)$ of length $n-1$ between $v$ and $w$, creating an induced $n$-cycle $C$. We have that
  \[
    \max\{\prank(G), \prank(C)\}\le \prank(H) \leq \prank(G)+\prank(C).
  \]
\end{proposition}

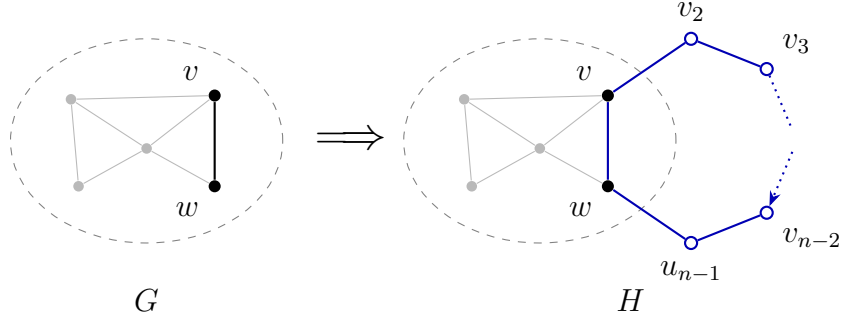
\begin{figure}
\begin{center}
\begin{tikzpicture}[
  >={Stealth[length=2mm]},
  vtx/.style   = {circle, fill=black, inner sep=1.6pt},
  oldv/.style  = {circle, fill=gray!55, inner sep=1.4pt},
  newv/.style  = {circle, draw=blue!70!black, fill=white, thick, inner sep=1.6pt},
  oldE/.style  = {gray!55},
  cyc/.style   = {-, thick, blue!70!black},
  base/.pic = {
    \draw[dashed, gray] (0,0) ellipse (1.8 and 1.35);
    \node[oldv] (-a) at (-1.0, 0.55) {};
    \node[oldv] (-b) at (-0.9,-0.60) {};
    \node[oldv] (-c) at ( 0.0,-0.10) {};
    \node[vtx, label={above left:$v$}] (-v) at (0.9, 0.6) {};
    \node[vtx, label={below left:$w$}] (-w) at (0.9,-0.6) {};
    \draw[oldE] (-a) -- (-b) -- (-c) -- (-a);
    \draw[oldE] (-c) -- (-v);
    \draw[oldE] (-c) -- (-w);
    \draw[oldE] (-a) -- (-v);
  }
]

% ---------- Left panel: original graph G with edge w -> v ----------
\pic (L) at (0,0) {base};
\draw[-, thick] (L-w) -- (L-v);
\node at (0,-2.1) {$G$};

% ---------- Transition arrow ----------
\node at (2.7,0) {\Large$\Longrightarrow$};

% ---------- Right panel: G with the new path attached ----------
\pic (R) at (5.2,0) {base};

\begin{scope}[shift={(5.2,0)}]
  \node[newv, label={above:$v_2$}]        (u1)  at (2.0, 1.35) {};
  \node[newv, label={above right:$v_3$}]  (u2)  at (3.0, 0.95) {};
  \node                                   (dd)  at (3.4, 0.0)  {};
  \node[newv, label={below right:$v_{n-2}$}] (u3) at (3.0,-0.95) {};
  \node[newv, label={below:$u_{n-1}$}]    (u4)  at (2.0,-1.35) {};
\end{scope}

\draw[cyc] (R-v) -- (u1);
\draw[cyc] (u1)  -- (u2);
\draw[thick, blue!70!black, dotted] (u2) -- (dd);
\draw[thick, blue!70!black, dotted, ->] (dd) -- (u3);
\draw[cyc] (u3)  -- (u4);
\draw[cyc] (u4)  -- (R-w);
\draw[cyc] (R-w) -- (R-v);

\node at (6.4,-2.1) {$H$};

\end{tikzpicture}
\end{center}
\caption{Adding a cycle to $G$ as in Proposition \ref{prop:joiningPath}}
\label{fig: add a cycle to G}
\end{figure}

\begin{proof}
  The first inequality follows from the observation that $G$ and $C$ are isometric subgraphs of $H$. It is straightforward to see that any shortest path between two vertices of $G$ must be completely contained in $G$ and similarly for any shortest path between two vertices of $C$. So, by Proposition~\ref{prop: isometric subgraph rank bound}, we have that $\prank(G) \le \prank(H)$ and $\prank(C) \le \prank(H)$.

  For the other inequality, let
  \[
    \iota \coloneqq \iota_1 \times\dots \times\iota_\ell\colon G\rightarrow \Gamma_1 \times \dots \times \Gamma_\ell
    \quad \text{and} \quad
    \kappa \coloneqq \kappa_1 \times \dots \times \kappa_m\colon C \rightarrow \Lambda_1 \times \dots \times \Lambda_m
  \]
  be tree embeddings that achieves the ranks of $G$ and $C$ respectively.

  We proceed by constructing a $1$-Lipschitz extension $\hat \iota_i \colon H \rightarrow \Gamma_i$ of $\iota_i$ for each $i \in [\ell]$ and a $1$-Lipschitz extension $\hat \kappa_j \colon C \rightarrow \Lambda'_j$ of $\kappa_j$ for each $j \in [m]$. Note that we will modify the tree $\Lambda_j$ to form $\Lambda'_j$ but keep the tree $\Gamma_j$ unchanged.

  Fix $i \in [\ell]$ and define the extension $\hat \iota_i$ of $\iota_i$ as follows
  \[
    \hat \iota_i\colon H \rightarrow \Gamma_i, \quad
    \hat \iota_i(x) = \begin{cases}
      \iota_i(x) & \text{if } x \in V(G),\\
      \iota_i(v) & \text{if } x \in V(C),\, x \notin \{v, w\}.
    \end{cases}
  \]
  In words, the map $\hat \iota_i$ maps vertices of $G$ by $\iota_i$ and crushes everything else to the image of the vertex $v$. Since $\kappa_j$ is $1$-Lipschitz, it follows that $\hat \kappa_j$ is also $1$-Lipschitz.

  % \color{red}
  % We'll need the following things:
  % \begin{itemize}
  %   \item If $n = 1$, then this is a special case where $m = 1$ and we replace $\kappa_1$ with the \textit{path map} from the new vertex in $C$ (Define \textit{path map} earlier in the section and use it in proof of Proposition~\ref{prop: phylogenetic rank less than n-1})

  %   \item If $n \ge 2$, then there is a path with $\lfloor n/2 \rfloor$ vertices, call them $U$, that includes the two or three vertices that are furthest away from $v$ and $w$ in $C$.

  %   \item Assume that $\kappa$ is given by the product of path maps starting at vertices of $U$.

  %   \item Extend each of these path maps to $H$ in the natural way to obtain $\hat \kappa_j$. And we are ready to take cases on all the minimal paths in $H$.

  %   \begin{itemize}
  %     \item Every shortest path in $G$ is realised by some extension $\hat \iota_i$

  %     \item Every shortest path in $C$ is realised by some extension $\hat \kappa_j$

  %     \item Every shortest path $P$ from a vertex in $C$ to a vertex outside of $C$ can be decomposed into two shortest paths $P = P_C \cup P_G$ that leaves $C$ though $v$ or $w$. WLOG, assume $P$ \textit{`leaves through'} $v$. Then, this path's distance is attained by (the unique map if $n$ is even, or one of the two maps if $n$ is odd) $\hat \kappa_j$ that is the path map from a furthest point of $v$ in $C$.
  %   \end{itemize}

  % \end{itemize}

  % \color{black}

  Let us now extend the maps $\kappa_j$ to $H$. We first deal with the special case $n = 3$. Since $C$ is the complete graph $K_3$, by Example~\ref{example: Kn has rank 1}, we have $m = 1$. We define $\hat \kappa \colon H \rightarrow \Lambda'$ to be the path map based at the vertex $u \in V(C)$ with $u \notin \{v, w\}$, and show that the product map
  \[
    \hat \iota_1 \times \dots \times \hat \iota_\ell \times \hat \kappa\colon H \rightarrow \Gamma_1 \times \dots \times \Gamma_\ell \times \Lambda'
  \]
  is a tree embedding. Let $x, y \in V(H)$ be any vertices. If $x,y \in V(G)$, then, since $G$ is an isometric subgraph of $H$, there is an index $i \in [\ell]$ such that $d_H(x,y) = d_{\Gamma_i}(\iota_i(x), \iota_i(y)) = d_{\Gamma_i}(\hat \iota_i(x), \hat \iota_i(y))$. On the other hand, if at least one of $x, y$ does not lie in $V(G)$, then we may assume without loss of generality that $x = u$. By the construction of the path map $\hat \kappa$, it immediately follows that $d_H(x,y) = d_{\Lambda'}(\hat \kappa(x), \hat \kappa(y))$. So, in each case we showed that the distance between $x,y$ is achieved by some coordinate of the product. Since the extensions $\hat \iota_i$ and the map $\hat \kappa$ are $1$-Lipschitz, we have the map is a tree embedding. So $\prank(H) \le \prank(G) + 1 = \prank(G) + \prank(C)$.

  From now on, let us assume $n \ge 4$. For each $j \in [m]$, we construct an extension of $\kappa_j$ to $H$. Recall that the vertices of $C$ are labelled cyclically $v = v_1, v_2, \dots, v_n = w$. Define $k = \lfloor n/2 \rfloor$. By Proposition~\ref{prop: cycle has tree embedding number n/2}, we may assume that $\kappa_j$ is the path map based at vertex $v_{j+1}$. In particular, all path maps based vertices with distance $k$ from $v$ or $w$ are included among the coordinates of $\kappa$.

  We define $\hat \kappa_j \colon H \rightarrow \Lambda'_j$ to be the path map based at $v_{j+1}$ in $H$. Since $C$ is an isometric subgraph of $H$, it follows that the restriction of $\hat \kappa_j$ to $C$ coincides with $\kappa_j$. Since $\hat \kappa_j$ is a path map, we have that $\Lambda'_j$ is a tree metric and $\hat \kappa_j$ is $1$-Lipschitz.

  We will now prove that the product of maps
  \[
    \hat \iota_1 \times \dots \times \hat \iota_\ell \times \hat \kappa_1 \times \dots \times \hat \kappa_m\colon H \rightarrow \Gamma_1 \times \dots \times \Gamma_\ell \times \Lambda'_1 \times \dots \times \Lambda'_m
  \]
  is a tree embedding. We proceed by fixing $x,y \in V(H)$ and taking cases.

  \smallskip
  \noindent \textbf{Case 1.} Assume $x, y \in V(G)$. Since $\iota$ is a tree embedding of $G$, there exists an index $i \in [\ell]$ such that $d_G(x,y) = d_{\Gamma_i}(\iota_i(x), \iota_i(y))$. Since $G$ is a isometric subgraph of $H$ and the restriction of $\hat \iota_i$ to $G$ is equal to $\iota_i$, it follows that $d_H(x,y) = d_{\Gamma_i}(\hat \iota_i(x), \hat \iota_i(y))$.

  \smallskip
  \noindent \textbf{Case 2.} Assume $x,y \in V(C)$. Since $\kappa$ is a tree embedding of $C$, there exists $j \in [m]$ such that $d_G(x,y) = d_{\Lambda_j}(\kappa_j(x), \kappa_j(y))$. Since $C$ is a isometric subgraph of $H$ and the restriction of $\hat \kappa_j$ to $C$ is equal to $\kappa_j$, it follows that $d_H(x,y) = d_{\Lambda'_j}(\hat \kappa_j(x), \hat \kappa_j(y))$.

  \smallskip
  \noindent \textbf{Case 3.} Assume $x \in V(G)$ and $y \notin V(G)$. Consider a shortest path $P$ from $x$ to $y$. Since $V(G) \cap V(C) = \{v, w\} \subseteq V(H)$, the path $P$ must pass through at least one of $v$ or $w$. We take further cases based on whether $P$ passes through $v$, $w$, or both.

  \smallskip
  \noindent \textbf{Case 3(i).}
  Assume that $P$ passes through $v$ and not $w$. So we have
  \[
    P \colon x \rightarrow \dots \rightarrow v \rightarrow \dots \rightarrow y.
  \]
  We extend $P$ to a path $P'$ from $x$ to $v_{k+1}$ as follows
  \[
    P' \colon x \rightarrow \dots \rightarrow y = v_s \rightarrow v_{s+1} \rightarrow \dots \rightarrow v_{k+1}.
  \]
  Since $k \le n/2$, the subpath of $P'$ given by $v \rightarrow \dots \rightarrow v_{k+1}$ is a minimum length path. Since $w$ does not appear in $P$, it follows that $P'$ is a minimum length path. Note that, by Proposition~\ref{prop: cycle has tree embedding number n/2}, we have $m = k$. Since $\hat \kappa_k$ is the path map based at $v_{k+1}$, it follows that $d_H(x,v_{k+1}) = d_{\Lambda'_k}(\hat \kappa_k(x), \hat \kappa_k(v_{k+1}))$. Since $P$ is a subpath of $P'$, and $P'$ is mapped isometrically into $\Lambda'_k$ by $\hat \kappa_k$, it follows that $d_H(x,y) = d_{\Lambda'_k}(\hat \kappa_k(x), \hat \kappa_k(y))$.

  \smallskip
  \noindent \textbf{Case 3(ii).}
  Suppose that $P$ passes through $v$ after $w$, i.e., we have:
  \[
    P \colon x \rightarrow \dots \rightarrow w \rightarrow v \rightarrow \dots \rightarrow y.
  \]
  Since $P$ is a minimum length path, we must have that $y$ is closer to $v$ than $w$, thus $y \neq v_{k+1}$.
  We extend $P$ to a path $P'$ from $x$ to $v_k$ as follows
  \[
    P' \colon x \rightarrow \dots \rightarrow y = v_s \rightarrow v_{s+1} \rightarrow \dots \rightarrow v_k.
  \]
  Since $k \le n/2$, the subpath of $P'$ given by $w \rightarrow v \rightarrow \dots \rightarrow v_k$ is a minimum length path. Note that, by Proposition~\ref{prop: cycle has tree embedding number n/2}, we have $m = k$. Since $\hat \kappa_{k-1}$ is the path map based at $v_k$, it follows that $d_H(x,v_k) = d_{\Lambda'_{k-1}}(\hat \kappa_{k-1}(x), \hat \kappa_{k-1}(v_k))$. Since $P$ is a subpath of $P'$ and $P$ is mapped isometrically into $\Lambda'_{k-1}$ by $\hat \kappa_k$, it follows that $d_H(x,y) = d_{\Lambda'_{k-1}}(\hat \kappa_{k-1}(x), \hat \kappa_{k-1}(y))$.

  \smallskip
  \noindent \textbf{Case 3(iii).} Assume that $P$ basses through $w$ and not $V$ or through $w$ after $v$. Then, by same argument from my case 3(i) or 3(ii) with $v$ and $w$ swapped and the path $P$ going around the other side of the cycle $C$, we deduce that $d_H(x,y) = d_{\Lambda'_j}(\hat \kappa_j(x), \hat \kappa_j(y))$ for some $j \in [k]$.

  \smallskip

  In each case, we have shown that the distance between $x,y$ is attained by some coordinate of the product. So the map is a tree embedding. Hence $\prank(H) \le \prank(G) + \prank(C)$. This concludes the proof. \qedhere

\end{proof}

\section{Graphs with known phylogenetic ranks}\label{sec:examples}

In this section, we study the ranks of special families of graphs and characterise the graphs with rank one. In Proposition~\ref{prop: cycle has tree embedding number n/2} and Theorem~\ref{thm: examples with rank n/2}, we identify families of graphs whose rank is half the number of vertices including the $n$-cycle and its complement. In Theorem~\ref{thm: phylogenetic rank one}, we classify the rank-one graphs with an forbidden induced subgraph classification.

We recall the \textit{four point condition}, which is classifies tree metrics.

\begin{proposition}[{\cite[Theorem~2]{buneman1974note}, \cite[Theorem~2.35]{AlgStatsForBio}}]\label{prop: four point condition}
A finite metric space $X$ is a tree metric if and only if it satisfies the \textbf{four point condition}: for any distinct $s,t,u,v \in X$ the maximum value among:
\[
d(s,t) + d(u,v),\quad
d(s,u) + d(t,v),\quad
d(s,v) + d(u,t)
\]
is attained at least twice. 
% \yue{I'm putting it here, since I don't think it is needed in Section 2.}
\end{proposition}

We will use the four point condition throughout for showing a lower bound on the rank on graphs.

\begin{proposition}\label{prop: cycle has tree embedding number n/2}
    Fix $n \ge 4$. The $n$-cycle $C_n$ has phylogenetic rank $k \coloneqq \lceil n/2 \rceil$. Moreover, let $x_1, \dots, x_k$ be a sequence of consecutive vertices in $C_n$ and $p_i \colon C_n \rightarrow \Gamma_i$ be the path map based at $x_i$ for each $i \in [k]$. Then a minimal tree embedding is given by the product $p_1 \times \dots \times p_k \colon C_n \rightarrow \Gamma_1 \times \dots \times \Gamma_k$.
\end{proposition}

% \begin{proof}
%     \textcolor{blue}{The proof of this is a lot easier if we allow ourselves to use the $T$-shaped incompatibility result.}

%     \textcolor{red}{Add an example to section~5 and reference that.}
    
%     \textcolor{blue}{Actually, this is only in the case of $n$ even. For $n$ odd we need something stronger (copy previous proof).}
    
% \end{proof}

\begin{proof}
Let $V(C_n) = [n]$ and assume that the edges are $\{1,2\}$, $\{2,3\}$, \dots, $\{n-1,n\}$,$\{1,n\}$.
First, we show that the map
\[
p \coloneqq p_1 \times \dots \times p_k \colon C_n \rightarrow \Gamma_1 \times \dots \times \Gamma_k
\]
is a tree embedding, where $p_i \colon C_n \rightarrow \Gamma_i$ is the path map based at vertex $i \in [k]$. Since $p_i$ is a path map, we have that the $p$ is $1$-Lipschitz and each $\Gamma_i$ is a tree. So it remains to show that for any $x, y \in V(C_n)$, we have that $d_{C_n}(x,y) = d_{\Gamma_i}(p_i(x), p_i(y))$ for some $i \in [k]$.

Fix $x,y \in V(C_n)$ and let $P$ be the shortest path from $x$ to $y$. So we have one of
\[
P \colon x \rightarrow x+1 \rightarrow \dots \rightarrow y
\quad \text{or} \quad
P \colon x \rightarrow x-1 \rightarrow \dots \rightarrow y.
\]
In each case, we extend $P$ to a path $P'$ of length $\lfloor n/2 \rfloor$, which is a shortest path from $x$ to some vertex $z \in V(C_n)$. So the path $P'$ is of the form
\[
P' \colon x \rightarrow \dots \rightarrow y \rightarrow \dots \rightarrow z.
\]
Observe that the induced subgraph of $C_n$ obtained by removing $x$ and $z$ consists of two paths of length strictly less than $k = \lceil n/2 \rceil$. It follows that exactly one of $x$ or $z$ lies in $[k]$. If $x \in [k]$, then the path $P'$ is isometrically embedded into $\Gamma_x$ by $p_x$, so we have $d_{C_n}(x,y) = d_{\Gamma_x}(p_x(x), p_x(y))$. Similarly if $z \in [k]$ then  $P'$ is isometrically embedded into $\Gamma_z$ by $p_z$, so we have $d_{C_n}(x,y) = d_{\Gamma_z}(p_z(x), p_z(y))$. So we have shown that $p$ is a tree embedding and so $\prank(C_n) \le k$.

Next, we will show that $\prank(C_n) = k$ by contradiction. Assume that we have a tree embedding $\iota = \iota_1 \times \iota_m \colon C_n \rightarrow \Gamma_1 \times \dots \times \Gamma_m$ for some $m < k$.

Assume that $n$ is even. The proof for $n$ odd follows similarly. For each $x \in [k]$ note that the distance, in $C_n$,  between the vertices $x$ and $x+k$ is $d_{C_n}(x, x+k) = k$. So, for each $x \in [k]$ there is some coordinate $j = j_x \in [m]$ of the product space such that $d_{\Gamma_j}(\iota_j(x), \iota_j(x+k)) = k$. Since $m < k$, by the pigeonhole principal, there exists distinct $x, x' \in [k]$ such that $j_x = j_{x'} \eqqcolon j$. Since $x \neq x'$ and $x \neq x'+m$, it follows that
\[
d_{C_n}(x, x') < k,\, d_{C_n}(x+k, x'+k) < k,\, 
d_{C_n}(x, x'+k) < k,\, d_{C_n}(x', x+k) < k. 
\]
Since $d_{C_n}(x, x+k) = d_{C_n}(x',x'+k) = d_{\Gamma_j}(\iota_j(x), \iota_j(x+k)) = d_{\Gamma_j}(\iota_j(x'),\iota_j(x'+k)) = k$, and by the four point condition on $\Gamma_j$, we have that either
\begin{align*}
    & d_{\Gamma_j}(\iota_j(x),\iota_j(x')) + d_{\Gamma_j}(\iota_j(x+k),\iota_j(x'+k)) = 2k \\
    \text{ or } &
    d_{\Gamma_j}(\iota_j(x),\iota_j(x'+k)) + d_{\Gamma_j}(\iota_j(x'),\iota_j(x+k)) = 2k.
\end{align*}
But, in each case, this contradicts the lengths of the edges in $C_n$. So we have shown that there are no tree embeddings of $C_n$ into strictly fewer than $k$ trees. Hence, the rank of $C_n$ is exactly $k$.
\end{proof}

\begin{theorem}\label{thm: examples with rank n/2}
    For $n \ge 6$ even, the following graphs have phylogenetic rank $n/2$:
    \begin{enumerate}
        \item $K_n \setminus C_n$,
        \item $K_n$ minus a perfect matching,
        \item $Q_k$ vertex-edge graph of the $k$-dimensional hypercube where $n = 2^k$,
        \item $K_n$ minus two even vertex-disjoint cycles each of length at least $6$ that cover the vertices of $G$,
        % \item $K_n$ minus two edge-disjoint copies of $C_n$ with $n \ge 10$,
        % \ollie{I can't make the graph $K_n$ minus two edge-disjoint copies of $C_n$ work. If $n \ge 10$ then the graph has diameter $2$ but we may need more conditions}
        \item $J_{2k,k}$ Johnson graph (edge graph of the hypersimplex $\Delta_{2k,k}$) where $n = \binom{2k}{k}$.
    \end{enumerate}
\end{theorem}

\begin{proof}
    For each graph $G$ among the listed families, we apply the following sequence of arguments. First, we exhibit an embedding into $n/2$ trees and conclude $\prank(G) \le n/2$. Then, we construct a partition of the vertices into pairs $P \subseteq \{vw \mid v,w \in V(G)\}$ such that
    for all $vw, xy \in P$ with $vw \neq xy$ the set $v,w,x,y$ does not satisfy the four point condition.
    So, each $vw \in P$ attains its distance in a distinct coordinate of a minimal embedding of $G$. Thus, any minimal embedding has rank at least $|P| = n/2$, thus $\prank(G) = n/2$.

    So, for each graph, we will provide a minimal embedding and the description of the partition $P$ and a proof that it satisfies the criterion.

    \smallskip
    \noindent \textbf{(1)} {Let $G = K_n \setminus C_n$.} Assume that the edges of $C_n$ are $12,\, 23,\, 34, \dots, 1n$. Let
    \[
    p \coloneqq p_1 \times p_3 \times \dots \times p_{n-1} \colon G \rightarrow \Gamma_1 \times \Gamma_3 \times \dots \times \Gamma_{n-1},
    \]
    where $p_i \colon G \rightarrow \Gamma_i$ is the path map based at vertex $i \in [n]$. The map $p$ is $1$-Lipschitz because it is a product of path maps. To see that $p$ is a tree embedding, let $v, w \in [n]$ be vertices. If one of $v, w$ is odd, say $v$, then we have $d_G(v,w) = d_{\Gamma_v}(p_v(v), p_v(w))$. If both $v,w$ are even, then since $n \ge 6$, there exists a vertex $x \in [n]$ such that $xv \in E(C_n)$ and $xw \notin E(C_n)$. Then we have $d_G(x,v) = 2$ and $d_G(x,w) = 1$ so $1 = d_G(v,w) = d_{\Gamma_x}(p_x(v), p_x(w))$. Thus $p$ is a tree embedding of $G$.

    Define $P = \{\{1, 2\}, \{3,4\}, \{5,6\}, \dots, \{n-1, n\}\}$ a partition of the non-edges of $G$. Let $vw, xy \in P$ with $vw \neq xy$. Since $n \ge 6$, there is at most one non-edge among $vx, vy, wx, wy$, so we have:
    \[
    d_G(v,w) + d_G(x,y) = 4, \quad
    d_G(v,x) + d_G(w,y) \le 3, \quad
    d_G(v,y) + d_G(w,y) \le 3.
    \]
    Thus $v,w,x,y$ does not satisfy the four point condition. This concludes the case. 
    
    \smallskip
    \noindent \textbf{(2)} {Let $G = K_n$ minus a perfect matching.}
    Let $\{1,2\},\, \{3,4\},\, \dots, \{n-1,n\}$ be the perfect matching, which we may assume without loss of generality. Let
    \[
    p \coloneqq p_1 \times p_3 \times \dots \times p_{n-1} \colon G \rightarrow \Gamma_1 \times \Gamma_3 \times \dots \times \Gamma_{n-1},
    \]
    where $p_i \colon G \rightarrow \Gamma_i$ is the path map based at vertex $i \in [n]$. The map $p$ is $1$-Lipschitz because it is a product of path maps. To see that $p$ is a tree embedding, let $v, w \in [n]$ be vertices. If one of $v, w$ is odd, say $v$, then we have $d_G(v,w) = d_{\Gamma_v}(p_v(v), p_v(w))$. If both $v,w$ are even, then let $x = v-1$, which is odd. Then we have $d_G(x,v) = 2$ and $d_G(x,w) = 1$ so $1 = d_G(v,w) = d_{\Gamma_x}(p_x(v), p_x(w))$. Thus $p$ is a tree embedding of $G$.

    Define $P = \{\{1, 2\}, \dots, \{n-1, n\}\}$ to be the perfect matching. Let $vw, xy \in P$ with $vw \neq xy$. Notice that $vx, vy, wx, wy$ are all edges of $G$, so we have:
    \[
    d_G(v,w) + d_G(x,y) = 4, \quad
    d_G(v,x) + d_G(w,y) = 2, \quad
    d_G(v,y) + d_G(w,y) = 2.
    \]
    Thus $v,w,x,y$ does not satisfy the four point condition. This concludes the case.

    \smallskip
    \noindent \textbf{(3)} {Let $G = Q_k$ be the $k$-dimensional hypercube with $n = 2^k$ vertices.}
    Let the vertices be $V = \{0,1\}^k$ zero-one vectors of length $k$ and the edges be $E = \{vw \mid ||v - w||_1 = 1\}$ pairs of sequences that differ in exactly one position. Let $V' = \{0,1\}^{k-1}$ be the zero-one sequences of length $k-1$. Define the map
    \[
    p \coloneqq \prod_{v \in V'} p_{(v,0)} \colon G \rightarrow \prod_{v \in V'} \Gamma_{(v,0)}
    \]
    where $p_x \colon G \rightarrow \Gamma_x$ is the path map based at $x \in V$ where we identify the sequences $((v_1, \dots, v_{k-1}), v_k) \equiv (v_1, \dots, v_k)$. Since $p_x$ is a path map, we have that their product $p$ is $1$-Lipschitz. To see that $p$ is a tree embedding fix $v, w \in V$. If the last coordinate is zero, say $v_k = 0$, then we have $d_G(v,w) = d_{\Gamma_v}(p_v(v), p_v(w))$, similarly for $w_k = 0$. Otherwise assume $v_k = w_k = 1$. Then define $x = (v_1, \dots, v_{k-1}, 0)$. It follows that $d_G(x,v) = 1$ and $d_G(x,w) = d_G(v,w) + 1$, so we have $d_G(v,w) = d_{\Gamma_x}(p_x(v), p_x(w))$. Thus, we have shown that $p$ is a tree embedding.

    Let $P = \{v \overline v \mid v \in V\}$ be a partition of $V$ into pairs where $\overline v$ is the coordinate-wise opposite of $v$ with coordinates $\overline v_i \coloneqq 1 - v_i$ for each $i \in [k]$. Let $v\overline v, w \overline w \in P$ with $v \overline v \neq w \overline w$. Since $v \neq w$ and $v \neq \overline w$, it follows that $v_i = w_i$ and $v_j = \overline w_j$ for some $i, j \in [k]$. So we have
    \[
    d(v,\overline v) + d(w,\overline w) = 2k, \quad
    d(v, w) + d(\overline v, \overline w) \le 2k-2, \quad
    d(v, \overline w) + d(\overline v, w) \le 2k-2.
    \]
    Thus $v, \overline v, w, \overline w$ does not satisfy the four point condition. This concludes the case.
    
    \smallskip
    \noindent \textbf{(4)} {Let $G = K_n$ minus two even vertex-disjoint cycles each of length at least $6$ that cover the vertices $G$.}
    Without loss of generality, let the $C_1$ be the cycle with edges $\{1,2\}, \{2,3\}, \dots, \{k-1,k\}, \{1,k\}$ and $C_2$ be the cycle with edges $\{k+1,k+2\}, \dots, \{n-1,n\}, \{k+1,n\}$, where $k \in [n]$ is even. Let
    \[
    p \coloneqq p_1 \times p_3 \times \dots \times p_{n-1} \colon G \rightarrow \Gamma_1 \times \Gamma_3 \times \dots \times \Gamma_{n-1},
    \]
    where $p_i \colon G \rightarrow \Gamma_i$ is the augmented path map based at vertex $i \in [n]$ which we define as follows.
    The space $\Gamma_i$ is a metric tree with vertices $\overline 1, \overline 2, \overline 3, \overline 4, \overline 5$ and edges $\overline 1 \overline 2, \overline 2 \overline 3, \overline 3 \overline 4, \overline 3 \overline 5$ whose lengths are defined to be 
    \[
    \ell(\overline 1 \overline 2) = 1
    \quad \text{and} \quad 
    \ell(\overline 2 \overline 3) = \ell(\overline 3 \overline 4) = \ell(\overline 3 \overline 5) = 1/2.
    \]
    Let $s, t$ be the neighbours of vertex $i$ in the cycles $C_1 \cup C_2$. Define the map
    \[
    p_i \colon G \rightarrow \Gamma_i, \quad
    p_i(x) = \begin{cases}
        \overline 1 & \text{if } x = i,\\
        \overline 4 & \text{if } x = s,\\
        \overline 5 & \text{if } x = t,\\
        \overline 2 & \text{otherwise}.
    \end{cases}
    \]
    The distance between $i$ and each vertex that is neither $s$ nor $n$ is one, so $p_i$ realises every distance involving $i$ and $p_i$ is $1$-Lipschitz. To see that $p$ is a tree embedding, let $v, w \in [n]$ be vertices. If one of $v, w$ is odd, say $v$, then we have $d_G(v,w) = d_{\Gamma_v}(p_v(v), p_v(w))$. If both $v,w$ are even, then $d_G(v,w) = 1$ and there exists an odd vertex $x \in [n]$ such that $xv \in E(C_1) \cup E(C_2)$. Then, by construction $p_x(v)$ is a leaf of $\Gamma_x$ and is at distance from the image of all other vertices of $G$. So we have $d_G(v,w) = 1 = d_{\Gamma_x}(p_x(v), p_x(w))$. Thus $p$ is a tree embedding of $G$.

    Define $P = \{\{1, 2\}, \{3,4\}, \{5,6\}, \dots, \{n-1, n\}\}$ a partition of the non-edges of $G$. Let $vw, xy \in P$ with $vw \neq xy$. Since each cycle has length at least six, there is at most one non-edge among $vx, vy, wx, wy$, so we have:
    \[
    d_G(v,w) + d_G(x,y) = 4, \quad
    d_G(v,x) + d_G(w,y) \le 3, \quad
    d_G(v,y) + d_G(w,y) \le 3.
    \]
    Thus $v,w,x,y$ does not satisfy the four point condition. This concludes the case.

    \smallskip
    \noindent \textbf{(5)} {Let $G = J_{2k,k}$ be the Johnson graph with $n = \binom{2k}{k}$ vertices.}
    Let the vertices of $G$ be $V = \binom{[2k]}{k}$ the $k$-subsets of $[2k]$ and the edges be $E = \{vw \mid  |v \cap w| = k-1\}$ pairs of subsets that intersect on $k-1$ elements. Let $V' = \{v \in V \mid 1 \in v\}$ be the vertices containing $1$. As a consequence we have $d_G(v,w) = k - |v \cap w|$. Define the map
    \[
    p \coloneqq \prod_{v \in V'} p_v \colon G \rightarrow \prod_{v \in V'} \Gamma_v
    \]
    where $p_v \colon G \rightarrow \Gamma_v$ is the path map based at $v$. Since $p_v$ is a path map, the product $p$ is $1$-Lipschitz. To see that $p$ is a tree embedding, fix $v, w \in V$. If $v \in V'$, then $d_G(v,w) = d_{\Gamma_v}(p_v(v), p_v(w))$, similarly if $w \in V'$. Otherwise assume $v \notin V'$ and $w \notin V'$. Since $v, w$ are $k$-subsets of $\{2,3, \dots, 2k\}$, by the pigeonhole principle, there exists $i \in v \cap w$. Define $x = (v \setminus \{i\}) \cup \{1\}$. It follows that $d_G(x,v) = 1$ and $d_G(x,w) = d_G(v,w) + 1$, so we have $d_G(v,w) = d_{\Gamma_x}(p_x(v), p_x(w))$. Thus, we have shown that $p$ is a tree embedding.

    Let $P = \{v \overline v \mid v \in V\}$ be a partition of $V$ into pairs where $\overline v = [2k] \setminus v$ is the complement of $v$. Let $v\overline v, w \overline w \in P$ with $v \overline v \neq w \overline w$. Since $v \neq w$ and $v \neq \overline w$, it follows that $v \cap w \neq \emptyset$ and $v \cap \overline w \neq \emptyset$. So we have
    \[
    d(v,\overline v) + d(w,\overline w) = 2k, \quad
    d(v, w) + d(\overline v, \overline w) \le 2k-2, \quad
    d(v, \overline w) + d(\overline v, w) \le 2k-2.
    \]
    Thus $v, \overline v, w, \overline w$ does not satisfy the four point condition. This concludes the case.

    In each case we provided an embedding of $G$ into $n/2$ trees and constructed a partition of the vertices into pairs such that no pair of pairs satisfies the four point condition. By the initial argument of the proof, we have that $\prank(G) = n/2$ for each graph and we are done.
\end{proof}

In the theorem above, we showed that the family of graphs consisting of $K_n$ minus two even vertex-disjoint cycles of length at least $6$ such that the vertices of the cycles cover the graph have rank $n/2$. If we relax the condition on the lengths of the cycles and allow cycles of length four, then the next example show that the rank may drop.

\begin{example}\label{example: K8 minus 2x C4}
    Let $G$ be the graph $K_8$ minus two vertex-disjoint copies of $C_4$. Let $1,2,3,4$ and $5,6,7,8$ be the vertices of the cycles. Since $G$ contains an induced $4$-cycle $1,5,2,6,1$, by Theorem~\ref{thm: phylogenetic rank one}, the graph $G$ has rank at least two. In fact, we can construct a rank-two embedding $G \rightarrow \Gamma_1 \times \Gamma_2$ as shown in Figure~\ref{fig: K8 minus 2x C4 embed}. In particular, we see that pairs of vertices with distance two such that $12$ and $34$ have their distance realised in the same tree $\Gamma_1$. This is not possible if the removed cycles have length at least six.

    \begin{figure}
        \centering
        \begin{tikzpicture}[
            thick,
            leaf/.style     = {circle, fill, inner sep=1.6pt},
            internal/.style = {circle, inner sep=0pt, outer sep=0pt, minimum size=0pt},
            centre/.style   = {circle, fill, inner sep=1.6pt},
            lbl/.style      = {font=\small, inner sep=2.5pt}
          ]
         
          %% ======================= first tree =======================
          % --- vertices ---
          \node[internal] (u1) at (-1.25  , 0  ) {};   % left  internal vertex (unlabelled)
          \node[internal] (v1) at ( 1.25  , 0  ) {};   % right internal vertex (unlabelled)
          \node[centre]   (m1) at ( 0  , 0  ) {};   % midpoint of the internal edge
          \node[leaf] (a1) at (-2.2, 1.2) {};
          \node[leaf] (a2) at (-2.2,-1.2) {};
          \node[leaf] (a3) at ( 2.2, 1.2) {};
          \node[leaf] (a4) at ( 2.2,-1.2) {};
          % --- edges ---
          \draw (a1) -- (u1) -- (a2);
          \draw (a3) -- (v1) -- (a4);
          \draw (u1) -- (m1) -- (v1);
          % --- labels ---
          \node[lbl, above left ] at (a1) {$1$};
          \node[lbl, below left ] at (a2) {$3$};
          \node[lbl, above right] at (a3) {$2$};
          \node[lbl, below right] at (a4) {$4$};
          \node[lbl, above=5pt  ] at (m1) {$5,6,7,8$};
          \node[font=\small] at (0,-1.25)  {$\Gamma_1$};
         
          %% ======================= second tree ======================
          % --- vertices ---
          \node[internal] (u2) at ( 6.75  , 0  ) {};   % left  internal vertex (unlabelled)
          \node[internal] (v2) at ( 9.25  , 0  ) {};   % right internal vertex (unlabelled)
          \node[centre]   (m2) at ( 8  , 0  ) {};   % midpoint of the internal edge
          \node[leaf] (b1) at ( 5.8, 1.2) {};
          \node[leaf] (b2) at ( 5.8,-1.2) {};
          \node[leaf] (b3) at (10.2, 1.2) {};
          \node[leaf] (b4) at (10.2,-1.2) {};
          % --- edges ---
          \draw (b1) -- (u2) -- (b2);
          \draw (b3) -- (v2) -- (b4);
          \draw (u2) -- (m2) -- (v2);
          % --- labels ---
          \node[lbl, above left ] at (b1) {$5$};
          \node[lbl, below left ] at (b2) {$7$};
          \node[lbl, above right] at (b3) {$6$};
          \node[lbl, below right] at (b4) {$8$};
          \node[lbl, above=5pt  ] at (m2) {$1,2,3,4$};
          \node[font=\small] at (8,-1.25)  {$\Gamma_2$};
        \end{tikzpicture}\vspace{-2mm}
        \caption{Coordinates of the tree embedding in Example~\ref{example: K8 minus 2x C4}}
        \label{fig: K8 minus 2x C4 embed}
    \end{figure}
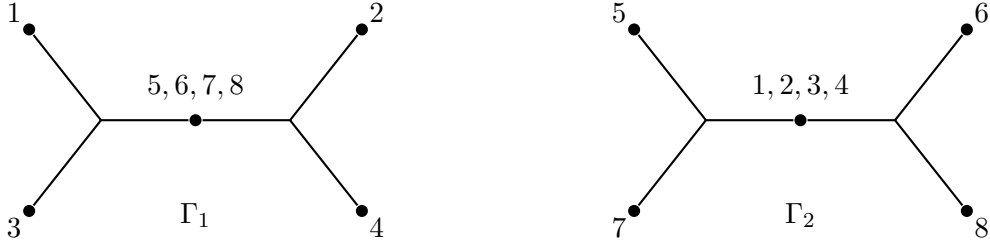
\end{example}

\begin{proposition}\label{prop:bipartiteGraph}
    The complete bipartite graph $K_{n,m}$ has phylogenetic rank two for $n,m \ge 2$.
\end{proposition}

\begin{proof}
    Let us explicitly label the graph $K_{n,m}$ with vertices $V = \{a_1, \dots, a_n, b_1, \dots, b_m\}$ and edges $E = \{a_i b_j \mid i \in [n],\, j \in [m]\}$. Since $2 \le n, m$, the graph $K_{n,m}$ has an induced four-cycle on vertex set $\{a_1, a_2, b_1, b_2\}$. So by Theorem~\ref{thm: phylogenetic rank one}, we have that $K_{n,m}$ does not have phylogenetic rank one. Hence $K_{n,m}$ has phylogenetic rank at least two.

    Let $\Gamma_1 = (V_1, E_1)$ be the metric tree with vertices $V_1 = \{a, b_1, \dots, b_m\}$ and unit-length edges $E_1 = \{ab_j \mid j \in [m]\}$, and $\Gamma_2 = (V_2, E_2)$ be the tree with vertices $V_2 = \{a_1, \dots, a_n, b\}$ and unit-length edges $E_2 = \{a_ib \mid i \in [n]\}$. We define a graph embedding by
    \[
    \iota\colon K_{n,m} \rightarrow \Gamma_1 \times \Gamma_2,
    \quad 
    \iota(a_i) = (a, a_i) \text{ and }
    \iota(b_j) = (b_j, b) \text{ for all }
    i \in [n],\, j \in [m].
    \]
    It is straightforward to show that $\iota$ is a $1$-Lipshitz map. To show that the map defines an embedding, note that the diameter of $K_{n,m}$ is two. All length two paths are of the form $(a_i, b_j, a_k)$, which are mapped isometrically into $\Gamma_2$, or $(b_i, a_j, b_k)$, which are mapped isometrically into $\Gamma_1$. This concludes the proof.
\end{proof}

We will now classify the rank one graphs in Theorem~\ref{thm: phylogenetic rank one}, which intuitively says: all phylogenetic rank one graphs are ``trees of cliques'' (or more concretely, they are $1$-sums of cliques); and secondly the only obstructions to phylogenetic rank one are cycles of length four or more, and $K_4 \setminus e$, see Proposition~\ref{prop: cycle has tree embedding number n/2} and the following example. 

\begin{example}\label{example: K4 minus edge}Let $G = K_4 \setminus e$ be the complete graph $K_4$ with some edge removed. Explicitly, we take $V(G) = [4]$ and $E(G) = \{12, 13, 23, 14, 24\}$.
To see that $G$ does not satisfy the four point condition, observe that
\[
3 = d(1,2) + d(3,4) > d(1,3) + d(2,4) = d(1,4) + d(2,3) = 2. 
\]
Therefore, the metric graph $G$ must have phylogenetic rank at least two. If $p_1\colon G \rightarrow \Gamma_1$ and $p_3\colon G \rightarrow \Gamma_3$ are the path maps based at $1$ and $3$ respectively, then it is straightforward to show that $p_1 \times p_3\colon G \rightarrow \Gamma_1 \times \Gamma_3$ is a tree embedding of $G$. So we have that $\prank(G) = 2$.
\end{example}

We now recall the necessary terminology for Theorem~\ref{thm: phylogenetic rank one}. Let $G$ be a graph and $C$ a cycle of $G$. A \textit{chord} of $C$ is an edge of $G$ that connects two non-consecutive vertices of $C$. We say that a graph $G$ is \textit{chordal} if every cycle of $G$ of length at least four has a chord. Let $H$ be a graph. We say that $G$ is \textit{$H$-free} if no induced subgraph of $G$ is isomorphic to $H$.

Given two graphs $G$ and $H$, a \textit{$1$-sum} of $G$ and $H$ is a graph obtained from their disjoint union with one vertex in each of $G$ and $H$ identified. We say that $G$ is a \textit{$1$-sum of cliques} if it can be obtained from a clique by performing a sequence of $1$-sums with cliques. For instance, a tree is $1$-sum of cliques as it can be obtained from a single vertex by performing a sequence of $1$-sums with $K_2$. Note that if $G$ is a $1$-sum of cliques, then the set of maximal cliques of $G$ are exactly cliques required to form $G$ with a sequence of $1$-sums.

\begin{theorem}\label{thm: phylogenetic rank one}
    Let $G$ be a graph. Then the following are equivalent:
    \begin{enumerate}
        \item $G$ has phylogenetic rank one,
        \item $G$ is free of $K_4 \setminus e$ and the cycle graphs $C_i$ for all $i \ge 4$,
        \item $G$ is chordal and $(K_4 \setminus e)$-free,
        \item $G$ is a $1$-sum of cliques.
    \end{enumerate}
\end{theorem}

Before proving the result, we require a small lemma.

\begin{lemma}\label{lemma: isometric large cycles}
    Let $G$ be a $(K_4 \setminus e)$-free graph. If $G$ contains an induced $n$-cycle $C$ for some $n \ge 4$, then $G$ contains an isometric $m$-cycle $C'$ for some $4 \le m \le n$.
\end{lemma}

\begin{proof}
    We prove the result by induction on $n$. For the base case, assume $n = 4$. We show that $C$ is an isometric subgraph. Suppose by contradiction that $d_C(x,y) \neq d_G(x,y)$ for some $x,y \in V(C)$. Since $C$ is a subgraph of $G$, we must have $d_G(x,y) < d_C(x,y)$. Since $C$ has diameter $2$, we have that $d_C(x,y) \in \{1,2\}$. So we must have $d_G(x,y) = 1$ and $d_C(x,y) = 2$. So $xy \in E(G)$ is an edge of $G$ and $C$ is not an induced cycle, which is a contradiction. Therefore $C$ is an isometric subgraph.

    For the induction step, suppose $n \ge 5$. If $C$ is an isometric subgraph, then we are done, so assume that $C$ is not an isometric subgraph of $G$. Then there exist vertices $x, y \in V(C)$ such that $d_C(x,y) \neq d_G(x,y)$. Since $C$ is a subgraph of $G$, it follows that $d_G(x,y) < d_C(x,y)$. By definition of the metric, there is an induced path $P$ in $G$ from $x$ to $y$ of length $d_G(x,y)$. Since $C$ is an induced subgraph of $G$, it follows that $x,y$ are not adjacent in $G$, so $P$ has length at least two. Let $Q$ be longer of the two induced paths in $C$ from $x$ to $y$. Since $n \ge 5$, it follows that the length of $Q$ is at least $3$.

    Let $H$ be the induced subgraph of $G$ with vertices $V(P) \cup V(Q)$. If $H$ is an induced cycle, i.e., $H = P \cup Q$, then notice that its length $n'$ satisfies $4 \le n' < n$. The inequality $n' < n$ follows from the assumption that the length of $P$ is strictly smaller than $d_C(x,y)$. The inequality $4 \le n'$ follows from the fact that $Q$ is at least three. So by induction, $G$ contains an isometric $m$-cycle for some $4 \le m \le n'$.

    So, we may assume that $H$ is not an induced cycle. Let us label the vertices along the paths as follows
    \[
    P \colon x \eqqcolon v_0 \rightarrow v_1 \rightarrow \dots \rightarrow v_s \coloneqq y
    \quad \text{and} \quad
    Q \colon x \eqqcolon w_0 \rightarrow w_1 \rightarrow \dots \rightarrow w_t \coloneqq y.
    \]
    Since $P$ and $Q$ are induced paths of $G$, it follows that there are no edges $v_iv_j$ or $w_iw_j$ in $H$. In particular, the neighbourhood of $x$ in $H$ is exactly $\{v_1, w_1\}$. Similarly, the neighbourhood of $y$ in $H$ is $\{v_{s-1}, w_{t-1}\}$.
    Since the length of $Q$ is at least $3$, we have $t \ge 3$ and since the length of $P$ is at least $2$, we have $s \ge 2$.
    Since $H$ is not an induced cycle, there exists an edge $v_iw_j$ for some $i \in [s-1]$ and $j \in [t-1]$. We assume that $i$ is the minimum such value and $j$ is taken to be the minimum for $w_j$ in the neighbourhood of $v_i$. 
    
    \begin{claim}
        $H$ contains an induced cycle of length at least four.
    \end{claim}

    \begin{proof}
    Let us assume that $j > 1$. By the minimality of $i$ and $j$, there is an induced cycle of length at least four: $
    w_0, w_1, \dots, w_j, v_i, v_{i-1}, \dots, v_0 = w_0
    $.
    So from now on, we may assume $j = 1$.

    Let us assume that $i > 1$. By the minimality of $i$, there is an induced cycle of length at least four: $
    v_0, v_1, \dots, v_i, w_1, w_0 = v_0
    $.
    So from now on, we may assume $i = 1$.

    Assume that $v_1$ has another neighbour $w_{j'}$ for some $j' \in [t-1]$. Take $j'$ to be the minimum such value with $j' > 1$. We must have that $j' > 2$ because if $j' = 2$ then the vertices $v_0, v_1, w_1, w_2$ give rise to an induced copy of $K_4 \setminus e$, which contradicts our initial assumption that $G$ is $(K_4 \setminus e)$-free. Since $j' > 2$, there is an induced cycle of length at least four:
    $
    v_i, w_j, w_{j+1}, \dots, w_{j'}, v_i
    $. So from now on, we may assume that the only neighbours of $v_i$ are $v_{i-1}, v_{i+1}, w_j = w_1$. So from now on we may assume that the neighbours of $v_1$ are $v_0, v_2, w_1$.

    Assume that there are no other edges $v_{i'}w_{j'}$ for $i' \in [s-1]$ and $j' \in [t-1]$. Then, since $t \ge 3$, we have a cycle of length at least four: $v_1, v_2, \dots, v_s = w_t, w_{t-1}, \dots, w_1, v_1$. So from now on, we may assume there is another edge $v_{i'}w_{j'}$ with $i' \in [s-1]$ and $j' \in [t-1]$.

    If $i' = 2$, then we must have $j' > 1$ otherwise the set of vertices $v_0, v_1, v_2, w_1$ would give rise to an induced copy of $K_4 \setminus e$, which contradicts the initial assumption that $G$ is $(K_4 \setminus e)$-free. Thus $j' > 1$ and, by minimality of $i'$ and $j'$, there is an induced cycle of length at least four: $v_1, v_2, \dots, v_{i'}, w_{j'}, w_{j'-1}, \dots, w_1, v_1$. So from now on, we may assume that $i' > 2$.

    Finally, by the minimality of $i'$, we have an induced cycle of length at least four: $v_1, v_2, \dots, v_{i'}, w_{j'}, w_{j'-1}, \dots, w_j, v_i$. Note that in this case it is possible to have $j = j'$. This concludes the proof of the claim.
    \end{proof}

    So $H$ contains an induced cycle of length at least four, say on $n''$ vertices with $4 \le n'' < n' < n$. So by induction, we have that $G$ contains an isometric $m$-cycle for some $4 \le m \le n'' < n$. This concludes the proof of the result. 
\end{proof}

We are now ready to prove the classification of rank one graphs.

\begin{proof}[Proof of Theorem~\ref{thm: phylogenetic rank one}] $(1) \implies (2)$.
    We prove the contrapositive. Assume that $G$ either contains an induced copy of $K_4 \setminus e$ or a cycle $C_n$ for $n\ge 4$. 

    Suppose that $G$ contains an induced copy $H$ of $K_4 \setminus e$. We will show that $H$ is an isometric subgraph of $G$. If $H$ is not an isometric subgraph of $G$ then there exist vertices $x,y \in V(H)$ with $d_H(x,y) \neq d_G(x,y)$. Since the diameter of $H$ is two, we must have $d_H(x,y) = 2 > 1 = d_G(x,y)$. But this means $x,y$ are adjacent in $G$ so $H$ is not an induced subgraph, which is a contradiction. Thus $H$ is an isometric subgraph. But then, by Example~\ref{example: K4 minus edge}, we have that $K_4 \setminus e$ has rank two, so by Proposition~\ref{prop: isometric subgraph rank bound} it follows that $\prank(G) \ge \prank(H) \ge 2$, so $G$ does not have rank one. So, from now on, we may assume that $G$ is $(K_4\setminus e)$-free.

    Suppose that $G$ contains an induced $n$-cycle with $n \ge 4$. Then by Lemma~\ref{lemma: isometric large cycles}, there is an induced $m$-cycle $C$ with $m \ge 4$ such that $C$ an isometric subgraph of $G$. By Proposition~\ref{prop: cycle has tree embedding number n/2} and Proposition~\ref{prop: cycle has tree embedding number n/2}, it follows that $\prank(G) \ge \prank(C) \ge 2$. So $G$ does not hand rank one. This concludes the proof of this implication.

    \medskip \noindent
    $(2) \implies (3)$. This implication follows immediately from the definition.

    \medskip \noindent $(3) \implies (4)$.
    Assume that $G$ is chordal and contains no copies of $K_4 \setminus e$. We will use the well-known characterisation of chordal graphs in terms of \textit{perfect elimination orderings}, which says that $G$ is chordal if and only if there is an ordering of the vertices $v_1 < v_2 < \dots < v_n$ of $G$ such that for each $i \in [n]$ the vertices greater than $v_i$ in its neighbourhood form a clique, explicitly the set $\{v_j \in V(G) \mid v_i < v_j,\, v_iv_j \in E(G)\}$ is a clique of $G$. For ease of notation let us write $N_G(v) = \{w \in V(G) \mid vw \in E(G)\}$ for the neighbourhood of $v$ in $G$, and note that $v \notin N_G(v)$.

    To show that $G$ is a $1$-sum of cliques, we proceed by induction on $n$ the number of vertices of $G$. The base case is the graph with one vertex, which is trivially a $1$-sum of cliques. For the induction step, let us assume that we have a perfect elimination order $v_1 < \dots < v_{n-1} < v_{n}$ and that there is a $1$-sum decomposition of $G[V']$ into cliques where $V' = \{v_2, \dots, v_n\}$. We now check that $G$ has a $1$-sum decomposition into cliques by taking cases on the size of neighbourhood $N_G(v)$ of $v$.

    \smallskip \noindent 
    \textbf{Case 1.} Assume $|N_G(v)| = 1$, then $G$ is a $1$-sum of $G[V']$ and $K_2$, and we are done.

    \smallskip \noindent
    \textbf{Case 2.} Assume that $|N_G(v)| \ge 2$. Then, by the perfect elimination ordering, we have that $N_G(v)$ is a clique of $G[V']$. Next we show that $N_G(v)$ must be a clique of the $1$-sum decomposition of $G[V']$, i.e., a maximal clique of $G[V']$. Since $N_G(v)$ is a clique, it must be contained in a maximal clique of $G[V']$. If $N_G(v)$ is not maximal, then take $x$ be any vertex of that maximal clique that is not in $N_G(v)$ and let $y,z \in N_G(v)$ be any two distinct vertices. Then the set $\{v,x,y,z\}$ forms an induced $K_4 \setminus e$, a contradiction. Thus $N_G(v)$ is a maximal clique of $G$. So, by adding $v$ to the clique $N_G(v)$ in the $1$-sum decomposition of $G[V']$, we obtain a $1$-sum decomposition of $G$ into cliques.

    \medskip \noindent
    $(4) \implies (1)$. Assume that $G$ is a $1$-sum of cliques. We will show that $G$ embeds into a metric tree by induction on the number of maximal cliques. If $G$ has one clique, i.e., $G = K_n$, then it embeds into the star graph, see Example~\ref{example: Kn has rank 1}. For the induction step, assume that $G$ is a $1$-sum of at least two cliques. Let $G'$ be a maximal clique of $G$. Then $G$ is the $1$-sum of $G'$ and a graph $H$ with one fewer maximal cliques than $G$. By induction we have that $H$ embeds into a tree $\Gamma$ and $G'$ embeds into a tree $\Lambda$. It is easy to see that a $1$-sum that $G$ embeds into a $1$-sum of $\Gamma$ and $\Lambda$, which is a tree. Thus $G$ has phylogenetic rank one. This completes the proof.
\end{proof}

%%% Local Variables:
%%% mode: LaTeX
%%% TeX-master: "graph_embedding_paper"
%%% End:

\section{A heuristic algorithm}\label{sec:heuristicAlgorithm}

In this section, we describe our greedy algorithm for constructing an embedding $\iota\colon G\rightarrow \Gamma_1 \times \dots \times \Gamma_r$.  It starts with an empty partial map $G\dashrightarrow \emptyset$.  At each step of the iteration, we extend the current partial map $\iota G\dashrightarrow \Gamma_1 \times \cdots \times \Gamma_k$, either by adding a new metric tree $\Gamma_{k+1}$ or by extending one $\Gamma_i$, until it eventually becomes an isometry.

To phrase our algorithm, we require the following con to realize one unrealized distance $d_G(v_i,v_j)$,  in the following sense:

\begin{definition}
  \label{def:greedyExtension}
  Let $G \overset{\iota}{\dashrightarrow} \Gamma$ be a partial $1$-Lipschitz map, and let $v_i,v_j\in V$ be two vertices of $G$.  An \emph{($1$-Lipschitz) extension realizing distance $d_G(v_i,v_j)$} is a partial map $\iota'\colon G\dashrightarrow \Gamma'$, where
  \begin{enumerate}
    \item $\Gamma'\supseteq\Gamma$ is an extension of $\Gamma$,
    \item $\iota'|_{\mathrm{domain}(\iota)} = \iota$,
    \item $v_i,v_j\in\mathrm{domain}(\iota')$, and
    \item $d_{\Gamma'}(\iota'(v_i),\iota'(v_j))=d_G(v_i,v_j)$.
  \end{enumerate}
\end{definition}

This extension can happen in one of four ways, outlined in Algorithms \ref{alg:extension1},  \ref{alg:extension2}, \ref{alg:extension3}, and \ref{alg:extension4}.  All algorithms contain steps where choices have to be made, which we address in \cref{rem:extensionChoices}.

\begin{algorithm}[\texttt{extension\_type\_1}]\label{alg:extension1}\
  \begin{algorithmic}[1]
    \REQUIRE{$(G\overset{\iota}{\dashrightarrow}\Gamma, \{v_i,v_j\}\in \binom{V}{2})$, where
      \begin{enumerate}
        \item $G\overset{\iota}{\dashrightarrow}\Gamma$ is a partial $1$-Lipschitz map into a metric tree $\Gamma$,
        \item $v_i\notin \domain(\iota)$,
        \item $v_j\in\domain(\iota)$.
      \end{enumerate}}
    \ENSURE{$(p,G\overset{\iota'}{\dashrightarrow}\Gamma')$, where
      \begin{enumerate}
        \item $p\in\{\texttt{true},\texttt{false}\}$ indicates whether the extension was successful,
        \item if $p=\texttt{true}$, then $G\overset{\iota'}{\dashrightarrow}\Gamma'$ an extension realizing distance $d_G(v_i,v_j)$.
        \item if $p=\texttt{false}$, then $G\overset{\iota'}{\dashrightarrow}\Gamma'=G\overset{\iota}{\dashrightarrow}\Gamma$.
      \end{enumerate}}
    \FOR{$\{a,b\}\in E(\Gamma)$}
    \STATE Consider the following linear system on the variables $x_a,x_b,x_c$:
    {
      \allowdisplaybreaks
      \setlength{\jot}{1pt}
      \begin{align}
        \notag & x_a\geq 0, x_b\geq 0, x_c\geq 0, x_a+x_c = d_{\Gamma}(a,b), \\
        \notag & d_\Gamma(\iota(v_j),a)) + x_a + x_c = d_G(v_i,v_j) \text{ if } d_\Gamma(\iota(v_j),a)<d_\Gamma(\iota(v_j),b)\\
        \notag & d_\Gamma(\iota(v_j),b)) + x_b + x_c = d_G(v_i,v_j) \text{ if } d_\Gamma(\iota(v_j),a)>d_\Gamma(\iota(v_j),b)\\
        \label{eq:systemForType1} & d_\Gamma(\iota(v),a)) + x_a + x_c \leq d_G(v_i,v)\\
        \notag & \hspace{25mm}\text{for } v\in\domain(\iota) \text{ with } d_\Gamma(\iota(v),a)<d_\Gamma(\iota(v),b)\\
        \notag & d_\Gamma(\iota(v),b)) + x_b + x_c \leq d_G(v_i,v)\\
        \notag & \hspace{25mm}\text{for } v\in\domain(\iota) \text{ with } d_\Gamma(\iota(v),a)>d_\Gamma(\iota(v),b)
      \end{align}%
    }\vspace{-1em}
    \IF{System \eqref{eq:systemForType1} has a solution $(x_a,x_b,x_c)\in\RR^3_{\geq 0}$}
    \STATE Let $\Gamma'$ be $\Gamma$ but with the following added (see \cref{fig:extensions} (1)):
    \begin{enumerate}
      \item \hspace{-10mm} a vertex $o$ on edge $\{a,b\}$ at distance $x_a$ from $a$ and $x_b$ from $b$,
      \item \hspace{-10mm} a vertex $c$, and
      \item \hspace{-10mm} an edge $\{c,o\}$ of length $x_c$.
    \end{enumerate}
    \STATE Let $G\overset{\iota'}{\dashrightarrow}\Gamma'$ be $G\overset{\iota}{\dashrightarrow}\Gamma$ but with $v_i$ mapped to $c$.
    \RETURN{$(\texttt{true}, G\overset{\iota'}{\dashrightarrow}\Gamma')$}
    \ENDIF
    \ENDFOR
    \RETURN{$(\texttt{false}, G\overset{\iota}{\dashrightarrow}\Gamma)$}
  \end{algorithmic}
\end{algorithm}

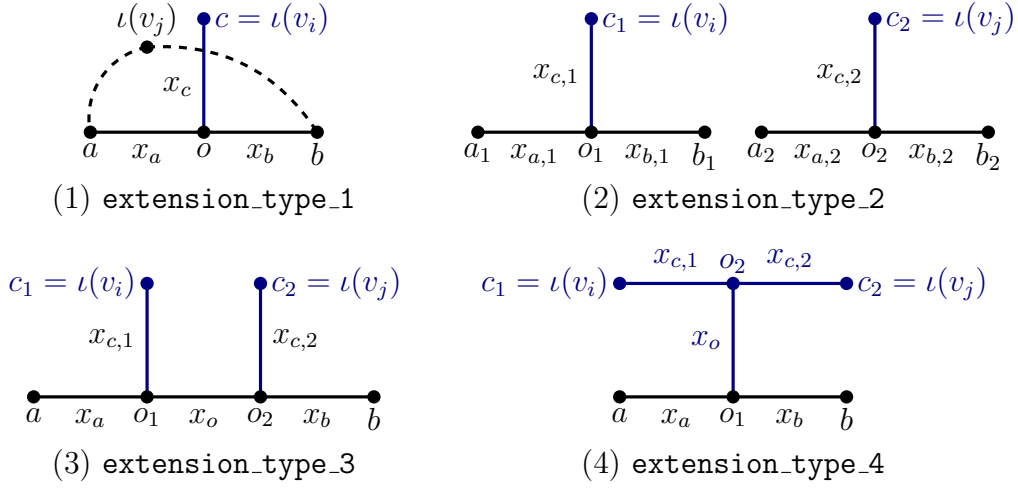
\begin{figure}[t]
  \centering
  \begin{tikzpicture}
    \node (type1) at (0,0)
    {
      \begin{tikzpicture}[x={(1.5,0)},y={(0,1.5)}]
        \useasboundingbox (-1,0) rectangle (1,1);
        \coordinate (o) at (0,0);
        \coordinate (a) at (-1,0);
        \coordinate (b) at (1,0);
        \coordinate (c) at (0,1);
        \coordinate (vj) at (-0.5,0.75);
        \draw[very thick]
        (a) -- node[below] {$x_a$} (o)
        (b) -- node[below] {$x_b$} (o);
        \draw[very thick, blue!50!black]
        (c) -- node[black,left,pos=0.6] {$x_c$} (o);
        \draw[dashed, very thick]
        (vj) to[bend right=45] (a)
        (vj) to[bend left=30] (b);
        \fill
        (o) circle (2.5pt) node[below] {$o$}
        (a) circle (2.5pt) node[below] {$a$}
        (b) circle (2.5pt) node[below] {$b$}
        (vj) circle (2.5pt) node[above] {$\iota(v_j)$};
        \fill[blue!50!black]
        (c) circle (2.5pt) node[right] {$c=\iota(v_i)$};
      \end{tikzpicture}
    };
    \node[below,yshift=-4mm] at (type1.south) {(1) \texttt{extension\_type\_1}};
    \node (type2) at (7,0)
    {%
      \begin{tikzpicture}[x={(1.5,0)},y={(0,1.5)}]
        \useasboundingbox (-2.25,0) rectangle (2.25,1);
        \coordinate (o1) at (-1.25,0);
        \coordinate (a1) at ($(o1)+(-1,0)$);
        \coordinate (b1) at ($(o1)+(1,0)$);
        \coordinate (c1) at ($(o1)+(0,1)$);
        \coordinate (o2) at (1.25,0);
        \coordinate (a2) at ($(o2)+(-1,0)$);
        \coordinate (b2) at ($(o2)+(1,0)$);
        \coordinate (c2) at ($(o2)+(0,1)$);
        \draw[very thick]
        (a1) -- node[below] {$x_{a,1}$} (o1)
        (b1) -- node[below] {$x_{b,1}$} (o1)
        (a2) -- node[below] {$x_{a,2}$} (o2)
        (b2) -- node[below] {$x_{b,2}$} (o2);
        \draw[very thick, blue!50!black]
        (c1) -- node[black,left] {$x_{c,1}$} (o1)
        (c2) -- node[black,left] {$x_{c,2}$} (o2);
        \fill
        (o1) circle (2.5pt) node[below] {$o_1$}
        (a1) circle (2.5pt) node[below] {$a_1$}
        (b1) circle (2.5pt) node[below] {$b_1$}
        (o2) circle (2.5pt) node[below] {$o_2$}
        (a2) circle (2.5pt) node[below] {$a_2$}
        (b2) circle (2.5pt) node[below] {$b_2$};
        \fill[blue!50!black]
        (c1) circle (2.5pt) node[right] {$c_1=\iota(v_i)$}
        (c2) circle(2.5pt) node[right] {$c_2=\iota(v_j)$};
      \end{tikzpicture}
    };
    \node[below,yshift=-4mm] at (type2.south) {(2) \texttt{extension\_type\_2}};
    \node (type3) at (0,-3.5)
    {%
      \begin{tikzpicture}[x={(1.5,0)},y={(0,1.5)}]
        \useasboundingbox (-1.5,0) rectangle (1.5,1);
        \coordinate (o1) at (-0.5,0);
        \coordinate (a) at ($(o1)+(-1,0)$);
        \coordinate (c1) at ($(o1)+(0,1)$);
        \coordinate (o2) at (0.5,0);
        \coordinate (b) at ($(o2)+(1,0)$);
        \coordinate (c2) at ($(o2)+(0,1)$);
        \draw[very thick]
        (a) -- node[below] {$x_a$} (o1)
        (o1) -- node[below] {$x_o$} (o2)
        (o2) -- node[below] {$x_b$} (b);
        \draw[very thick, blue!50!black]
        (c1) -- node[black,left] {$x_{c,1}$} (o1)
        (c2) -- node[black,right] {$x_{c,2}$} (o2);
        \fill
        (o1) circle (2.5pt) node[below] {$o_1$}
        (o2) circle (2.5pt) node[below] {$o_2$}
        (a) circle (2.5pt) node[below] {$a$}
        (b) circle (2.5pt) node[below] {$b$};
        \fill[blue!50!black]
        (c1) circle (2.5pt) node[left] {$c_1=\iota(v_i)$}
        (c2) circle(2.5pt) node[right] {$c_2=\iota(v_j)$};
      \end{tikzpicture}
    };
    \node[below,yshift=-4mm] at (type3.south) {(3) \texttt{extension\_type\_3}};
    \node (type4) at (7,-3.5)
    {%
      \begin{tikzpicture}[x={(1.5,0)},y={(0,1.5)}]
        \useasboundingbox (-1,0) rectangle (1,1);
        \coordinate (o1) at (0,0);
        \coordinate (a) at ($(o1)+(-1,0)$);
        \coordinate (b) at ($(o1)+(1,0)$);
        \coordinate (o2) at ($(o1)+(0,1)$);
        \coordinate (c1) at ($(o2)+(-1,0)$);
        \coordinate (c2) at ($(o2)+(1,0)$);
        \draw[very thick]
        (a) -- node[below] {$x_a$} (o1)
        (b) -- node[below] {$x_b$}  (o1);
        \draw[very thick, blue!50!black]
        (o1) -- node[left] {$x_o$} (o2)
        (c1) -- node[above] {$x_{c,1}$} (o2)
        (c2) -- node[above] {$x_{c,2}$} (o2);
        \fill
        (o1) circle (2.5pt) node[below] {$o_1$}
        (a) circle (2.5pt) node[below] {$a$}
        (b) circle (2.5pt) node[below] {$b$};
        \fill[blue!50!black]
        (o2) circle(2.5pt) node[above] {$o_2$}
        (c1) circle(2.5pt) node[left] {$c_1=\iota(v_i)$}
        (c2) circle(2.5pt) node[right] {$c_2=\iota(v_j)$};
      \end{tikzpicture}
    };
    \node[below,yshift=-4mm] at (type4.south) {(4) \texttt{extension\_type\_4}};
  \end{tikzpicture}\vspace{-3mm}
  \caption{The four extensions.}
  \label{fig:extensions}
\end{figure}

\begin{example}
  \label{ex:extension1}
  Let $G$ be the path graph on $6$ vertices with additional edges $\{3,5\}$ and $\{2,6\}$, let $\Gamma$ be a path graph on vertices $\{d,c,b,a\}$, and let $G\overset{\iota}{\dashrightarrow}\Gamma$ be the partial map:
  \begin{equation*}
    \iota\colon\quad G\dashrightarrow\Gamma,\quad 1\mapsto d,\;\; 2\mapsto c,\;\; 3,6\mapsto b,\;\; 5\mapsto a.
  \end{equation*}
  Consider \cref{alg:extension1} with $\{v_i,v_j\}=\{1,4\}$ and $\{a,b\}$ as defined.  Then System~\eqref{eq:systemForType1} equals:
  \begin{equation*}
    \left\{
      \begin{array}{l}
          x_a,x_b,x_c \geq 0, \\
          x_a + x_b = 1 \ (\text{Condition given by the length of the edge } \{\iota(3),\iota(5)\}),\\
          2 + x_a + x_c = 3 \ (\text{Condition to realize the distance between } 4 \text{ and } 1),\\
          x_a + x_c \leq 1, (\text{Condition given by distance between  } 4 \text{ and } 3),\\
          1 + x_a + x_c \leq 1 \ (\text{Condition given by distance between } 1 \text{ and } 2), \\
          x_a + x_c \leq 2 \ (\text{Condition given by distance between } 4\text{ and } 6), \text{ and }\\
          x_b + x_c \leq 1\ (\text{Condition given by distance between } 4 \text{ and } 2).
        \end{array}
    \right.
  \end{equation*}
  Solutions of this system are of the form $(x_a,x_b,x_c) = (t,1-t,1-t)$ for $t\in [0.5,1]$, see \cref{fig: house with garden first map}.  They all realize distance $d_G(1,4)$.  However, the solution $(x_a,x_b,x_c) = (0.5,0.5,0.5)$ also realizes distances $d_G(4,2), d_G(4,3), d_G(4,5)$, and it is the preferred solution in our implementation, see \cref{rem:extensionChoices} (1).
\end{example}

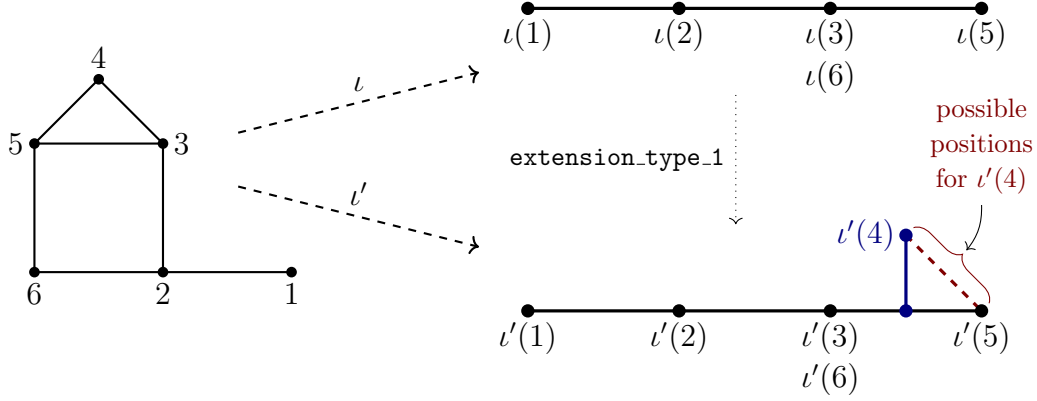
\begin{figure}[t]
  \centering
  \begin{tikzpicture}
    \node (house) at (-2,0)
    {
      \begin{tikzpicture}[x={(1.7,0)}, y={(0,1.7)}]
        \useasboundingbox (0,0) rectangle (1.5,1.75);
        \coordinate (v1) at (2,0);
        \coordinate (v2) at (1,0);
        \coordinate (v3) at (1,1);
        \coordinate (v4) at (0.5,1.5);
        \coordinate (v5) at (0,1);
        \coordinate (v6) at (0,0);
        \draw[thick]
        (v1) -- (v2) -- (v3) -- (v4) -- (v5) -- (v6)
        (v2) -- (v6)
        (v3) -- (v5);
        \fill
        (v1) circle (2pt)
        (v2) circle (2pt)
        (v3) circle (2pt)
        (v4) circle (2pt)
        (v5) circle (2pt)
        (v6) circle (2pt);
        \node[below] at (v1) {$1$};
        \node[below] at (v2) {$2$};
        \node[right] at (v3) {$3$};
        \node[above] at (v4) {$4$};
        \node[left] at (v5) {$5$};
        \node[below] at (v6) {$6$};
      \end{tikzpicture}
    };
    \node (iota) at (6,2)
    {
      \begin{tikzpicture}
        \useasboundingbox (-0.5,-2) rectangle (6,0);
        \coordinate (iota1) at (0,-1);
        \coordinate (iota2) at (2,-1);
        \coordinate (iota3) at (4,-1);
        \coordinate (iota6) at (4,-1);
        \coordinate (iota5) at (6,-1);
        \draw[very thick]
        (iota1) -- (iota2)
        (iota2) -- (iota3)
        (iota3) -- (iota5);
        \fill
        (iota1) circle(2.5pt) node[below] {$\iota(1)$}
        (iota2) circle(2.5pt) node[below] {$\iota(2)$}
        (iota3) circle(2.5pt) node[below] {$\iota(3)$} node[below,yshift=-5.5mm] {$\iota(6)$}
        (iota5) circle(2.5pt) node[below] {$\iota(5)$};
      \end{tikzpicture}
    };
    \draw[->, dashed, thick] (house) -- (iota) node[midway, above] {$\iota$};
    \node (iotaPrime) at (6,-2)
    {
      \begin{tikzpicture}
        \useasboundingbox (-0.5,-2) rectangle (6,0);
        \coordinate (iota1) at (0,-1);
        \coordinate (iota2) at (2,-1);
        \coordinate (iota3) at (4,-1);
        \coordinate (iota6) at (4,-1);
        \coordinate (iota5) at (6,-1);
        \coordinate (iota4) at (5,0);
        \coordinate (junction) at (5,-1);
        \draw[very thick]
        (iota1) -- (iota2)
        (iota2) -- (iota3)
        (iota3) --  (junction)
        (junction) -- (iota5);
        \draw[very thick, red!50!black, dashed]
        (iota4) -- (iota5);
        \draw[decorate,decoration={brace,amplitude=6pt},xshift=2mm,yshift=2mm,red!50!black]
        ($(iota4)+(0.1,0.1)$) -- ($(iota5)+(0.1,0.1)$) node[midway] (foo) {};
        \node[above,text width=22mm,align=center,font=\footnotesize,xshift=4mm,yshift=8mm,red!50!black] (fooText) at (foo) {possible positions for $\iota'(4)$};
        \draw[->] (fooText.south) to[bend left=20] ($(foo)+(0.2,0.2)$);
        \draw[very thick, blue!50!black]
        (junction) -- (iota4);
        \fill
        (iota1) circle(2.5pt) node[below] {$\iota'(1)$}
        (iota2) circle(2.5pt) node[below] {$\iota'(2)$}
        (iota3) circle(2.5pt) node[below] {$\iota'(3)$}  node[below,yshift=-5.5mm] {$\iota'(6)$}
        (iota5) circle(2.5pt) node[below] {$\iota'(5)$};
        \fill[blue!50!black]
        (junction) circle(2.5pt)
        (iota4) circle(2.5pt) node[left] {$\iota'(4)$};
      \end{tikzpicture}
    };
    \draw[->, dashed, thick] (house) -- (iotaPrime) node[midway, above] {$\iota'$};
    \draw[->, dotted] (iota) -- (iotaPrime) node[midway,left,font=\footnotesize] {\texttt{extension\_type\_1}};
  \end{tikzpicture}\vspace{-3mm}
  \caption{A $1$-lipschitz map with an extension of the first type.}
  \label{fig: house with garden first map}
\end{figure}

\begin{algorithm}[\texttt{extension\_type\_2}]\label{alg:extension2}\
  \begin{algorithmic}[1]
    \REQUIRE{$\left(G\overset{\iota}{\dashrightarrow}\Gamma, \{v_i,v_j\}\in \binom{V}{2}\right)$, where
      \begin{enumerate}
        \item $G\overset{\iota}{\dashrightarrow}\Gamma$ is a partial $1$-Lipschitz map into a metric tree $\Gamma$,
        \item $v_i\notin \domain(\iota)$,
        \item $v_j\notin \domain(\iota)$.
      \end{enumerate}}
    \ENSURE{$(p,G\overset{\iota'}{\dashrightarrow}\Gamma')$, where
      \begin{enumerate}
        \item $p\in\{\texttt{true},\texttt{false}\}$ indicates whether the extension was successful,
        \item if $p=\texttt{true}$, then $G\overset{\iota'}{\dashrightarrow}\Gamma'$ an extension realizing distance $d_G(v_i,v_j)$.
        \item if $p=\texttt{false}$, then $G\overset{\iota'}{\dashrightarrow}\Gamma'=G\overset{\iota}{\dashrightarrow}\Gamma$.
      \end{enumerate}
    }
    \FOR{$(\{a_1,b_1\},\{a_2,b_2\}) \in E(\Gamma)\times E(\Gamma)$ with $\{a_1,b_1\} \neq \{a_2,b_2\}$}
    \STATE Let $d = \min{\{d_\Gamma(a_1,a_2),d_\Gamma(b_1,b_2),d_\Gamma(b_1,a_2),d_\Gamma(a_1,b_2)\}}$.
    \STATE Relabel so that $\{a_1,b_1\}$ and $\{a_2,b_2\}$ so that 
    $d$ is $d_\Gamma(b_1,a_2)$ or $d_\Gamma(b_2,a_1)$.
    \STATE Consider the following linear system on the variables $x_{a,k},x_{b,k},x_{c,k}$, $k \in \{1,2\}$: \vspace{-1em}
    {
      \allowdisplaybreaks
      \setlength{\jot}{1pt}
      \begin{align}
        \notag & x_{a,k}\geq 0, x_{b,k}\geq 0, x_{c,k}\geq 0, x_{a,k}+x_{c,k} = d_{\Gamma}(a_k,b_k),\text{ for } k \in \{1,2\} \\
        \notag & d_\Gamma(a_2,b_1) + x_{b,1} + x_{a,2} + x_{c,1} + x_{c,2} = d_G(v_i,v_j)\text{ if } d = d_\Gamma(a_2,b_1) \neq 0\\
        \notag & d_\Gamma(b_2,a_1) + x_{b,2} + x_{a,1} + x_{c,1} + x_{c,2} = d_G(v_i,v_j) \text{ if } d = d_\Gamma(b_2,a_1) \neq 0\\
        \notag & d_\Gamma(\iota(v),a_1)) + x_{a,1} + x_{c,1} \leq d_G(v_i,v)\\
        \notag & \hspace{25mm}\text{for } v\in\domain(\iota) \text{ with } d_\Gamma(\iota(v),a_1)<d_\Gamma(\iota(v),b_1)\\
        \label{eq:systemForType2} & d_\Gamma(\iota(v),b_1)) + x_{b,1} + x_{c,1} \leq d_G(v_i,v)\\
        \notag & \hspace{25mm}\text{for } v\in\domain(\iota) \text{ with } d_\Gamma(\iota(v),a_1)>d_\Gamma(\iota(v),b_1)\\
        \notag & d_\Gamma(\iota(v),a_2)) + x_{a,2} + x_{c,2} \leq d_G(v_j,v)\\
        \notag & \hspace{25mm}\text{for } v\in\domain(\iota) \text{ with } d_\Gamma(\iota(v),a_2)<d_\Gamma(\iota(v),b_2)\\
        \notag & d_\Gamma(\iota(v),b_2)) + x_{b,2} + x_{c,2} \leq d_G(v_j,v)\\
        \notag & \hspace{25mm}\text{for } v\in\domain(\iota) \text{ with } d_\Gamma(\iota(v),a_2)>d_\Gamma(\iota(v),b_2)
      \end{align}
    }\vspace{-1em}
    \IF{System \eqref{eq:systemForType2} has a solution $(x_{a,1},x_{b,1},x_{c,1},x_{a,2},x_{b,2},x_{c,2})\in\RR^6_{\geq 0}$}
    \STATE Let $\Gamma'$ be $\Gamma$ but with but with the following added (see \cref{fig:extensions} (2)):
    \begin{enumerate}
      \item \hspace{-10mm} a vertex $o_k$ on edge $\{a_k,b_k\}$ at distance $x_{a,k}$ from $a_k$ and $x_{b,k}$ from $b_k$ for $k \in \{1,2\}$,
      \item \hspace{-10mm} a vertex $c_k$ for $k \in \{1,2\}$
      \item \hspace{-10mm} an edge $\{c_k,o_k\}$ of length $x_{c,k}$ for $k \in \{1,2\}$.
    \end{enumerate}
    \STATE Let $G\overset{\iota'}{\dashrightarrow}\Gamma'$ be $G\overset{\iota}{\dashrightarrow}\Gamma$ with $v_i$ mapped to $c_1$ and $v_j$ to $c_2$.
    \RETURN{$(\texttt{true}, G\overset{\iota'}{\dashrightarrow}\Gamma')$}
    \ENDIF
    \ENDFOR
    \RETURN{$(\texttt{false}, G\overset{\iota}{\dashrightarrow}\Gamma)$}
  \end{algorithmic}
\end{algorithm}

\begin{algorithm}[\texttt{extension\_type\_3}]\label{alg:extension3}\
  \begin{algorithmic}[1]
    \REQUIRE{$\left(G\overset{\iota}{\dashrightarrow}\Gamma, \{v_i,v_j\}\in \binom{V}{2}\right)$, where
      \begin{enumerate}
        \item $G\overset{\iota}{\dashrightarrow}\Gamma$ is a partial $1$-Lipschitz map into a metric tree $\Gamma$,
        \item $v_i\notin \domain(\iota)$,
        \item $v_j\notin \domain(\iota)$.
      \end{enumerate}}
    \ENSURE{$(p,G\overset{\iota'}{\dashrightarrow}\Gamma')$, where
      \begin{enumerate}
        \item $p\in\{\texttt{true},\texttt{false}\}$ indicates whether the extension was successful,
        \item if $p=\texttt{true}$, then $G\overset{\iota'}{\dashrightarrow}\Gamma'$ an extension realizing distance $d_G(v_i,v_j)$.
        \item if $p=\texttt{false}$, then $G\overset{\iota'}{\dashrightarrow}\Gamma'=G\overset{\iota}{\dashrightarrow}\Gamma$.
      \end{enumerate}
    }
    \FOR{$\{a,b\} \in E(\Gamma)$}
    \STATE Consider the following linear system on the variables $x_{a},x_{b},x_{o},x_{c,1},x_{c,2}$:
    {
      \allowdisplaybreaks
      \setlength{\jot}{1pt}
      \begin{align}
        \notag & x_{a}\geq 0, x_{b}\geq 0, x_{o},x_{c,1},x_{c,2}\geq 0, x_{a}+x_{b} = d_{\Gamma}(a,b) \\
        \notag & x_{c,1} + x_{c,2} = d_G(v_i,v_j) \\
        \notag & d_\Gamma(a,\iota(v)) + x_a + x_o + x_{c,1} = d_G(v,v_i) \\
        \notag & \hspace{25mm}\text{ for } v \in \domain(\iota) \text{ with } d_\Gamma(\iota(v),a) < d_\Gamma(\iota(v),b)\\
        \label{eq:systemForType3} & d_\Gamma(a,\iota(v)) + x_a + x_o + x_{c,2} = d_G(v,v_j) \\
        \notag & \hspace{25mm}\text{ for } v \in \domain(\iota)\text{ with }d_\Gamma(\iota(v),a) < d_\Gamma(\iota(v),b)\\
        \notag & d_\Gamma(b,\iota(v)) + x_b + x_o + x_{c,1} = d_G(v,v_i) \\
        \notag & \hspace{25mm}\text{ for } v \in \domain(\iota) \text{ with }d_\Gamma(\iota(v),a) > d_\Gamma(\iota(v),b)\\
        \notag & d_\Gamma(b,\iota(v)) + x_b + x_o + x_{c,2} = d_G(v,v_j) \\
        \notag & \hspace{25mm}\text{ for } v \in \domain(\iota)\text{ with }d_\Gamma(\iota(v),a) > d_\Gamma(\iota(v),b)
      \end{align}%
    }\vspace{-1em}
    \IF{System \eqref{eq:systemForType3} has a solution $(x_{a},x_{b},x_{o},x_{c,1},x_{c,2})\in\RR^5_{\geq 0}$}
    \STATE Let $\Gamma'$ be $\Gamma$ but with the following added (see \cref{fig:extensions} (3)):
    \begin{enumerate}
      \item \hspace{-10mm} a vertex $o_1$ on edge $\{a,b\}$ at distance $x_a$ from $a$ and $x_b$ from $b$,
      \item \hspace{-10mm} three vertices $o_2$ and $c_1$ and $c_2$,
      \item \hspace{-10mm} an edge $\{o_1,o_2\}$ of length $x_o$,
      \item \hspace{-10mm} an edge $\{c_1,o_2\}$ of length $x_{c,1}$, and
      \item \hspace{-10mm} an edge $\{c_2,o_2\}$ of length $x_{c,2}$.
    \end{enumerate}
    \STATE Let $G\overset{\iota'}{\dashrightarrow}\Gamma'$ be $G\overset{\iota}{\dashrightarrow}\Gamma$ with $v_i$ mapped to $c_1$ and $v_j$ to $c_2$.
    \RETURN{$(\texttt{true}, G\overset{\iota'}{\dashrightarrow}\Gamma')$}
    \ENDIF
    \ENDFOR
    \RETURN{$(\texttt{false}, G\overset{\iota}{\dashrightarrow}\Gamma)$}
  \end{algorithmic}
\end{algorithm}

And the final algorithm which describes the extension obtained by adding a $T$ shape at one of the edges.

\begin{algorithm}[\texttt{extension\_type\_4}]\label{alg:extension4}\
  \begin{algorithmic}[1]
    \REQUIRE{$\left(G\overset{\iota}{\dashrightarrow}\Gamma, \{v_i,v_j\}\in \binom{V}{2}\right)$, where
      \begin{enumerate}
        \item $G\overset{\iota}{\dashrightarrow}\Gamma$ is a partial $1$-Lipschitz map into a metric tree $\Gamma$,
        \item $v_i\notin \domain(\iota)$,
        \item $v_j\notin \domain(\iota)$.
      \end{enumerate}}
    \ENSURE{$(p,G\overset{\iota'}{\dashrightarrow}\Gamma')$, where
      \begin{enumerate}
        \item $p\in\{\texttt{true},\texttt{false}\}$ indicates whether the extension was successful,
        \item if $p=\texttt{true}$, then $G\overset{\iota'}{\dashrightarrow}\Gamma'$ an extension realizing distance $d_G(v_i,v_j)$.
        \item if $p=\texttt{false}$, then $G\overset{\iota'}{\dashrightarrow}\Gamma'=G\overset{\iota}{\dashrightarrow}\Gamma$.
      \end{enumerate}
    }
    \FOR{$\{a,b\} \in E(\Gamma)$}
    \STATE Consider the following linear system on the variables $x_{a},x_{b},x_{o},x_{c,1},x_{c,2}$:
    {
      \allowdisplaybreaks
      \setlength{\jot}{1pt}
      \begin{align}
        \notag & x_{a}\geq 0, x_{b}\geq 0, x_{o},x_{c,1},x_{c,2}\geq 0, x_{a}+x_{o}+x_{b} = d_{\Gamma}(a,b) \\
        \notag & x_{c,1} + x_{o} + x_{c,2} = d_G(v_i,v_j) \\
        \notag & d_\Gamma(a,\iota(v)) + x_a + x_{c,1} = d_G(v,v_i) \\
        \notag & \hspace{25mm}\text{ for } v \in \domain(\iota) \text{ with } d_\Gamma(\iota(v),a) < d_\Gamma(\iota(v),b)\\
        \notag & d_\Gamma(a,\iota(v)) + x_a + x_o + x_{c,2} = d_G(v,v_j) \\
        \label{eq:systemForType4} & \hspace{25mm}\text{ for } v \in \domain(\iota)\text{ with }d_\Gamma(\iota(v),a) < d_\Gamma(\iota(v),b)\\
        \notag & d_\Gamma(b,\iota(v)) + x_b + x_o + x_{c,1} = d_G(v,v_i) \\
        \notag & \hspace{25mm}\text{ for } v \in \domain(\iota) \text{ with }d_\Gamma(\iota(v),a) > d_\Gamma(\iota(v),b)\\
        \notag & d_\Gamma(b,\iota(v)) + x_b + x_{c,2} = d_G(v,v_j) \\
        \notag & \hspace{25mm}\text{ for } v \in \domain(\iota)\text{ with }d_\Gamma(\iota(v),a) > d_\Gamma(\iota(v),b)
      \end{align}%
    }\vspace{-1em}
    \IF{System \eqref{eq:systemForType4} has a solution $(x_{a},x_{b},x_{o},x_{c,1},x_{c,2})\in\RR^5_{\geq 0}$}
    \STATE Let $\Gamma'$ be $\Gamma$ but with the following added (see \cref{fig:extensions} (4)):
    \begin{enumerate}
      \item \hspace{-10mm} two vertices $o_1$ and $o_2$ at the edge $\{a,b\}$ distance $x_o$ apart, with $o_1$ at distance $x_a$ from $a$ and $o_2$ at distance $x_b$ from $b$.
      \item \hspace{-10mm} two vertices $c_1$ and $c_2$.
      \item \hspace{-10mm} two edges $\{c_1,o_1\}$ and $\{c_2,o_2\}$ of length $x_{c,1}$ and $x_{c,2}$, respectively.
    \end{enumerate}
    \STATE Let $G\overset{\iota'}{\dashrightarrow}\Gamma'$ be $G\overset{\iota}{\dashrightarrow}\Gamma$ and $v_i$ mapped to $c_1$ and $v_j$ to $c_2$.
    \RETURN{$(\texttt{true}, G\overset{\iota'}{\dashrightarrow}\Gamma')$}
    \ENDIF
    \ENDFOR
    \RETURN{$(\texttt{false}, G\overset{\iota}{\dashrightarrow}\Gamma)$}
  \end{algorithmic}
\end{algorithm}

We can now state the greedy algorithm for constructing an isometry $G\overset{\iota}{\rightarrow}\Gamma\coloneqq \Gamma_1\times\dots\times\Gamma_r$.  In addition to Algorithms \ref{alg:extension1} to \ref{alg:extension4}, the greedy algorithm further relies on the following functions:
\begin{description}
  \item[$\texttt{pick\_unrealized\_distance}(G\overset{\iota}{\rightarrow}\Gamma)$] Return a pair of vertices $\{v,w\}\in \binom{[n]}{2}$ with $d_G(v,w)>d_\Gamma(\iota(v),\iota(w))$.
  \item[$\texttt{initialize\_new\_coordinate}(G,\{v,w\})$] Return a partial map $G\overset{\iota_{r+1}}{\dashrightarrow}\Gamma_{r+1}$ with $v,w\in\mathrm{domain}(\iota_{r+1})$ and $d_G(v,w)=d_\Gamma(\iota(v),\iota(w))$.
  \item[$\texttt{extension\_final}(G\overset{\iota}{\dashrightarrow}\Gamma)$] Return a $1$-Lipschitz map $G\rightarrow\Gamma$ extending $\iota$.
\end{description}

\begin{algorithm}[\texttt{greedyAlgorithm}]\label{alg:greedyAlgorithm}\
  \begin{algorithmic}[1]
    \REQUIRE{$G$, a simple connected graph with $n$ vertices.}
    \ENSURE{$G\overset{\iota}{\rightarrow}\Gamma_1\times\dots\times\Gamma_r$, an isometry into a product of metric trees.}
    \STATE{Initialize $\iota\colon G \dashrightarrow \Gamma_1\times\dots\times\Gamma_r$ with $r\coloneqq 0$, a partial map with empty domain and codomain.  From hereon, $G\overset{\iota_k}{\dashrightarrow}\Gamma_k$ denotes the $k$-th coordinate of $\iota$.}
    \WHILE{$G\overset{\iota}{\dashrightarrow}\Gamma_1\times\dots\times\Gamma_r$ no isometry}
    \STATE{$\{v,w\} \coloneqq \texttt{pick\_unrealised\_distance}(G\overset{\iota}{\dashrightarrow}\Gamma_1\times\dots\times\Gamma_r)$}
    \FOR{$k\in\{1,\dots r\}$ and $\texttt{p}\in\{1,2,3,4\}$}
    \STATE{$(\texttt{extensionSuccessful},G\overset{\iota_k}{\dashrightarrow}\Gamma_k) \coloneqq \texttt{extension\_type\_p}(G\overset{\iota_\ell}{\dashrightarrow}\Gamma_k, \{v,w\})$}\vspace{-\baselineskip}
    \IF{$\texttt{extensionSuccessful}=\texttt{true}$}
    \STATE Go to Step 2.
    \ENDIF
    \ENDFOR
    \STATE $G\overset{\iota_{r+1}}{\dashrightarrow}\Gamma_{r+1}\coloneqq\texttt{initialize\_new\_coordinate}(G,\{v,w\})$
    \STATE $r\coloneqq r+1$
    \ENDWHILE
    \RETURN{$\texttt{extension\_final}(G\overset{\iota}{\dashrightarrow}\Gamma_1\times\dots\times\Gamma_r)$}
  \end{algorithmic}
\end{algorithm}

We close this section with a remark on our implementation and a proof that, if \cref{alg:greedyAlgorithm} returns an embedding into a product of one or two trees, then said embedding is minimal.

\begin{remark}\label{rem:extensionChoices}
  Note that all algorithms have significant room for optimization, and the most sensible target for optimization is to minimize the number of \textbf{while} loop iterations in \cref{alg:greedyAlgorithm}.  To achieve that, we made the following design decisions in our implementation \cite{ACGJQR2026github}:
  \begin{enumerate}
    \item Algorithms \ref{alg:extension1}, \ref{alg:extension2}, \ref{alg:extension3}, and \ref{alg:extension4} require solving a linear system whose solution may not be unique.  We pick the solution that will lead to $G\overset{\iota_k}{\dashrightarrow} \Gamma_k$ realizing the maximal number of previously unrealized distances in the sense of \cref{def:greedyExtension}.  These solutions are found among the vertices of the feasible region.
    \item In \cref{alg:greedyAlgorithm} Step 3, there usually are many possible pairs of vertices.  We pick a pair $\{v,w\}\in \binom{[n]}{2}$ such that $d_G(v,w)$ is maximal among the unrealized distances.  This is because any $G\overset{\iota_k}{\dashrightarrow} \Gamma_k$ realizing $d_G(v,w)$ will also realize $d_G(v',w')$ for $v',w'$ on any distance realizing path connecting $v$ and $w$, see \cref{lem:geodesicsIsometry}.
  \end{enumerate}
\end{remark}

\begin{theorem}\label{thm: greedy alg correct for rank 1/2}
  Let $G\overset\iota\hookrightarrow \Gamma_1\times\dots\times\Gamma_r$ be the output of \cref{alg:greedyAlgorithm}.  Then we have
  \begin{equation*}
    r=1\qquad\Longleftrightarrow\qquad \phylogeneticrank(G)=1.\phantom{.}
  \end{equation*}
  In particular, if $r=2$ then $\phylogeneticrank(G)=2$.
\end{theorem}
\begin{proof}
  The $\Rightarrow$ implication of the first equivalence is clear. For the $\Leftarrow$ implication, suppose that $\phylogeneticrank(G) = 1$ and let $p\colon G \hookrightarrow \Gamma$ be an embedding where $\Gamma$ is as small as possible. By \cite{buneman1974note}, the tree $\Gamma$ is unique. In particular, the restriction $p|_{V'}$ to any subset of vertices $V' \subseteq V(G)$ is the unique such tree embedding. By induction on $|V'|$, it is straightforward to verify that Algorithm~\ref{alg:greedyAlgorithm}, at each step, recovers the restriction of $p$ to a subset of vertices. Explicitly, assume, during the algorithm, we have embedded $V'$ isometrically into a tree $\iota\colon V' \dashrightarrow \Gamma'$. Recall from Remark~\ref{rem:extensionChoices} that the algorithm chooses the position of a new vertex $v \in V(G) \setminus V'$ in a greedy way to maximise the number of distances in $G$ attained in the partial map. By the uniqueness of the embedding $p|_{V' \cup \{v\}}$, there exists a unique extension of $\iota$ that achieves all distances between $v$ and $w \in V'$. This extension is precisely the one given by the algorithm. So we have shown that algorithm recovers the embedding $p\colon G \rightarrow \Gamma$, so we have $r = 1$.

  In particular, if $r = 2$ then we immediately have $\phylogeneticrank(G) \le 2$. By the first part, if $\phylogeneticrank(G) = 1$, then $r = 1$, so it follows that $\phylogeneticrank(G) = 2$. This concludes the proof.
\end{proof}

%%% Local Variables:
%%% mode: LaTeX
%%% TeX-master: "graph_embedding_paper"
%%% End:

\section{An exact algorithm}\label{sec:exactAlgorithm}

In this section, we present an exact algorithm for computing all minimal embeddings.  It was used to compute the phylogenetic ranks of most graphs on 6 and 7 vertices, as well as selected counterexamples.

\subsection{The overall algorithm}

We begin by defining what we mean by a geodesic, and how $1$-Lipschitz maps being isometries on them depends only on their endpoints.

\begin{definition}
  \label{def:geodesic}
  A \emph{geodesic} is a finite sequence of vertices $\gamma=(v_0,\dots,v_k)\in V^{k+1}$ such that $\{v_{i-1},v_{i}\}\in E$ for $i=1,\dots,k$ and $d(v_0,v_k)=k$.
  We say that the vertices $v_0,\dots,v_k$ \emph{lie on} the geodesic $\gamma$ and that $\gamma$ \emph{contains} $v_0,\dots,v_k$.  We say that $v_0, v_k$ are the \emph{endpoints} of $\gamma$ and that $\gamma$ \emph{connects} $v_0, v_k$.

  A geodesic $\gamma\in V^{k+1}$ is \emph{maximal}, if there is no longer geodesic $\gamma'\in V^{k'+1}$, $k'>k$, such that $\gamma$ is a contiguous subsequence of $\gamma'$.  We use $\mathcal G$ to denote the set of geodesics of $G$, and $\mathcal G_{\max}$ to denote the set of maximal geodesics of $G$.
\end{definition}

\begin{lemma}
  \label{lem:geodesicsAlternativeDefinition}
  Let $\gamma=(v_0,\dots,v_k)\in \mathcal G$ be a geodesic. Then $d_G(v_{l_1},v_{l_2})=l_2-l_1$ for all $0\leq l_1< l_2\leq k$.
\end{lemma}
\begin{proof}
  Follows from a combination of $d_G(v_0,v_k)=k$, $d_G(v_{i-1},v_i)=1$ for $i=1,\dots,k$, and the triangle inequalities.
\end{proof}

\begin{lemma}
  \label{lem:geodesicsIsometry}
Let $\iota\colon G\rightarrow\Gamma$ be a $1$-Lipschitz map.  Then for any geodesic $\gamma=(v_0,\dots,v_k)\in \mathcal G$ we have
  \begin{equation*}
    \iota|_\gamma\text{ is an isometry}\quad\Longleftrightarrow\quad d_G(v_0,v_n) = d_{\Gamma}(\iota(v_0),\iota(v_n)).
  \end{equation*}
  In particular, if $\iota_1\times\dots\times\iota_r\colon G\rightarrow\Gamma_1\times\dots\times\Gamma_r$ is an isometric embedding.  Then for any  geodesic $\gamma$, there is some $\Gamma_i$ such that $\iota_i|_{\gamma}$ is an isometry.
\end{lemma}
\begin{proof}
  The forwards direction is trivial, if $\iota|_\gamma$ is an isometry, then it must preserve the full distance by definition. For the reverse direction, we pick a vertex $v_j$ in the geodesic. Then
  \begin{align*}
    &d_G(v_0,v_j)+d_G(v_j,v_n)\overset{(1)}{=}d_G(v_0,v_n) \overset{(2)}{=} d_{\Gamma}(\iota(v_0),\iota(v_n))\\
    &\qquad \overset{(3)}{\leq} d_{\Gamma}(\iota(v_0),\iota(v_j))+d_{\Gamma}(\iota(v_j),\iota(v_n)) \overset{(4)}{\leq} d_G(v_0,v_j) + d_G(v_j,v_n),
  \end{align*}
  where Equality (1) is because of \cref{lem:geodesicsAlternativeDefinition}, Equality (2) is the given equality, Inequality (3) is the triangle inequality, and Inequality (4) follows from $\iota$ being $1$-Lipschitz.

  Since the equation above begins and ends with $d_G(v_0,v_j) + d_G(v_j,v_n)$, all inequalities must be equalities. In particular, Equality (4) together with the $1$-Lipschitz property implies that
  $$d_G (v_0,v_j)= d_{\Gamma}(\iota(v_0),\iota(v_j)), \quad d_G (v_j,v_n)= d_{\Gamma}(\iota(v_j),\iota(v_n)). $$

  This argument shows that $\iota$ is an isometry on sub-paths which have an endpoint at one of the endpoints of the original geodesic. Since we now know that $\iota$ is an isometry on sub-paths of this type, we can use the same argument again in order to conclude that the distance is preserved for any sub-path within the geodesic - which means $\iota|_\gamma$ is an isometry.

  To prove the following statement: pick a geodesic $\gamma$, note that if the map is an isometry, there must be some $\iota_i$ which realises the full distance between its endpoints. By the above argument, $\iota_i|_\gamma$ must be an isometry.
\end{proof}

Next, we introduce laps, and show that they are closely related to maximal geodesics.

\begin{definition}
  \label{def:locallyAntipodalPair}
  A \emph{locally antipodal pair} or \emph{lap} is a pair of vertices $\{a,b\}\in\binom{V}{2}$ such that
    \begin{align*}
      & d_G(a,b)\geq d_G(a',b) \qquad\text{for all } a'\in \mathrm{Neighbours}(a), \quad \text{and}\\
      & d_G(a,b)\geq d_G(a,b') \qquad\text{for all } b'\in \mathrm{Neighbours}(b),
    \end{align*}
    where $\mathrm{Neighbours}(\cdot)$ denotes all adjacent vertices.  In words, $d(a,b)$ cannot be increased by replacing either $a$ or $b$ with one of its neighbours.  We denote the set of laps by $\mathcal A\subseteq\binom{V}{2}$.  We say $v\in V$ \emph{lies on} $\{a,b\}\in\mathcal A$ or $\{a,b\}$ \emph{contains} $v$, if $v$ lies on a geodesic connecting $a$ and $b$.
\end{definition}

\begin{lemma}
  \label{lem:lapsAndMaximalGeodesics}
  For any two vertices $a,b\in V$ we have:
  \begin{align*}
    &\{a,b\}\in \mathcal A\text{ is a lap}\quad\Longleftrightarrow\quad a,b \text{ are endpoints of a maximal geodesic}.
  \end{align*}
\end{lemma}
\begin{proof}
  Let $\gamma=(a,v_1,\dots,v_{k-1},b)\in V^{k+1}$ be a geodesic connecting $a$ and $b$.  We now show that $\{a,b\}\in\mathcal A$ is a lap if and only if $\gamma$ is maximal..

  For the ``$\Rightarrow$'' implication, suppose that $\gamma$ is not maximal.  Then there must be some $a'\in\mathrm{Neighbours}(a)$ with $d_G(a',b)=d_G(a,b)+1$ or some $b'\in\mathrm{Neighbours}(b)$ with $d_G(a,b')=d_G(a,b)+1$.  In both cases $\{a,b\}$ is no lap.

  For the ``$\Leftarrow$'' implication, suppose $\{a,b\}$ is no lap.  Then without loss of generality there is some $a'\in\mathrm{Neighbours}(a)$ with $d_G(a',b)>d_G(a,b)$.  By the triangle inequality, we have $d(a',b)\leq d(a',a)+d(a,b)=d(a,b)+1=k+1$.  Combining both, we see that $d(a',b)=k+1$ and that $\gamma$ can be extended to $(a',a,v_1,\dots,v_{k-1},b)$.
\end{proof}

\begin{corollary}
  \label{cor:lapsAndMaximalGeodesics}
  For any two vertices $u,v\in V$ there is a lap $\{a,b\}\in\mathcal A$ containing them.
\end{corollary}

\begin{example}
  \label{ex:laps}\
  \begin{enumerate}
    \item Let $G$ be the prism graph from \cref{example: prism}.  It has $6$ laps, which are illustrated in \cref{fig:lapsPrism}: $\{1,5\}$, $\{2,4\}$, $\{1,6\}$, $\{2,6\}$, $\{3,4\}$ and $\{3,5\}$.
    \item Let $G$ be the path graph on $6$ vertices with additional edges $\{3,5\}$ and $\{2,6\}$.  It has $4$ laps which are illustated in \cref{fig:lapsHouseAndGarden}: $\{1,4\}$, $\{1,5\}$, $\{3,6\}$ and $\{4,6\}$.
  \end{enumerate}
\end{example}

\begin{figure}
  \centering
  \begin{tikzpicture}
    \node (prism) at (0,-2.5)
    {%
      \begin{tikzpicture}
        \coordinate (v1) at (1,1.0);
        \coordinate (v2) at (1,-1,0);
        \coordinate (v3) at (2.2,0);
        \coordinate (v4) at (4.5,1.0);
        \coordinate (v5) at (4.5,-1.0);
        \coordinate (v6) at (3.3,0);
        \fill
        (v1) circle (2pt)
        (v2) circle (2pt)
        (v3) circle (2pt)
        (v4) circle (2pt)
        (v5) circle (2pt)
        (v6) circle (2pt);
        \draw[thick]
        (v1) -- (v2) -- (v3) -- (v1)
        (v4) -- (v5) -- (v6) -- (v4)
        (v1) -- (v4)
        (v2) -- (v5)
        (v3) -- (v6);
        \node[left] at (v1) {$1$};
        \node[left] at (v2) {$2$};
        \node[left,xshift=-1mm] at (v3) {$3$};
        \node[right] at (v4) {$4$};
        \node[right] at (v5) {$5$};
        \node[right,xshift=1mm] at (v6) {$6$};
      \end{tikzpicture}
    };
    \node (lap15) at (5,0)
    {%
      \begin{tikzpicture}% [x={(0.7,0)}, y={(0,0.7)}]
        \coordinate (v1) at (1,1.0);
        \coordinate (v2) at (1,-1,0);
        \coordinate (v3) at (2.2,0);
        \coordinate (v4) at (4.5,1.0);
        \coordinate (v5) at (4.5,-1.0);
        \coordinate (v6) at (3.3,0);
        \draw[orange,thick]
        (v1) -- (v4) -- (v5)
        (v1) -- (v2) -- (v5);
        \fill
        (v1) circle (2pt)
        (v5) circle (2pt);
        \fill[orange]
        (v2) circle (2pt)
        (v4) circle (2pt);
        \node[left] at (v1) {$1$};
        \node[left] at (v2) {$2$};
        \node[right] at (v4) {$4$};
        \node[right] at (v5) {$5$};
        \node at (2.75,0) {$\{1,5\}$};
      \end{tikzpicture}
    };
    \node (lap24) at (10,0)
    {%
      \begin{tikzpicture}% [x={(0.7,0)}, y={(0,0.7)}]
        \coordinate (v1) at (1,1.0);
        \coordinate (v2) at (1,-1,0);
        \coordinate (v3) at (2.2,0);
        \coordinate (v4) at (4.5,1.0);
        \coordinate (v5) at (4.5,-1.0);
        \coordinate (v6) at (3.3,0);
        \draw[orange,thick]
        (v1) -- (v4) -- (v5)
        (v1) -- (v2) -- (v5);
        \fill
        (v2) circle (2pt)
        (v4) circle (2pt);
        \fill[orange]
        (v1) circle (2pt)
        (v5) circle (2pt);
        \node[left] at (v1) {$1$};
        \node[left] at (v2) {$2$};
        \node[right] at (v4) {$4$};
        \node[right] at (v5) {$5$};
        \node at (2.75,0) {$\{2,4\}$};
      \end{tikzpicture}
    };
    \node (lap16) at (5,-3)
    {%
      \begin{tikzpicture}% [x={(0.7,0)}, y={(0,0.7)}]
        \coordinate (v1) at (1,1.0);
        \coordinate (v2) at (1,-1,0);
        \coordinate (v3) at (2.2,0);
        \coordinate (v4) at (4.5,1.0);
        \coordinate (v5) at (4.5,-1.0);
        \coordinate (v6) at (3.3,0);
        \draw[orange,thick]
        (v1) -- (v3) -- (v6)
        (v1) -- (v4) -- (v6);
        \fill
        (v1) circle (2pt)
        (v6) circle (2pt);
        \fill[orange]
        (v3) circle (2pt)
        (v4) circle (2pt);
        \node[left] at (v1) {$1$};
        \node[left,xshift=-2mm] at (v3) {$3$};
        \node[right] at (v4) {$4$};
        \node[right,xshift=2mm] at (v6) {$6$};
        \node at (2.75,0.5) {$\{1,6\}$};
      \end{tikzpicture}
    };
    \node (lap26) at (5,-5)
    {%
      \begin{tikzpicture}% [x={(0.7,0)}, y={(0,0.7)}]
        \coordinate (v1) at (1,1.0);
        \coordinate (v2) at (1,-1,0);
        \coordinate (v3) at (2.2,0);
        \coordinate (v4) at (4.5,1.0);
        \coordinate (v5) at (4.5,-1.0);
        \coordinate (v6) at (3.3,0);
        \draw[orange,thick]
        (v2) -- (v3) -- (v6)
        (v2) -- (v5) -- (v6);
        \fill
        (v2) circle (2pt)
        (v6) circle (2pt);
        \fill[orange]
        (v3) circle (2pt)
        (v5) circle (2pt);
        \node[left] at (v2) {$2$};
        \node[left,xshift=-2mm] at (v3) {$3$};
        \node[right] at (v5) {$5$};
        \node[right,xshift=2mm] at (v6) {$6$};
        \node at (2.75,-0.5) {$\{2,6\}$};
      \end{tikzpicture}
    };
    \node (lap34) at (10,-3)
    {%
      \begin{tikzpicture}% [x={(0.7,0)}, y={(0,0.7)}]
        \coordinate (v1) at (1,1.0);
        \coordinate (v2) at (1,-1,0);
        \coordinate (v3) at (2.2,0);
        \coordinate (v4) at (4.5,1.0);
        \coordinate (v5) at (4.5,-1.0);
        \coordinate (v6) at (3.3,0);
        \draw[orange,thick]
        (v1) -- (v3) -- (v6)
        (v1) -- (v4) -- (v6);
        \fill
        (v3) circle (2pt)
        (v4) circle (2pt);
        \fill[orange]
        (v1) circle (2pt)
        (v6) circle (2pt);
        \node[left] at (v1) {$1$};
        \node[left,xshift=-2mm] at (v3) {$3$};
        \node[right] at (v4) {$4$};
        \node[right,xshift=2mm] at (v6) {$6$};
        \node at (2.75,0.5) {$\{3,4\}$};
      \end{tikzpicture}
    };
    \node (lap35) at (10,-5)
    {%
      \begin{tikzpicture}% [x={(0.7,0)}, y={(0,0.7)}]
        \coordinate (v1) at (1,1.0);
        \coordinate (v2) at (1,-1,0);
        \coordinate (v3) at (2.2,0);
        \coordinate (v4) at (4.5,1.0);
        \coordinate (v5) at (4.5,-1.0);
        \coordinate (v6) at (3.3,0);
        \draw[orange,thick]
        (v2) -- (v3) -- (v6)
        (v2) -- (v5) -- (v6);
        \fill
        (v3) circle (2pt)
        (v5) circle (2pt);
        \fill[orange]
        (v2) circle (2pt)
        (v6) circle (2pt);
        \node[left] at (v2) {$2$};
        \node[left,xshift=-2mm] at (v3) {$3$};
        \node[right] at (v5) {$5$};
        \node[right,xshift=2mm] at (v6) {$6$};
        \node at (2.75,-0.5) {$\{3,5\}$};
      \end{tikzpicture}
    };
  \end{tikzpicture}\vspace{-3mm}
  \caption{The laps of the prism graph from \cref{ex:laps} (1).}
  \label{fig:lapsPrism}
\end{figure}
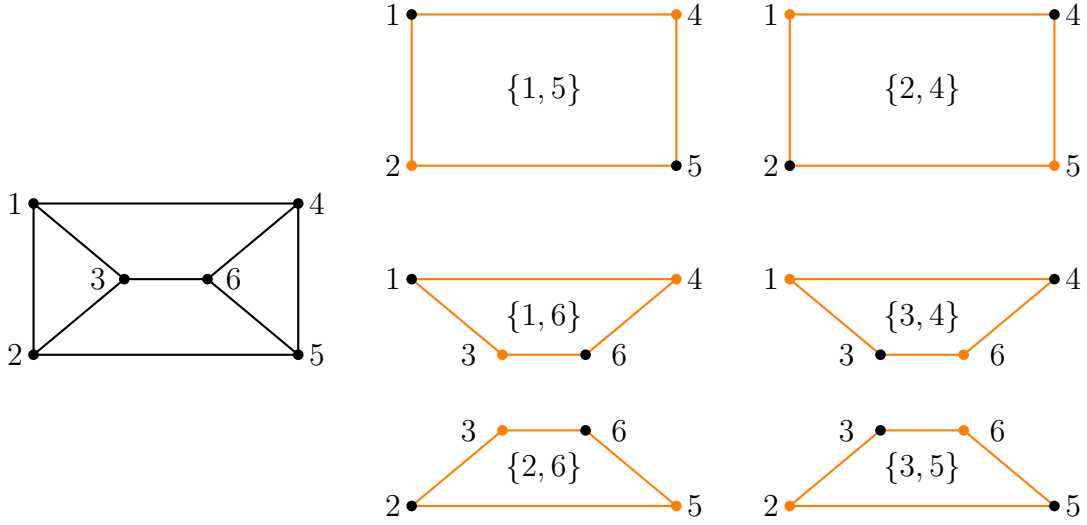

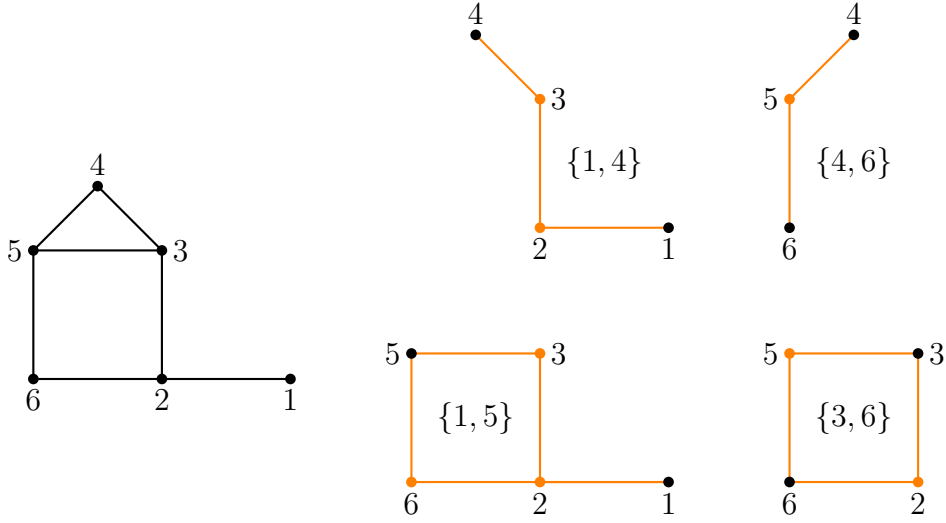
\begin{figure}[t]
  \centering
  \begin{tikzpicture}
    \node at (0,-2)
    {%
      \begin{tikzpicture}[x={(1.7,0)}, y={(0,1.7)}]
        \useasboundingbox (0,0) rectangle (2,1.75);
        \coordinate (v1) at (2,0);
        \coordinate (v2) at (1,0);
        \coordinate (v3) at (1,1);
        \coordinate (v4) at (0.5,1.5);
        \coordinate (v5) at (0,1);
        \coordinate (v6) at (0,0);
        \draw[thick]
        (v1) -- (v2) -- (v3) -- (v4) -- (v5) -- (v6)
        (v2) -- (v6)
        (v3) -- (v5);
        \fill
        (v1) circle (2pt)
        (v2) circle (2pt)
        (v3) circle (2pt)
        (v4) circle (2pt)
        (v5) circle (2pt)
        (v6) circle (2pt);
        \node[below] at (v1) {$1$};
        \node[below] at (v2) {$2$};
        \node[right] at (v3) {$3$};
        \node[above] at (v4) {$4$};
        \node[left] at (v5) {$5$};
        \node[below] at (v6) {$6$};
      \end{tikzpicture}
    };
    \node at (5,0)
    {%
      \begin{tikzpicture}[x={(1.7,0)}, y={(0,1.7)}]
        \useasboundingbox (0,0) rectangle (2,1.75);
        \coordinate (v1) at (2,0);
        \coordinate (v2) at (1,0);
        \coordinate (v3) at (1,1);
        \coordinate (v4) at (0.5,1.5);
        \coordinate (v5) at (0,1);
        \coordinate (v6) at (0,0);
        \draw[orange,thick]
        (v1) -- (v2) -- (v3) -- (v4);
        \fill
        (v1) circle (2pt)
        (v4) circle (2pt);
        \fill[orange]
        (v2) circle (2pt)
        (v3) circle (2pt);
        \node[below] at (v1) {$1$};
        \node[below] at (v2) {$2$};
        \node[right] at (v3) {$3$};
        \node[above] at (v4) {$4$};
        \node at (1.5,0.5) {$\{1,4\}$};
      \end{tikzpicture}
    };
    \node at (10,0)
    {%
      \begin{tikzpicture}[x={(1.7,0)}, y={(0,1.7)}]
        \useasboundingbox (0,0) rectangle (2,1.75);
        \coordinate (v1) at (2,0);
        \coordinate (v2) at (1,0);
        \coordinate (v3) at (1,1);
        \coordinate (v4) at (0.5,1.5);
        \coordinate (v5) at (0,1);
        \coordinate (v6) at (0,0);
        \draw[orange,thick]
        (v4) -- (v5) -- (v6);
        \fill
        (v4) circle (2pt)
        (v6) circle (2pt);
        \fill[orange]
        (v5) circle (2pt);
        \node[above] at (v4) {$4$};
        \node[left] at (v5) {$5$};
        \node[below] at (v6) {$6$};
        \node at (0.5,0.5) {$\{4,6\}$};
      \end{tikzpicture}
    };
    \node at (5,-4)
    {%
      \begin{tikzpicture}[x={(1.7,0)}, y={(0,1.7)}]
        \useasboundingbox (0,0) rectangle (2,1);
        \coordinate (v1) at (2,0);
        \coordinate (v2) at (1,0);
        \coordinate (v3) at (1,1);
        \coordinate (v4) at (0.5,1.5);
        \coordinate (v5) at (0,1);
        \coordinate (v6) at (0,0);
        \draw[orange,thick]
        (v1) -- (v2) -- (v3) -- (v5)
        (v1) -- (v2) -- (v6) -- (v5);
        \fill
        (v1) circle (2pt)
        (v5) circle (2pt);
        \fill[orange]
        (v2) circle (2pt)
        (v3) circle (2pt)
        (v6) circle (2pt);
        \node[below] at (v1) {$1$};
        \node[below] at (v2) {$2$};
        \node[right] at (v3) {$3$};
        \node[left] at (v5) {$5$};
        \node[below] at (v6) {$6$};
        \node at (0.5,0.5) {$\{1,5\}$};
      \end{tikzpicture}
    };
    \node at (10,-4)
    {%
      \begin{tikzpicture}[x={(1.7,0)}, y={(0,1.7)}]
        \useasboundingbox (0,0) rectangle (2,1);
        \coordinate (v1) at (2,0);
        \coordinate (v2) at (1,0);
        \coordinate (v3) at (1,1);
        \coordinate (v4) at (0.5,1.5);
        \coordinate (v5) at (0,1);
        \coordinate (v6) at (0,0);
        \draw[orange,thick]
        (v3) -- (v2) -- (v6)
        (v3) -- (v5) -- (v6);
        \fill
        (v3) circle (2pt)
        (v6) circle (2pt);
        \fill[orange]
        (v2) circle (2pt)
        (v5) circle (2pt);
        \node[below] at (v2) {$2$};
        \node[right] at (v3) {$3$};
        \node[left] at (v5) {$5$};
        \node[below] at (v6) {$6$};
        \node at (0.5,0.5) {$\{3,6\}$};
      \end{tikzpicture}
    };
  \end{tikzpicture}\vspace{2mm}
  \caption{The laps of the augmented path from \cref{ex:laps} (2).}
  \label{fig:lapsHouseAndGarden}
\end{figure}

Finally, we introduce claps, and show that sets thereof are closely related to isometric embeddings.

\begin{definition}
  \label{def:compatibility}
  A set of laps $C\subseteq \mathcal A$ is a \emph{compatible lap set}, or $C$ is a \emph{claps} for short, if there exists a metric tree $(\Gamma_C,d_{\Gamma_C})$ and a $1$-Lipschitz map $\iota_C\colon G\hookrightarrow\Gamma_B$ such that $d_G(a,b)=d_{\Gamma_C}(\iota_C(a),\iota_C(b))$ for all laps $\{a,b\}$ in $C$.

  A claps $C\subseteq \mathcal A$ is \textit{maximal}, if there is no claps $C'\subseteq\mathcal A$ such that $C\subsetneq C'$. We denote the set of claps by $\mathcal C$ and the set of maximal claps by $\mathcal C_{\max}$.
\end{definition}

% \ollie{Do we want to point out that the set of claps $\mathcal C$ is an abstract simplicial complex on $\mathcal A$, which we could call the \textit{claps complex}.}

\begin{lemma}
  \label{lem:minimalCoveringsAndMinimalEmbeddings}
  Let $C_1,\dots, C_r\in\mathcal C$ be claps covering $\mathcal A$, and let $\iota_i\colon G\rightarrow \Gamma_i$ be the $1$-Lipschitz maps such that for all laps $\{a,b\} \in C_i$ we have $d_{\Gamma_i}(\iota_i(a),\iota_i(b))=d_G(a,b)$. Then $\iota\coloneqq \iota_1\times\dots\times\iota_r\colon G\rightarrow\Gamma_1\times\dots\times\Gamma_r$ is an isometry.

  Moreover, if $C_1,\dots,C_r\in \mathcal C_{\max}$ are maximal claps, then we have
  \begin{align*}
    &C_1,\dots,C_r \text{ is a minimal covering of } \mathcal A \qquad \Longleftrightarrow \\
    &\qquad \iota_1\times\dots\times\iota_r \text{ is minimal isometric embedding of } G.
  \end{align*}
\end{lemma}
\begin{proof}
  Let $v,w\in V(G)$.  Pick a lap $\{a,b\} \in\mathcal A$ containing $v,w$, which is possibly by \cref{cor:lapsAndMaximalGeodesics}, and suppose w.l.o.g. $\{a,b\}\in C_1$.  Since $d_{\Gamma_i}(\iota_i(a),\iota_i(b))=d_G(a,b)$, $\iota_1|_\gamma$ is an isometry by \cref{lem:geodesicsIsometry}, which in turn implies $d_{\Gamma_1}(\iota_1(v),\iota_1(w))=d_G(v,w)$. And because $\iota_2,\dots,\iota_r$ are $1$-Lipschitz, we have $d_{\Gamma_j}(\iota_j(v),\iota_j(w))\leq d_G(v,w)$ for $j>1$.  Combining both, we obtain
  \begin{align*}
    &d_{\Gamma_1\times\dots\times\Gamma_r}(\iota(v),\iota(w)) \\
    &\quad= \max(\underbrace{d_{\Gamma_1}(\iota_1(v),\iota_1(w))}_{=d_G(v,w)},\underbrace{d_{\Gamma_2}(\iota_2(v),\iota_2(w))}_{\leq d_G(v,w)},\dots,\underbrace{d_{\Gamma_r}(\iota_r(v),\iota_r(w))}_{\leq d_G(v,w)})
    =d_G(v,w),
  \end{align*}
  showing that $\iota$ is an isometry.

  From hereon, assume $C_1,\dots,C_r$ are maximal as sets of compatible laps.

  For the ``$\Rightarrow$'' implication of the equivalence, let $C_1,\dots,C_r$ be a minimal covering.  W.l.o.g., it suffices to show that $\iota_2\times\dots\times\iota_r\colon G\rightarrow\Gamma_2\times\dots\times\Gamma_r$ is not an isometry.
  Let $\{a,b\} \in C_1\setminus\bigcup_{i=2}^r C_i$. If $\iota_2\times\dots\times\iota_r$ were an isometry, then there must be some $i>1$ such that $d_{\Gamma_i}(\iota_i(a),\iota_i(b))=d_G(a,b)$.  This implies that $C_i\cup\{a,b\}$ is a clap and contradicts the maximality of $C_i$.

  For the ``$\Leftarrow$'' implication of the equivalance, assume that $C_2,\dots,C_r$ cover $\mathcal A$.  Then by the first part $\iota_2\times\dots\times\iota_r$ is already an isomorphism, contradicting the minimality of $\iota_1\times\dots\times\iota_r$.
\end{proof}

\begin{example}
  \label{ex:claps}\
  \begin{enumerate}
    \item Let $G$ be the prism graph with the $6$ laps from \cref{ex:laps} (1).  Let $\Gamma_2$ be a path with vertices $x,y,z$ and edges $\{x,y\}, \{y,z\}$ that are all length one., One can check that the set of laps $\{ \{1,5\}, \{1,6\}\}$, $\{\{2,4\},\{2,6\}\}$ and $\{\{3,4\},\{3,5\}\}$ are all compatible using the maps
    \begin{flalign*}
      {\hspace{15mm}
        \setlength{\arraycolsep}{10pt}
        \begin{array}{rrrrr}
          \iota_1\colon& G\rightarrow \Gamma_2,& 1\mapsto x,& 2,3,4\mapsto y,& 5,6\mapsto z,\\
          \iota_2\colon& G\rightarrow \Gamma_2,& 2\mapsto x,& 1,3,5\mapsto y,& 4,6\mapsto z,\\
          \iota_3\colon& G\rightarrow \Gamma_2,& 3\mapsto x,& 1,2,6\mapsto y,& 4,5\mapsto z.
        \end{array}
      }&&
    \end{flalign*}
    By \cref{lem:minimalCoveringsAndMinimalEmbeddings}, $\iota_1\times\iota_2\times\iota_3$ is an isometry.
    \item Let $G$ be the augmented path with the $4$ laps from \cref{ex:laps} (2).  As above, one can check that the set of laps $\{\{1,4\},\{1,5\}\}$ and $\{\{3,6\},\{4,6\}\}$ are compatible using the maps
    \begin{flalign*}
      {\hspace{15mm}
        \setlength{\arraycolsep}{10pt}
        \begin{array}{rrrrrr}
          \iota_1\colon& G\rightarrow \Gamma_3,& 1\mapsto x,& 2\mapsto y,& 3,6\mapsto z,& 4,5\mapsto w\\
          \iota_2\colon& G\rightarrow \Gamma_2,& 6\mapsto x,& 1,2,5\mapsto y,& 3,4\mapsto z,
        \end{array}
      }&&
    \end{flalign*}
    where $\Gamma_3$ is a path with vertices $x,y,z,w$ and $\Gamma_2$ is as above.     By \cref{lem:minimalCoveringsAndMinimalEmbeddings}, $\iota_1\times\iota_2$ is an isometry.
  \end{enumerate}
\end{example}

We now conclude this subsection with our main algorithm.

\begin{algorithm}\
  \label{alg:exactAlgorithm}
  \begin{algorithmic}[1]
    \REQUIRE{$G$, a simple connected metric graph}
    \ENSURE{all minimal embeddings}
    \STATE Construct the set of laps $\mathcal{A}$.
    \STATE Construct the set of maximal claps $\mathcal{C}_{\max}$.
    \STATE Find all minimal covers:
    \begin{equation*}
      \mathcal B\coloneqq \Big\{ \{C_1,\dots,C_r\} \subseteq \mathcal{C}_{\max} \text{ minimal} \bigmid \mathcal{A}=\textstyle\bigcup_{i=1}^r C_i \Big\}. \vspace{-1em}
    \end{equation*}
    \RETURN{$\{\iota_{C_1}\times\dots\times\iota_{C_r}\mid \{C_1,\dots,C_r\}\in \mathcal B\}$.}
  \end{algorithmic}
\end{algorithm}
\begin{proof}
  The correctness of the algorithm follows directly from \cref{lem:minimalCoveringsAndMinimalEmbeddings}.
\end{proof}

\begin{remark}
  \label{rem:exactAlgorithm}
  In \textsc{PhylogeneticRank.jl}, all three steps of \cref{alg:exactAlgorithm} are implemented as poset traversals:
  \begin{enumerate}
    \item For $\mathcal{A}$, we construct the poset of all geodesics $\mathcal G$, beginning with the edges at the bottom and ending with the maximal geodesics $\mathcal G_{\max}$ at the top.
    \item For $\mathcal{C}_{\max}$, one could construct the poset of all claps $\mathcal C$, beginning with the singleton claps at the bottom and ending with the maximal claps $\mathcal{C}_{\max}$ at the top.
    In practise however, it is much faster us to check incompatible sets of laps and large set of laps for two reasons:
    \begin{itemize}
      \item if $C$ is incompatible, it may contain a pair of incompatible laps that can be quickly identified with \cref{thm:compatiblePairs}, and
      \item the larger $C$ is, the more dimensions are eliminated in \cref{prop:compatibility} due to Condition (C-consistent).
    \end{itemize}
    Therefore, our code for computing $\mathcal{C}_{\max}$ traverses the poset of incompatible sets of laps $2^{\mathcal A}\setminus\mathcal C$, beginning with the complete set $\mathcal A$ at the bottom.
    \item For $\mathcal B$, we construct the poset of all subsets of $\mathcal C_{\max}$, beginning with singleton sets at the bottom and ending with covers at the top.
  \end{enumerate}
  Implementing all constituents of \cref{alg:exactAlgorithm} as poset traversals allows us to streamline the code.  The bottleneck of the algorithm is in Step 2.
\end{remark}

\subsection{Checks for Compatibility}

As mentioned in \cref{rem:exactAlgorithm}, the bottleneck of \cref{alg:exactAlgorithm} is checking sets of laps for compatibility in Step 2.  This is because we are using a brute-force approach using the tropical Grassmannian $\mathrm{TGr}(2,n)$ (see \cite[Section 4.3]{MaclaganSturmfels2015} and \cite[Section 10.6]{Joswig2021} for details on $\mathrm{TGr}(2,n)$):

\begin{proposition}
  \label{prop:compatibility}
  Let $C\subseteq\mathcal A$ be a set of laps.  Write $V=\{1,\dots,n\}$ and consider the following polytope $P_C\subseteq\RR^{\binom{n}{2}}$ consisting of all points $(d_{ij})_{ij\in\binom{n}{2}}$ satisfying the following linear constraints:
  \setlist[description]{font=\normalfont}
  \begin{itemize}[leftmargin=35mm,itemsep=1mm]
    \item[\normalfont (non-negative):] $d_{ij}\geq 0$ for all $i,j$,
    \item[\normalfont (1-Lipschitz):] $d_{ij}\leq d_G(i,j)$ for all $i,j$,
    \item[\normalfont (triangle-ineq):] $d_{ij}\leq d_{ik}+d_{kj}$ for all $i,j,k$,
    \item[\normalfont ($C$-consistent):] $d_{ij}= d_G(i,j)$ for all $i,j$ on some lap $\{a,b\}\in C$.
  \end{itemize}
  Then
  \begin{equation*}
    C\text{ is compatible}\qquad\iff\qquad P_C\cap\mathrm{TGr}(2,n)\neq\emptyset,
  \end{equation*}
  where $\mathrm{TGr}(2,n)\subseteq\RR^{\binom{n}{2}}$ is the tropical Grassmannian of tropical planes in $\RR^n$.
\end{proposition}
\begin{proof}
  The proof is straighforward.  The polytope $P_C$ is the set of distances that contract $d_G$ and coincide with $d_G$ on geodesics that give rise to laps in $C$.  And $\mathrm{TGr}(2,n)$ is the space of tree metrics \cite[Corollary 10.48]{Joswig2021}.
\end{proof}

To speed up our checks for compatibility, we now give a complete classification of pairs of compatible laps in the form of a relaxed 4-point condition.  In our code, it is used to quickly identify sets containing incompatible pairs, which must be incompatible themselves.  For said classification, we first require the following small lemma:

\begin{lemma}
  \label{lem:compatiblePair}
  Let $X=\{1,2,3,4\}$ and let $d\colon X\times X\rightarrow\RR_{\geq 0}$ be a metric such that
  \begin{equation}
    \label{eq:crossSumOrdering}
    d(1,4)+d(2,3)\leq d(1,3)+d(2,4)\leq d(1,2)+d(3,4).
  \end{equation}
  Then there is a metric $d'\colon X\times X\rightarrow\RR_{\geq 0}$ satisfying
  \begin{itemize}
    \item $d'(i,j)=d(i,j)$ for $\{i,j\}\neq \{1,2\}, \{3,4\}$,
    \item $d'(i,j)\leq d(i,j)$ for $\{i,j\}=\{1,2\}, \{3,4\}$,
    \item $d'(1,4)+d'(2,3)\leq d'(1,3)+d'(2,4)=d'(1,2)+d'(3,4)$.
  \end{itemize}
\end{lemma}
\begin{proof}
  This lemma can be shown through direct computation.  Consider $\RR^6$ with coordinates $d_{12},d_{34},d_{13},d_{24},d_{14},d_{23}$ as well as $\RR^8$ with two extra coordinates $d_{12}',d_{34}'$.

  In $\RR^6$, consider the polytope $P$ cut out by the following linear constraints:
  \begin{itemize}
    \item $d_{ij}\geq 0$ for all $i,j$ and $d_{ij}\leq d_{ik}+d_{kj}$ for all $i,j,k$,
    \item $d_{14}+d_{23}\leq d_{13}+d_{24}\leq d_{12}+d_{34}$ as in Equation~\eqref{eq:crossSumOrdering}.
  \end{itemize}

  In $\RR^8$, and with $d_{ij}'\coloneqq d_{ij}$ for $\{ij\}\neq\{1,2\},\{3,4\}$, consider the polytope $Q$ cut out by the following linear constraints:
  \begin{itemize}
    \item $d_{ij}\geq 0$ for all $i,j$ and $d_{ij}\leq d_{ik}+d_{kj}$ for all $i,j,k$,
    \item $d_{ij}'\geq 0$ for all $i,j$ and $d_{ij}'\leq d_{ik}'+d_{kj}'$ for all $i,j,k$,
    \item $d_{ij}'\leq d_{ij}$ for $\{ij\}=\{1,2\},\{3,4\}$,
    \item $d_{12}'+d_{34}' = d_{13} + d_{24}$.
  \end{itemize}

  In words, each point in $P$ is a metric $d$ on $X$ satisfying Equation~\eqref{eq:crossSumOrdering}, and each point in $Q$ is an altered distance with the desired properties.

  A quick computation shows that the projection $\RR^8\twoheadrightarrow\RR^6$ onto $d_{ij}$ maps $Q$ to $P$, which proves the lemma.  OSCAR code verifying this computation can be found in our repository \cite{ACGJQR2026github}.
\end{proof}

\begin{theorem}
  \label{thm:compatiblePairs}
  Let $\{a_1,a_2\}, \{b_1,b_2\}\in\mathcal A$ be two laps.  Then
  \begin{align*}
    & \{a_1,a_2\}, \{b_1,b_2\}\text{ compatible} \qquad\Longleftrightarrow\\
    & \quad d_G(a_1,a_2)+d_G(b_1,b_2) \leq \max\{d_G(a_1,b_1)+d_G(a_2,b_2), d_G(a_1,b_2)+d_G(a_2,b_1)\}.
  \end{align*}
\end{theorem}
\begin{proof}
  For the ``$\Rightarrow$'' direction, let $\{a_1,a_2\}$, $\{b_1,b_2\}$ is compatible, and let $\iota\colon G\rightarrow\Gamma$ be the $1$-Lipschitz map with $d_G(a_1,a_2)=d_\Gamma(\iota(a_1),\iota(a_2))$ and $d_G(b_1,b_2)=d_\Gamma(\iota(b_1),\iota(b_2))$.  Then we have
  \begin{align*}
    &d_G(a_1,a_2)+d_G(b_1,b_2) = d_\Gamma(\iota(a_1),\iota(a_2))+d_\Gamma(\iota(b_1),\iota(b_2)) \\
    &\qquad \overset{\customlabel{eq:compatiblePairs1}{(1)}}{\leq} \max\{d_\Gamma(\iota(a_1),\iota(b_1))+d_\Gamma(\iota(a_2),\iota(b_2)), d_\Gamma(\iota(a_1),\iota(b_2))+d_\Gamma(\iota(a_2),\iota(b_1))\} \\
    &\qquad \overset{\customlabel{eq:compatiblePairs2}{(2)}}{\leq} \max\{d_G(a_1,b_1)+d_G(a_2,b_2),d_G(a_1,b_2)+d_G(a_2,b_1)\},
  \end{align*}
  where Inequality~\ref{eq:compatiblePairs1} follows from the 4-point condition on $\Gamma$ and Inequality~\ref{eq:compatiblePairs2} holds since $\iota$ is $1$-Lipschitz.

  For the ``$\Leftarrow$'' direction, consider the metric space $(X,d)$ where $X=\{a_1,a_2,b_1,b_2\}$ and $d=d_G$.  Since $d_G(a_1,a_2)+d_G(b_1,b_2) \leq \max\{d_G(a_1,b_1)+d_G(a_2,b_2), d_G(a_1,b_2)+d_G(a_2,b_1)\}$ we are in one of four cases:
  \begin{enumerate}
    \item $d_G(a_1,a_2)+d_G(b_1,b_2) \leq d_G(a_1,b_1)+d_G(a_2,b_2) \leq d_G(a_1,b_2)+d_G(a_2,b_1)$, or
    \item $d_G(a_1,a_2)+d_G(b_1,b_2) \leq d_G(a_1,b_2)+d_G(a_2,b_1) \leq d_G(a_1,b_1)+d_G(a_2,b_2)$, or
    \item $d_G(a_1,b_1)+d_G(a_2,b_2) \leq d_G(a_1,a_2)+d_G(b_1,b_2) \leq d_G(a_1,b_2)+d_G(a_2,b_1)$, or
    \item $d_G(a_1,b_2)+d_G(a_2,b_1) \leq d_G(a_1,a_2)+d_G(b_1,b_2) \leq d_G(a_1,b_1)+d_G(a_2,b_2)$.
  \end{enumerate}
  In all four cases, we are in a position to apply \cref{lem:compatiblePair} to obtain a $1$-Lipschitz map $\mathrm{id}_X\colon (X,d) \rightarrow (X,d')$ with $d(a_1,a_2)=d'(a_1,a_2)$, $d(b_1,b_2)=d'(b_1,b_2)$ and $(X,d')$ satisfying the four point condition, which gives us an isometry $(X,d')\overset{\sim}\longrightarrow \Gamma$.  The resulting composition $(X,d) \overset{\mathrm{id}_X}\longrightarrow (X,d')\overset{\sim}\longrightarrow \Gamma$ can then to a $1$-Lipschitz map $G\rightarrow \Gamma$ \cite{Kirk1998}, showing that $\{a_1,a_2\}$ and $\{b_1,b_2\}$ are compatible.
\end{proof}

As quick corollaries, we obtain:

\begin{definition}
  Let $\{a_1,a_2\},\{b_1,b_2\}\in \mathcal{A}$ be two distinct laps.  We say $\{a_1,a_2\},\{b_1,b_2\}$ are
  \begin{enumerate}
    \item \emph{T-shaped}, if
    \begin{enumerate}
      \item $a_i\neq b_1,b_2$ yet $a_i$ lies on the lap $\{b_1,b_2\}$ for some $i=1,2$, or
      \item $b_i\neq a_1,a_2$ yet $b_i$ lies on the lap $\{a_1,a_2\}$ for some $i=1,2$,
    \end{enumerate}
    \item \emph{V-shaped}, if $\{a_1,a_2\}\cap\{b_1,b_2\}\neq\varnothing$.
  \end{enumerate}
\end{definition}

\begin{corollary}
  \label{cor:compatiblePairs}
  Let $\{a_1,a_2\},\{b_1,b_2\}\in \mathcal{A}$ be two laps.  Then
  \begin{enumerate}
    \item if $\{a_1,a_2\},\{b_1,b_2\}$ are $T$-shaped, then they are incompatible, and
    \item if $\{a_1,a_2\},\{b_1,b_2\}$ are $V$-shaped, then they are compatible.
  \end{enumerate}
\end{corollary}
\begin{proof}\
  \begin{enumerate}[leftmargin=*]
    \item Without loss of generality, let $a_1\neq b_1, b_2$ yet $a_1$ lies on the lap $\{b_1,b_2\}$.  Assume that $\{a_1,a_2\}$, $\{b_1,b_2\}$ are compatible.  Without loss of generality, suppose $d_G(a_1,b_1)+d_G(a_2,b_2) \geq d_G(a_1,b_2)+d_G(a_2,b_1)$.  By \cref{thm:compatiblePairs}, we thus get $d_G(a_1,a_2)+d_G(b_1,b_2) \leq d_G(a_1,b_1)+d_G(a_2,b_2)$ which implies that
    \begin{equation*}
      d_G(a_2,b_2)\geq d_G(a_1,a_2) + (d_G(b_1,b_2)-d_G(a_1,b_1)) = d_G(a_1,a_2) + d_G(a_1,b_2),
    \end{equation*}
    where the final equality follows from the assumption that $a_1$ lies on a geodesic connecting $b_1$ and $b_2$.  This implies that $a_1$ lies in the interior of a geodesic connecting $a_2$ and $b_2$, contradicting the local antipodality of $a_1$ and $a_2$.

    \item In the case of V-shaped laps, we can construct a $1$-Lipschitz tree embedding which realises the distances in each lap using the same method as in \cref{prop: phylogenetic rank less than n-1}.\qedhere
  \end{enumerate}
\end{proof}

\begin{example}\label{ex:VTshapes}\
  \begin{enumerate}[leftmargin=*]
    \item Let $G$ be the prism graph from \cref{ex:claps} (1).  The pairs $\{ \{1,5\}, \{1,6\}\}$, $\{\{2,4\},\{2,6\}\}$ and $\{\{3,4\},\{3,5\}\}$ are all $V$-shaped, confirming again their compatibility.  In contrast, $\{\{1,5\},\{2,4\}\}$ is $T$-shaped which makes it incompatible.
    \item Let $G$ be the augmented path from \cref{ex:claps} (2).  As above, the pairs $\{\{1,4\},\{1,5\}\}$ and $\{\{3,6\},\{4,6\}\}$ are $V$-shaped, while $\{\{1,4\},\{3,6\}\}$ is $T$-shaped.
  \end{enumerate}
\end{example}

%%% Local Variables:
%%% mode: LaTeX
%%% TeX-master: "graph_embedding_paper"
%%% End:

\section{Computational results}\label{sec:computations}
We have implemented our algorithms in Sections \ref{sec:heuristicAlgorithm} and \ref{sec:exactAlgorithm} in \textsc{Julia}, relying on the computer algebra system OSCAR \cite{OSCAR,OSCAR-book}.  We have computed the exact ranks of all graphs up to $7$ vertices, and approximate ranks of all graphs up to $8$ vertices.  In this section, we go over some examples of interest, such as a counterexample to a conjecture of Pachter and Sturmfels on a generalisation of the four point condition.

The full data can be found in \cite{ACGJQR2026github}.  A summary of the data is shown in Table~\ref{tab: small graph ranks} and Table~\ref{tab: small graph ranks connected}.  The ranks were computed by going over all biconnected graphs using Algorithms \ref{alg:greedyAlgorithm} and \ref{alg:exactAlgorithm}, then extrapolating to all connected graphs using Proposition~\ref{prop: separable graph rank}.  Note that the approximate ranks are always an upper bound and that they are exact if equal to $1$ or $2$ by \cref{thm: greedy alg correct for rank 1/2}.

\begin{table}[t]
  \centering
  \begin{tabular}{m{1cm} m{1cm} m{1cm} m{1cm} m{1cm} m{1cm}}
    \toprule
    & \multicolumn{5}{c}{Phylogenetic rank $\prank(G)$}\\
    $|G|$ & 1 & 2 & 3 & 4 & 5 \\
    \midrule
    4  & 1 & \multicolumn{1}{m{1cm}|}{2}    & 0    & 0    & 0 \\
    \cline{4-4}
    5  & 1 & 8   & \multicolumn{1}{m{1cm}|}{1}    & 0    & 0 \\
    % \cline{5-5}
    6  & 1 & 32  & \multicolumn{1}{m{1cm}|}{23}   & 0    & 0 \\
    \cline{5-5}
    7 & 1 & 131 & 316 & \multicolumn{1}{m{1cm}|}{20} & 0\\
    8* & 1 & 662 & 4313 & \multicolumn{1}{m{1cm}|}{2145} & \textcolor{red}{\bf 2}\\
    \bottomrule
  \end{tabular}
  \caption{Phylogenetic ranks of $2$-connected graphs with a given number of vertices. The ranks of $8$-vertex graphs are an upper bound provided by the greedy algorithm: Algorithm~\ref{alg:greedyAlgorithm}. We verified the two rank five graphs. The zigzag line separates $\lceil n/2 \rceil$ and $\lceil n/2 \rceil +1$.}
  \label{tab: small graph ranks}
\end{table}

\begin{table}[t]
  \centering
  \begin{tabular}{m{1cm} m{1cm} m{1cm} m{1cm} m{1cm} m{1cm}}
    \toprule
    & \multicolumn{5}{c}{Phylogenetic rank $\prank(G)$}\\
    $|G|$ & 1 & 2 & 3 & 4 & 5 \\
    \midrule
    4  & 4 & \multicolumn{1}{m{1cm}|}{2}    & 0    & 0    & 0 \\
    \cline{4-4}
    5  & 9 & 11   & \multicolumn{1}{m{1cm}|}{1}    & 0    & 0 \\
    % \cline{5-5}
    6  & 22 & 66  & \multicolumn{1}{m{1cm}|}{24}   & 0    & 0 \\
    \cline{5-5}
    7  & 59 & 383 & 391  & \multicolumn{1}{m{1cm}|}{20}   & 0\\
    % \cline{6-6}
    8* & 165 & 2373 & 6266 & \multicolumn{1}{m{1cm}|}{2311} & \textcolor{red}{\bf 2}\\
    \bottomrule
  \end{tabular}
  \caption{Phylogenetic ranks of connected graphs with a given number of vertices. The ranks of $8$-vertex graphs are an upper bound provided by the greedy algorithm: Algorithm~\ref{alg:greedyAlgorithm}. We verified the two rank five graphs. The zigzag line separates $\lceil n/2 \rceil$ and $\lceil n/2 \rceil +1$}.
  \label{tab: small graph ranks connected}
\end{table}

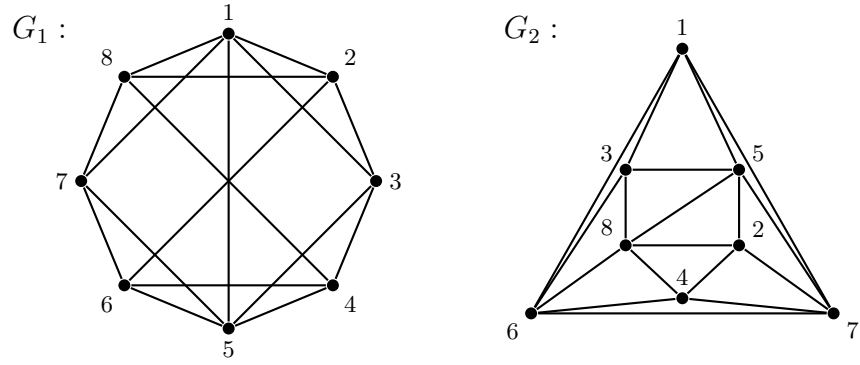
\begin{figure}[t]
  \centering
  \begin{tikzpicture}[thick]
    \tikzset{
      vx/.style ={circle,fill=black,inner sep=0pt,minimum size=4.5pt},
      lb/.style  ={font=\scriptsize,inner sep=2pt},
    }
    % vertices, cyclically placed
    \foreach \i/\a in {1/90,2/45,3/0,4/-45,5/-90,6/-135,7/180,8/135}{
      \node[vx,label={[lb]\a:$\i$}] (v\i) at (\a:1.95){};}
    % the 8-cycle
    \foreach \i/\j in {1/2,2/3,3/4,4/5,5/6,6/7,7/8,8/1}{\draw (v\i)--(v\j);}
    % short chords: 13, 35, 57, 71, 28, 46
    \foreach \i/\j in {1/3,3/5,5/7,7/1,2/8,4/6}{\draw (v\i)--(v\j);}
    % antipodal chords 15, 26, 48: bent so they do not all meet at the centre
    \draw (v1) to (v5);
    \draw (v2) to (v6);
    \draw (v4) to (v8);
    \node at (-2.5,2) {$G_1:$};

    % ---- second graph, drawn planar (no edge crossings), shifted to the right ----
    \begin{scope}[shift={(6,-0.25)}]
      \node[vx,label={[lb]90:$1$}]    (w1) at (0,2){};
      \node[vx,label={[lb]20:$2$}] (w2) at (.75,-.6){};
      \node[vx,label={[lb]160:$3$}]   (w3) at (-.75,0.4){};
      \node[vx,label={[lb]90:$4$}]    (w4) at (0,-1.3){};
      \node[vx,label={[lb]20:$5$}]    (w5) at (0.75,0.4){};
      \node[vx,label={[lb]-150:$6$}]  (w6) at (-2,-1.5){};
      \node[vx,label={[lb]-20:$7$}]   (w7) at ( 2,-1.5){};
      \node[vx,label={[lb]150:$8$}]  (w8) at (-.75,-.6){};
      \foreach \i/\j in {1/3,1/5,1/6,1/7,2/4,2/5,2/7,2/8,3/5,3/6,3/8,4/6,4/7,4/8,5/7,5/8,6/7,6/8}{\draw (w\i)--(w\j);}
    \end{scope}
    \node at (4,2) {$G_2:$};
  \end{tikzpicture}\vspace{-2mm}
  \caption{The two $8$-vertex graphs with rank five}
  \label{fig: rank 5 with 8 vertices}
\end{figure}

\begin{example}[graphs of high rank]
  \label{example: small graph ranks}
  In terms of number of vertices $n$, the smallest graphs with phylogenetic rank exceeding $\lceil \frac{n}{2} \rceil$ are the $8$-vertex graphs $G_1, G_2$ with the following edge sets, see also \cref{fig: rank 5 with 8 vertices}:
  \begin{align*}
    E(G_1)&= \{12,23,34,45,56,67,78,18,\; 13,35,57,17,28,46 \},\\
    E(G_2)&= \{13,15,16,17,24,25,27,28,35,36,38,46,47,48,57,58,67,68\}.
  \end{align*}
\end{example}

\begin{figure}[t]
  % \resizebox{\textwidth}{!}{%
  \begin{tikzpicture}[
    vtx/.style={circle, fill, inner sep=1.6pt},
    every label/.style={font=\small\linespread{0.9}\selectfont,
      align=center, label distance=2pt}, % line spread controls vertical spacing between lines
    scale=0.95]

    %%% left: the 6-cycle with hub %%%
    \begin{scope}[scale=0.7]
      \foreach \i [evaluate=\i as \ang using {90-(\i-1)*60}] in {1,...,6}{
        \node[vtx, label={\ang:\i}] (v\i) at (\ang:2.5) {};
      }
      \node[vtx, label={90:7}] (v7) at (0,0) {};
      \foreach \i [evaluate=\i as \j using {int(mod(\i,6)+1)}] in {1,...,6}{
        \draw (v\i) -- (v\j);
      }
      \foreach \i in {2,3,6,5}{ \draw (v7) -- (v\i); }
    \end{scope}

    \begin{scope}[shift={(3,1)}, scale=1.25]
      % --- T_1 ---
      \node[label={$T_1$}] at (.25,-.5) {};
      \node[vtx, label={below:$4$}] (y1) at (1,0) {};
      \node[vtx, label={below:$35$}] (y2) at (2,0) {};
      \node[vtx, label={below:$26$}] (y3) at (3,0) {};
      \node[vtx, label={below:$1$}] (y4) at (4,0) {};
      \node[vtx, label={above:$7$}] (y5) at (2.5,.5) {};
      \coordinate (junction) at (2.5,0) {};
      \draw (y1) -- (y2) -- (y3) -- (y4);
      \draw (junction) -- (y5);
      % --- T_2 ---
    \end{scope}
    \begin{scope}[shift={(3,-1)}, scale=1.25]
      \node[label={$T_2$}] at (.25,-.5) {};
      \node[vtx, label={below:$56$}]  (x1) at (1,0) {};
      \node[vtx, label={below:$147$}] (x2) at (2,0) {};
      \node[vtx, label={below:$23$}]  (x3) at (3,0) {};
      \draw (x1) -- (x2) -- (x3);

    \end{scope}
  \end{tikzpicture}\vspace{-2mm}%
  % }
  \caption{The graphs in Example~\ref{example: rank 5 with induced 15 cycle}. The graph on the left has a minimal embedding given by the trees on the right.}
  \label{fig: 6 cycle with 4 spokes}
\end{figure}
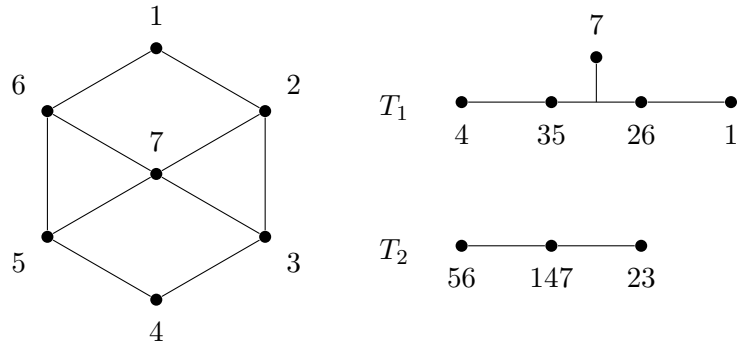

\begin{example}[graph with non-hereditary rank]\label{example: rank 5 with induced 15 cycle}
  \cref{thm: phylogenetic rank one} states that the property of having phylogenetic rank one is c-hereditary, i.e., that for any graph of rank one, its connected induced subgraphs also have rank one.

  To see that this is not true for phylogenetic rank at most two, let $G$ be the cycle on $6$ vertices with a vertex $7$ connected to $2$, $3$, $5$ and $6$.  One can see show that $\phylogeneticrank(G)=2$ by tree embedding $\iota\colon G \rightarrow T_1 \times T_2$ in \cref{fig: 6 cycle with 4 spokes}.
  Note that the induced subgraph $G[1,\dots,6]$ is the $6$-cycle, which has phylogenetic rank $3$. Thus, phylogenetic rank does not weakly decrease when taking induced subgraphs.
\end{example}

In \cite{AlgStatsForBio}, Pachter and Sturmfels conjectured the following:

\begin{conjecture}[{\cite[Section~3.5]{AlgStatsForBio}}]\label{conj: PS 4 pt cond generalisation}
  Let $X$ be a finite metric space and assume that for each subset $Y \subseteq X$ with $Y = 2k + 2$, we have $\prank(Y) \le k$, then $\prank(X) \le k$.
\end{conjecture}

\cref{conj: PS 4 pt cond generalisation} is a generalisation of the four point condition as in \cref{prop: four point condition}, which can be recovered by setting $k=1$.  The following example shows that \cref{conj: PS 4 pt cond generalisation} fails for $k = 2$:

\begin{figure}[t]
  \tikzset{
    vtx/.style      = {circle, fill, inner sep=1.7pt},          % ordinary vertex
    removed/.style  = {circle, draw, fill=white, inner sep=2.6pt, thin},
    % deleted vertex
    junction/.style = {circle, inner sep=0pt, outer sep=0pt, minimum size=0pt},
    % branch point, no dot
    glbl/.style     = {font=\small, inner sep=3pt},             % graph label
    tlbl/.style     = {font=\small, align=center, inner sep=0pt},% tree label(s)
  }
  \centering
  \begin{tikzpicture}
    \node (center) at (0,0)
    {
      \begin{tikzpicture}[thick, line join=round, scale=1.5,rotate=270]
        \tikzset{
          v/.style ={circle,fill=black,inner sep=0pt,minimum size=4.5pt},
          lb/.style  ={font=\scriptsize,inner sep=2pt},
        }
        \node[v, label={[lb] above: $1$}] at (-.5,-1)  (v1){};
        \node[v, label={[lb] below: $2$}] at (.5,-1)   (v2){};
        \node[v, label={[lb] above: $3$}] at (-.5,0)   (v3){};
        \node[v, label={[lb] below: $4$}] at (.5,0)    (v4){};
        \node[v, label={[lb] above: $5$}] at (-.5,1)   (v5){};
        \node[v, label={[lb] below: $6$}] at (.5,1)    (v6){};
        \node[v, label={[lb] above: $7$}] at (0,1.85)  (v7){};

        \draw (v1)--(v3)--(v4)--(v2)--(v1);
        \draw (v3)--(v5)--(v6)--(v4) (v3)--(v6) (v4)--(v5);
        \draw (v5)--(v7)--(v6);

      \end{tikzpicture}

    };
    \node (topLeft) at (-4,4)
    {
      \begin{tikzpicture}[thick, line join=round, scale=0.8]
        %% ------------------------- the graph -------------------------
        \node[vtx]     (g7) at (1,3) {};
        \node[vtx]     (g5) at (.25,2) {};
        \node[vtx]     (g6) at (1.75,2) {};
        \node[vtx]     (g3) at (.25,1) {};
        \node[vtx]     (g4) at (1.75,1) {};
        \node[vtx]     (g1) at (.25,0) {};
        \node[removed] (g2) at (1.75,0) {};  % vertex 2 is the deleted one

        \draw (g7) -- (g5) -- (g6) -- (g7);   % 7-5, 5-6, 6-7
        \draw (g5) -- (g3) -- (g4) -- (g6);   % 5-3, 3-4, 4-6
        \draw (g5) -- (g4);                   % diagonal
        \draw (g3) -- (g6);                   % diagonal
        \draw (g3) -- (g1) -- (g2) -- (g4);   % 3-1, 1-2, 2-4

        \node[glbl, above]      at (g7) {$7$};
        \node[glbl, left]       at (g5) {$5$};
        \node[glbl, right]      at (g6) {$6$};
        \node[glbl, left]       at (g3) {$3$};
        \node[glbl, right]      at (g4) {$4$};
        \node[glbl, left]       at (g1) {$1$};

        %% ------------------------- first tree ------------------------
        \begin{scope}[shift={(3,1.75)}]
          \node[vtx]      (a1) at (0,0) {};     % 2
          \node[vtx]      (a2) at (1.5,0) {};     % 4
          \node[junction] (a3) at (2.25,0) {};     % branch point (no label)
          \node[vtx]      (a4) at (3.0,0) {};     % 5,6
          \node[vtx]      (a5) at (4.5,0) {};     % 7
          \node[vtx]      (a6) at (2.25,0.75) {};     % 3

          \draw (a1) -- (a2) -- (a3) -- (a4) -- (a5);
          \draw (a3) -- (a6);

          \node[tlbl, below=4pt] at (a1) {$1$};
          \node[tlbl, below=4pt] at (a2) {$3$};
          \node[tlbl, below=4pt] at (a4) {$56$};
          \node[tlbl, below=4pt] at (a5) {$7$};
          \node[tlbl, above=4pt] at (a6) {$4$};
        \end{scope}

        %% ------------------------ second tree ------------------------
        \begin{scope}[shift={(3.75,.5)}]
          \node[vtx] (b1) at (0  ,0) {};          % 5
          \node[vtx] (b2) at (1.5,0) {};          % 3,4,6,7
          \node[vtx] (b3) at (3.0,0) {};          % 2

          \draw (b1) -- (b2) -- (b3);

          \node[tlbl, below=4pt] at (b1) {$5$};
          \node[tlbl, below=4pt] at (b2) {$3467$};
          \node[tlbl, below=4pt] at (b3) {$1$};
        \end{scope}
      \end{tikzpicture}
    };
    \node (topRight) at (4,4)
    {
      \begin{tikzpicture}[thick, line join=round, scale=0.8]

        %% ------------------------- the graph -------------------------
        \node[vtx]     (g7) at (1,3) {};
        \node[vtx]     (g5) at (.25,2) {};
        \node[vtx]     (g6) at (1.75,2) {};
        \node[vtx]     (g3) at (.25,1) {};
        \node[removed] (g4) at (1.75,1) {};   % vertex 4 is the deleted one
        \node[vtx]     (g1) at (.25,0) {};
        \node[vtx]     (g2) at (1.75,0) {};

        \draw (g7) -- (g5) -- (g6) -- (g7);
        \draw (g5) -- (g3) -- (g4) -- (g6);
        \draw (g5) -- (g4);
        \draw (g3) -- (g6);
        \draw (g3) -- (g1) -- (g2) -- (g4);

        \node[glbl, above] at (g7) {$7$};
        \node[glbl, left]  at (g5) {$5$};
        \node[glbl, right] at (g6) {$6$};
        \node[glbl, left]  at (g3) {$3$};
        \node[glbl, left]  at (g1) {$1$};
        \node[glbl, right] at (g2) {$2$};

        %% ------------------------- first tree ------------------------
        \begin{scope}[shift={(3,1.75)}]
          \node[vtx] (a1) at (0  ,0) {};          % 7
          \node[vtx] (a2) at (1.5,0) {};          % 5,6
          \node[vtx] (a3) at (3,0)   {};          % 3
          \node[vtx] (a4) at (4.5,0) {};          % 1
          \node[vtx] (a5) at (3.75,0.75) {};          % 2
          \node[junction] (a6) at (3.75,0) {};    % branch point

          \draw (a1) -- (a2) -- (a3) -- (a6) -- (a4);
          \draw (a6) -- (a5);

          \node[tlbl, below=4pt] at (a1) {$7$};
          \node[tlbl, below=4pt] at (a2) {$56$};
          \node[tlbl, below=4pt] at (a3) {$3$};
          \node[tlbl, below=4pt] at (a4) {$1$};
          \node[tlbl, above=4pt] at (a5) {$2$};
        \end{scope}

        %% ------------------------ second tree ------------------------
        \begin{scope}[shift={(3.75,0.5)}]
          \node[vtx] (b1) at (0  ,0) {};          % 5,3
          \node[vtx] (b2) at (1.5,0) {};          % 6,1,7
          \node[vtx] (b3) at (3.0,0) {};          % 2

          \draw (b1) -- (b2) -- (b3);

          \node[tlbl, below=4pt] at (b1) {$53$};
          \node[tlbl, below=4pt] at (b2) {$61 7$};
          \node[tlbl, below=4pt] at (b3) {$2$};
        \end{scope}

      \end{tikzpicture}
    };
    \node (botLeft) at (-4,-4)
    {
      \begin{tikzpicture}[thick, line join=round, scale=0.8]

        %% ------------------------- the graph -------------------------
        \node[vtx]     (g7) at (1,3) {};
        \node[vtx]     (g5) at (.25,2) {};
        \node[removed] (g6) at (1.75,2) {};   % vertex 6 is the deleted one
        \node[vtx]     (g3) at (.25,1) {};
        \node[vtx]     (g4) at (1.75,1) {};
        \node[vtx]     (g1) at (.25,0) {};
        \node[vtx]     (g2) at (1.75,0) {};

        \draw (g7) -- (g5) -- (g6) -- (g7);
        \draw (g5) -- (g3) -- (g4) -- (g6);
        \draw (g5) -- (g4);
        \draw (g3) -- (g6);
        \draw (g3) -- (g1) -- (g2) -- (g4);

        \node[glbl, above] at (g7) {$7$};
        \node[glbl, left]  at (g5) {$5$};
        \node[glbl, left]  at (g3) {$3$};
        \node[glbl, right] at (g4) {$4$};
        \node[glbl, left]  at (g1) {$1$};
        \node[glbl, right] at (g2) {$2$};

        %% ------------------------- first tree ------------------------
        \begin{scope}[shift={(3,2)}]
          \node[vtx] (a1) at (0  ,0) {};          % 1
          \node[vtx] (a2) at (1.2,0) {};          % 2,3
          \node[vtx] (a3) at (2.4,0) {};          % 4,5
          \node[vtx] (a4) at (3.6,0) {};          % 7

          \draw (a1) -- (a2) -- (a3) -- (a4);

          \node[tlbl, below=4pt] at (a1) {$1$};
          \node[tlbl, below=4pt] at (a2) {$23$};
          \node[tlbl, below=4pt] at (a3) {$45$};
          \node[tlbl, below=4pt] at (a4) {$7$};
        \end{scope}

        %% ------------------------ second tree ------------------------
        \begin{scope}[shift={(3,.5)}]
          \node[vtx] (b1) at (0  ,0) {};          % 2
          \node[vtx] (b2) at (1.2,0) {};          % 1,4
          \node[vtx] (b3) at (2.4,0) {};          % 3,5
          \node[vtx] (b4) at (3.6,0) {};          % 7

          \draw (b1) -- (b2) -- (b3) -- (b4);

          \node[tlbl, below=4pt] at (b1) {$2$};
          \node[tlbl, below=4pt] at (b2) {$14$};
          \node[tlbl, below=4pt] at (b3) {$35$};
          \node[tlbl, below=4pt] at (b4) {$7$};
        \end{scope}

      \end{tikzpicture}
    };
    \node (botRight) at (4,-4)
    {
      \begin{tikzpicture}[thick, line join=round, scale=0.8]

        %% ------------------------- the graph -------------------------
        \node[removed] (g7) at (1,3) {};      % vertex 7 is the deleted one
        \node[vtx]     (g5) at (.25,2) {};
        \node[vtx]     (g6) at (1.75,2) {};
        \node[vtx]     (g3) at (.25,1) {};
        \node[vtx]     (g4) at (1.75,1) {};
        \node[vtx]     (g1) at (.25,0) {};
        \node[vtx]     (g2) at (1.75,0) {};

        \draw (g7) -- (g5) -- (g6) -- (g7);
        \draw (g5) -- (g3) -- (g4) -- (g6);
        \draw (g5) -- (g4);
        \draw (g3) -- (g6);
        \draw (g3) -- (g1) -- (g2) -- (g4);

        \node[glbl, left]  at (g5) {$5$};
        \node[glbl, right] at (g6) {$6$};
        \node[glbl, left]  at (g3) {$3$};
        \node[glbl, right] at (g4) {$4$};
        \node[glbl, left]  at (g1) {$1$};
        \node[glbl, right] at (g2) {$2$};

        %% ------------------------- first tree ------------------------
        \begin{scope}[shift={(3,2)}]
          \node[vtx]      (a1) at (0  ,0) {};     % 1
          \node[vtx]      (a2) at (1.5,0) {};     % 2,3
          \node[junction] (a3) at (2.25,0) {};     % branch point (no label)
          \node[vtx]      (a4) at (3.0,0) {};     % 4
          \node[vtx]      (a5) at (1.75,.75) {};     % 5
          \node[vtx]      (a6) at (2.75,.75) {};     % 6

          \draw (a1) -- (a2) -- (a3) -- (a4);
          \draw (a5) -- (a3) -- (a6);

          \node[tlbl, below=4pt] at (a1) {$1$};
          \node[tlbl, below=4pt] at (a2) {$23$};
          \node[tlbl, below=4pt] at (a4) {$4$};
          \node[tlbl, left=4pt] at (a5) {$5$};
          \node[tlbl, right=4pt] at (a6) {$6$};
        \end{scope}

        %% ------------------------ second tree ------------------------
        \begin{scope}[shift={(3,0.25)}]
          \node[vtx] (b1) at (0  ,0) {};          % 2
          \node[vtx] (b2) at (1.5,0) {};          % 1,4
          \node[junction] (b3) at (2.25,0) {};     % branch point (no label)
          \node[vtx] (b4) at (3,0)   {};          % 3
          \node[vtx] (b5) at (1.75,.75) {};          % 5
          \node[vtx] (b6) at (2.75,.75) {};          % 6

          \draw (b1) -- (b2) -- (b3) -- (b4);
          \draw (b5) -- (b3) -- (b6);

          \node[tlbl, below=4pt] at (b1) {$2$};
          \node[tlbl, below=4pt] at (b2) {$14$};
          \node[tlbl, below=4pt] at (b4) {$3$};
          \node[tlbl, left=4pt] at (b5) {$5$};
          \node[tlbl, right=4pt] at (b6) {$6$};
        \end{scope}
      \end{tikzpicture}
    };
    \node[left,xshift=-27mm] at (center) {$G$};
    \draw[->] (center.north west) -- ++(-1,1) node[midway,left,xshift=-3mm] {removing vertex 2};
    \draw[->] (center.north east) -- ++(1,1) node[midway,right,xshift=3mm] {removing vertex 4};
    \draw[->] (center.south west) -- ++(-1,-1) node[midway,left,xshift=-3mm] {removing vertex 6};
    \draw[->] (center.south east) -- ++(1,-1) node[midway,right,xshift=3mm] {removing vertex 7};
  \end{tikzpicture}\vspace{-2mm}
  \caption{Counterexample to Conjecture~\ref{conj: PS 4 pt cond generalisation} in Example~\ref{example: PS counterexample}.  All $6$-point metric subspaces of $G$ have rank two.}
  \label{fig:PScounterexample}
\end{figure}
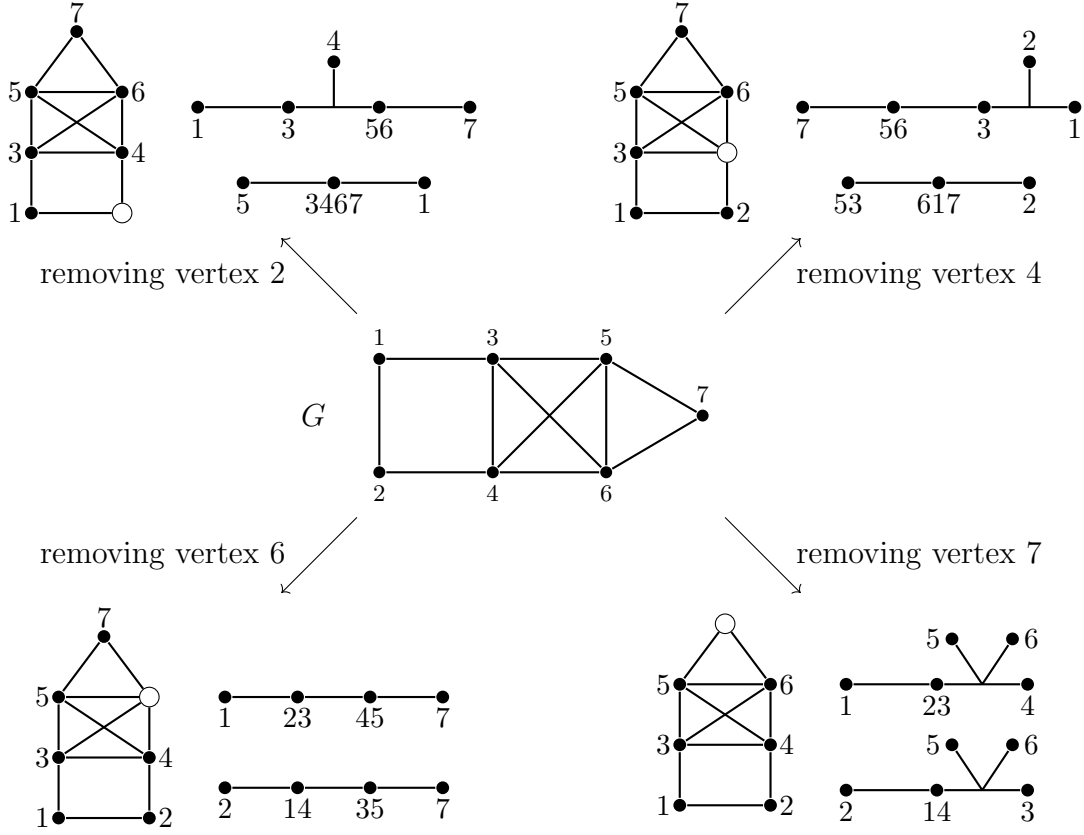

\begin{example}\label{example: PS counterexample}
  Consider the $7$-vertex graph $G$ with edge set
  \begin{equation*}
    E(G) = \{12,13,24,34,35,36,45,46,56,57,67 \}.
  \end{equation*}
  Our computation shows that $\prank(G) \ge 3$. Moreover, for each point $i \in [7]$, we can show that the metric subspace of $G$ obtained by removing $i$ has rank two as shown in Figure~\ref{fig:PScounterexample}.  Here, we make use of the fact \cref{alg:greedyAlgorithm} can be easily generalized to general finite metric spaces, and that it is exact if it returns $2$.
\end{example}

\section{Open questions}

We conclude the paper with a few open questions.  We will focus on questions that are of particular interest in the context of non-archimedean optimization \cite{LFMR2026}, in which a product of metric trees $\Gamma_1\times\dots\times\Gamma_n$ is a natural embedding space for data.

The first questions relate to the \emph{phylogenetic density} of a graph, which we define to be:
\begin{equation*}
  \phylogeneticdensity(G)\coloneqq \frac{\phylogeneticrank(G)}{|V(G)|}.
\end{equation*}

By \cref{prop: phylogenetic rank less than n-1}, we know that $0\leq \phylogeneticdensity(G)<1$.  We conjecture that the maximal possible phylogenetic density approaches $1$ as the number of vertices increases:

\begin{conjecture}\label{conj:phylogeneticDensity}
  We have
  \begin{math}
    \lim_{n\rightarrow\infty} \max_{|V(G)|=n} \phylogeneticdensity(G)=1.
  \end{math}
\end{conjecture}

In the light of the result by \cite{wolfe1967imbedding}, \cref{conj:phylogeneticDensity} implies that, asymptotically and in the worst case, there is no difference between the required dimension of embeddings into $\RR^n$ with the supremum norm, and the required dimension of embeddings into products of metric trees $\Gamma_1\times\dots\times\Gamma_n$ as studied in this paper.

In our experiments on small graphs in \cref{sec:computations} however, we observed that the maximal phylogenetic rank is rarely attained, see \cref{tab: small graph ranks connected}.  This naturally leads to the following question:

\begin{question}\label{ques:phylogeneticDistribution}
  What is the distribution of phylogenetic densities?
\end{question}

\cref{ques:phylogeneticDistribution} is left deliberately vague.  What is the expected phylogenetic density (for any model of random graph)?  Fixing a number of vertices $n$, is the sequence of numbers of graphs with the same phylogenetic rank (i.e., rows of \cref{tab: small graph ranks connected}) unimodular?  Is it log-concave?  Insights into the distribution would directly translate into important insights on required embedding dimensions in non-archimedean optimisation.

Finally, recall that \cref{thm: phylogenetic rank one} characterises graphs of phylogenetic rank one.  A natural generalisation is the following question:

\begin{question}\label{ques:phygeneticCharactisation}
  Characterise graphs with phylogenetic rank at most $k$ for $k\geq 2$.
\end{question}

For example, \cref{thm: phylogenetic rank one} (2) characterises graphs of phylogenetic rank one in terms of the forbidden induced subgraphs.
Unfortunately, \cref{example: rank 5 with induced 15 cycle} indicates that graphs of phylogenetic rank two or higher may not permit such a description.  Howver, since \cref{prop: isometric subgraph rank bound} shows that the phylogenetic rank is monotone for isometric subgraphs, maybe it is possible to characterise graphs of phylogenetic rank at most $k$ for $k\geq 2$ via their forbidden isometric subgraphs.  Insights into such a description would be useful in non-archimedean optimisation, as the local structure of data is generally much easier to access than the global structure of data.

\renewcommand{\bibfont}{\small}
\printbibliography

\end{document}